\documentclass[a4paper,12pt]{article}

\usepackage{amsmath}
\usepackage{amsthm}
\usepackage{dsfont}
\usepackage{suetterl}
\usepackage{adforn}
\usepackage{calligra}
\usepackage{pbsi}
\usepackage{aurical}
\usepackage[T1]{fontenc}
\usepackage{textgreek}

\usepackage[utf8]{inputenc}
\usepackage{makeidx}
\usepackage{enumerate}
\usepackage{auto-pst-pdf}
\usepackage{longtable}
\usepackage{pb-diagram}
\usepackage{multirow}

\usepackage{calc}
\usepackage{ifthen}
\usepackage{latexsym}
\usepackage{amssymb}
\usepackage{amsfonts}
\usepackage[mathscr]{eucal}
\usepackage{amscd}
\usepackage{array}
\usepackage{amsrefs}

 \DeclareMathOperator{\Rep}{Rep}
 \DeclareMathOperator{\soc}{soc}

\DeclareMathOperator{\R}{Re}\DeclareMathOperator{\I}{Im}
 \DeclareMathOperator{\ev}{ev}
\DeclareMathOperator{\Ker}{Ker}
\DeclareMathOperator{\Hom}{Hom} \DeclareMathOperator{\Mod}{Mod}
 \DeclareMathOperator{\Cor}{CoRep}
\DeclareMathOperator{\ch}{char} \DeclareMathOperator{\End}{End}
\DeclareMathOperator{\Ext}{Ext}

\theoremstyle{plain}
 \newtheorem{theorem}{Theorem}
 \newtheorem{lemma}[theorem]{Lemma}
 \newtheorem{proposition}[theorem]{Proposition}
 \newtheorem{corollary}[theorem]{Corollary}

\theoremstyle{remark}
 \newtheorem{remark}[theorem]{Remark}

\theoremstyle{definition}

 \newtheorem{definition}[theorem]{Definition}

\makeindex

\begin{document}

\title{Algebraically equipped posets}

\author{Raymundo Bautista \\ \textit{\small{raymundo@matmor.unam.mx}} \\ 
\text{Ivon Dorado} \\ \textit{\small{idorado@matmor.unam.mx}} \\ 
\\ \text{\small{Centro de Ciencias Matem\'aticas, Morelia}}\\ \text{\small{Universidad Nacional Nacional Aut\'onoma de Mexico}} }
\date{}
\maketitle

\centerline{\textit{To the memory of Alexander Zavadskij}}

\begin{abstract}
\noindent {\footnotesize We introduce partially ordered sets (posets) with an additional structure given by a collection of vector subspaces of an algebra $A$. We call them algebraically equipped posets. Some particular cases of these, are generalized equipped posets and $p$-equipped posets, for a prime number $p$.
We study their categories of representations and establish equivalences with some module categories, categories of morphisms and a subcategory of representations of a differential tensor algebra. Through this, we obtain matrix representations and its corresponding matrix classification problem.}
\end{abstract}

 {\small {\it Mathematics Subject Classification}: 06A11, 16G20, 16G60.\\

 {\small {\it Keywords}: Equipped poset, Algebraically equipped poset,  Representation, Corepresentation, Matrix problem, Right-peak algebra, Ditalgebra.


\section{Introduction}\label{intro}

Representation theory of partially ordered sets (posets), and posets with additional structure, has its roots in the representation theory of finite-dimensional algebras \cites{ARS-95,Gabriel-72,Gab-Roi-92,Ringel-84}. It is strongly related to the theory of rings and modules \cite{Zav-Kir-76,Zav-Kir-77,KS}, the theory of abelian groups \cite{Arnold2000,Arnold,Butler}, the theory of modular lattices, and the theory of bilinear and quadratic forms \cites{Ringel-84,Dr3,Zav-91-555,Naz-Roi-02,Naz-Roi-Sm-05}.

A significant example of posets with additional structure are equipped posets \cite{corep07,Zab-Zav,Zav-03-TEP,Zav-05-EPFG,R-11}, they are posets with an order relation of two kinds. The ideas of those researches were extended in two senses: In \cite{ZavGEP-11}, the terminology was modified and generalized equipped posets were introduced and their representations were studied over Galois field extensions, embracing completely the case studied in almost all the previous papers on the matter. On the other hand, in \cite{Dorado10}, there were introduced posets equipped with three kinds of binary relations (three-equipped posets) and its representations and corepresentations were defined over an inseparable cubic field extension. These works use, mainly, matrix methods.

Equipped posets (in all senses mentioned above) of finite-representation type were described in fact in \cite{KS}, in the course of a research about schurian representation-finite vector space categories.

In order to study relations of the representations of equipped posets with the representation theory of algebras, we introduce posets with an additional structure given by a collection of  vector subspaces of an algebra $A$. We call them algebraically equipped posets. Generalized equipped posets and $p$-equipped posets (also introduced in this paper), for a prime number $p$, are particular cases of these.

To every algebraically equipped poset, it is assigned an algebra $\Lambda$, in such a way that its representations are the objects of a subcategory of $\Lambda$-modules. If the equipment of the poset (i. e. the collection of vector subspaces) satisfies some conditions, $\Lambda$ is a right-peak algebra, in the sense of \cite{Simson} and \cite{KS}.

We can extend an algebraically equipped poset, as in \cite{BM}, to obtain an algebra $\Lambda$, which is right-peak, left-peak and 1-Gorenstein. Moreover, the category of representations of the original poset is equivalent to the category of right $\Lambda$-modules with projective socle, which do not have the projective-injective $\Lambda$-module as a direct summand. 

This last category is, in turn, equivalent to the category $\mathcal{V}$, of right $\Lambda$-modules with injective top, which do not have the projective-injective $\Lambda$-module as a direct summand. Every object $M\in \mathcal{V}$ has a minimal projective presentation $P_1 \rightarrow P_2$, where $P_2$ is projective-injective and $P_1$ belongs to $\mathcal{V}$.

The category $\mathcal{M}$ formed by the morphisms of the form $P_1 \rightarrow P_2$, has an exact structure $\mathcal{E}$ with enough projective and injective objects, and it holds $\Ext_{\mathcal{E}}^2(M,N)=0$, for all $M,N \in \mathcal{M}$. Then $\mathcal{M}$ is hereditary, in contrast to $\mathcal{U}$ and $\mathcal{V}$ which, in general, are not hereditary. We prove that $\mathcal{U}$ and $\mathcal{V}$ are equivalent to the the category $\mathcal{M}$ modulo the morphisms which have a factorization through objects of the form $P\rightarrow 0$.

We will show that $\mathcal{M}$ is equivalent to  the category of representations of a differential tensor algebra. Through this, we obtain matrix representations and its corresponding matrix classification problem. It should be mentioned that matrix problems have been an important classification tool in most of the researches about representation theory of partially ordered sets. 

To be more clear with the equivalences and constructions, this paper includes an appendix, following \cite{BSZ}, about differential tensor algebras (ditalgebras). 


\section{Equipped posets and its representations}


\subsection{$p$-equipped posets}\label{defPequipados}

Let $p$ be a prime number, a finite poset $(\mathscr{P},\leq )$ is called \textit{p-equipped} if to every pair $(x,y)\in\  \leq$, it is assigned only one value $\ell\in \{1,2,\dots,p\}$, with the notation $x\leq ^\ell y$, and the following condition holds:

\begin{equation}\label{definition}
\text{If }x\leq ^\ell y\leq ^mz\ \text{ and  }x\leq ^nz\text{, then }n\geq \min \{ \ell+m-1, p \}.
\end{equation}

A relation $x\leq y$ is called \textit{weak} or \textit{strong}, if $x\leq ^1 y$ or $x\leq ^p y$, respectively. It follows that the composition of a strong relation with any other relation is strong. 

For each point $x\in \mathscr{P}$, we have that $x\leq^\ell x\leq^\ell x$ implies $\ell \geq \min \{2\ell -1,p\}$. If $\ell \geq p$ then $\ell=p$. And if $\ell \geq 2 \ell -1$, then $\ell =1$. This is $x\leq ^1 x$, or $x\leq ^p x$. In the first case, we call $x$ a \textit{weak} point, and in the second one \textit{strong}. A relation between an arbitrary point and a strong is always strong.

We write $x<^\ell y$, if $x\leq^\ell y$ and $x\neq y$. In
particular, if $\ell\in \{2,3,\dots,p-1\}$ we have $x<^\ell y$.

If a $p$-equipped poset $\mathscr{P}$ is \textit{trivially equipped}, i. e. it contains only strong points, then it is an \textit{ordinary} poset.


\subsubsection{Representations of $p$-equipped posets}\label{reps}

Consider a $p$-equipped poset $\mathscr{P}$, and a pair of fields $(\textsf{F},\textsf{G})$, where \textsf{G} is a normal extension of degree $p$ over \textsf{F}. We know that $\textsf{G}=\textsf{F}(\xi)$, for some primitive element $\xi$, such that $\xi^p=q\in \textsf{F}$.

Depending on the characteristic of \textsf{F}, the extension $\textsf{G}/\textsf{F}$ may be cyclic or purely inseparable. 

When $\ch \textsf{F}=p$, the extension $\textsf{G}/\textsf{F}$ is purely inseparable. In this case, there is a natural derivation $\delta$ of \textsf{G}, which is an endomorphism of \textsf{G}, satisfying 
\[\delta(a)=0, \text{\ for all\ } a\in \textsf{F},\ \ \delta(\xi^i)=i\xi^{i-1}, \ \  \text{\ for each \ } i<p.\] 

Indeed, the minimal polynomial of $\xi$ over \textsf{F} is $t^p-q$, and the usual derivative of the polynomial ring $\textsf{F}[t]$, is null on the ideal $\langle t^p-q \rangle$. Then, there is an induced derivation of $\textsf{F}[t]/\langle t^p-q \rangle$, from which we obtain $\delta$, by using the isomorphism $\textsf{F}[t]/\langle t^p-q \rangle \cong \textsf{G}$. Clearly $\delta^p=0$, and the Leibniz rule holds for every $a,b\in \textsf{G}$,
$$\delta(ab)=a\delta(b)+\delta(a)b.$$

If $\ch \textsf{F}\neq p$, let $\sigma : \textsf{G}\rightarrow \textsf{G}$ be a generator of the Galois group. 

\begin{definition}\label{}
Let $V$ be a \textsf{G}-vector space, with $\ch \textsf{F}= p$. An \textsf{F}-linear application $d:V\rightarrow V$ is called $\delta$\textit{-derivation} if satisfies:
\begin{enumerate}[D.1.]
\item $d^p=0$;
\item $d(vg)=d(v)g+v\delta(g)$, for each $g\in \textsf{G}, v\in V$.
\end{enumerate}
\end{definition}

\begin{definition}\label{sigmaAccion}
Let $V$ be a \textsf{G}-vector space, and $s:V\rightarrow V$ an \textsf{F}-linear application, with $\ch \textsf{F}\neq p$. We say that $s$ is a $\sigma$-\textit{action}, if the following conditions hold
\begin{enumerate}[S.1.]
\item $s^p=\text{id}_V$;
\item $s(va)=s(v)\sigma(a)$, for every $v\in V$, $a \in \textsf{G}$;
\end{enumerate}
\end{definition}

A \textit{representation} $\widetilde{V}$ of a $p$-equipped poset $\mathscr{P}$, over the pair $(\textsf{F},\textsf{G})$ is a collection of the form
\[\widetilde{V}=(V,r;V_x : x \in \mathscr{P})\]
where $V$ is a \textsf{G}-vector space with a $\sigma$-action ($\delta$-derivation) $r$, if $\ch \textsf{F}\neq p$ ($\ch \textsf{F}= p$), every $V_x$ is a \textsf{G}-subspace of $V$, and for all $x, y\in \mathscr{P}$ it holds
\[x\leq^\ell y \Rightarrow \sum_{i=0}^{\ell-1} r^i(V_x)\subseteq V_y.\]

Let $\widetilde{V}=(V,r;V_x : x \in \mathscr{P})$ and $\widetilde{V'}=(V',r';V_x' : x \in \mathscr{P})$ be two representations, a \textit{morphism} $\psi:\widetilde{V}\rightarrow \widetilde{V'}$ is a \textsf{G}-linear transformation $\psi:V\rightarrow V'$ that commutes with the respective $\sigma$-actions or $\delta$-derivations, and is such that $\psi(V_x)\subseteq V_x'$, for every $x \in \mathscr{P}$.

We denote by $\Rep \mathscr{P}$ the category of representations of a poset $\mathscr{P}$, this is an additive category, if $\widetilde{V}$ and $\widetilde{V'}$ are two representations, its direct sum is
\[\widetilde{V} \oplus \widetilde{V'} = (V \oplus V',r \oplus r'; V_x \oplus V_x' : x \in \mathscr{P}).\]


\subsubsection{Corepresentations}\label{coreps}

The $p$-equipped posets lead naturally to the following definition:

A \textit{corepresentation} $\widetilde{V}$, of a $p$-equipped poset $\mathscr{P}$, over the pair $(\textsf{F},\textsf{G})$, is a collection of the form
\[\widetilde{V}=(V;V_x : x \in \mathscr{P})\]
where $V$ is a \textsf{G}-vector space, every $V_x$ is an \textsf{F}-subspace of $V$, and for all $x, y\in \mathscr{P}$ it holds
\[x\leq^\ell y \Rightarrow \sum_{i=0}^{\ell-1}V_x \xi^i \subseteq V_y.\]

A \textit{morphism} $\psi:\widetilde{V}\rightarrow \widetilde{V'}$ between two corepresentations $\widetilde{V}$ and $\widetilde{V'}$, is a \textsf{G}-linear transformation $\psi:V\rightarrow V'$, such that $\psi(V_x)\subseteq V_x'$, for every $x \in \mathscr{P}$.

The corepresentations and its morphisms form an additive category called $\Cor \mathscr{P}$. In fact, if $\widetilde{V}$ and $\widetilde{V'}$ are two corepresentations, its direct sum is
\[\widetilde{V} \oplus \widetilde{V'} = (V \oplus V', V_x \oplus V_x' : x \in \mathscr{P}).\]


\subsection{Generalized equipped posets}

These posets were introduced \cite{ZavGEP-11}, with the following definition:

A \textit{generalized equipped} poset is a triple $(\mathscr{P},\Delta,S)$, where $\mathscr{P}$ is a partially ordered set, $\Delta$ is a finite multiplicative group, 
$$S=\{\Delta_{xy} | x\leq y \text{ in } \mathscr{P} \}$$ 
is any collection of non-empty subsets of the group $\Delta$ satisfying the condition
\begin{equation}\label{e-definition of an equipped poset}
  \Delta_{xy} \Delta_{yz} \subseteq  \Delta_{xz}, \text{ for any chain } x\leq y \leq z \text{ in } \mathscr{P}.
\end{equation}

Their representations were studied over a Galois group $\Gamma$, of a field extension $\textsf{L}/\textsf{K}$. 

We will give a definition equivalent to that of \cite{ZavGEP-11}, by using the following idea of vector space with $\Gamma$-action.

An $\textsf{L}$-vector space $V$ is called \textit{space with $\Gamma$-action} if it has a $\textsf{K}$-linear transformation $s_{\sigma}:V \rightarrow V$, for each  $\sigma\in \Gamma$, such that
\begin{enumerate}[T.1.]
\item $s_{id_{\textsf{L}}}=\text{id}_V$;
\item $s_\tau s_\sigma=s_{\tau \sigma}$;
\item $s_\tau (va)= s_\tau (v) \tau (a)$, for every $a\in \textsf{L}$.
\end{enumerate}

A \textit{representation} of a generalized equipped poset  $(\mathscr{P},\Gamma,S)$, where $S=\{\Gamma_{xy} | x\leq y \text{ in } \mathscr{P} \}$, is any collection of the form
\[\widetilde{V}=(V;V_x : x \in \mathscr{P}),\]
where $V$ is an $\textsf{L}$-space with $\Gamma$-action, every $V_x$ is an $\textsf{L}$-subspace of $V$, and for all $x,y\in \mathscr{P}$, it holds
\[x\leq y \text{ implies } \sum_{\sigma\in \Gamma_{xy}} s_\sigma(V_x) \subseteq V_y.\]

Let $\widetilde{V}=(V;V_x : x \in \mathscr{P})$ and $\widetilde{V'}=(V';V_x' : x \in \mathscr{P})$ be representations. A \textit{morphism} $\psi:\widetilde{V}\rightarrow \widetilde{V'}$ is an \textsf{L}-linear transformation $\psi:V\rightarrow V'$ that commutes with the transformations $s_{\sigma}:V \rightarrow V$ and $s'_{\sigma}:V' \rightarrow V'$, for every $\sigma\in \Gamma$, and is such that $\psi(V_x)\subseteq V_x'$, for all $x \in \mathscr{P}$.

The category formed by the representations of a generalized equipped poset and the morphisms between them, is denoted $\Rep (\mathscr{P},\Gamma,S)$.

\begin{remark}\label{pequipadosSonGen}
If $\mathscr{P}$ is a $p$-equipped poset and we consider its representations over a cyclic extension $\textsf{G}/\textsf{F}$, we have a particular case of generalized equipped poset $(\mathscr{P},\Gamma,S)$ in which $\Gamma$ is the cyclic Galois group generated by $\sigma$ and for every $x\leq^{\ell}y$ in $\mathscr{P}$,
$$\Gamma_{xy} = \{\sigma^i | 0\leq i \leq \ell-1\}.$$

Notice that a \textsf{G}-space $V$, is a space with $\Gamma$-action if and only if is a space with $\sigma$-action, in the sense of the definition \ref{sigmaAccion}. So we conclude that the categories $\Rep \mathscr{P}$ and $\Rep (\mathscr{P},\Gamma,S)$ are the same.
\end{remark}


\section{Algebraically equipped posets}

\begin{definition}\label{sisMulti}
Let $\textsf{K}$ be a field, $A$ be a $\textsf{K}$-algebra and $(\mathscr{P},\leq )$ be a partially ordered set. A \textit{$\mathscr{P}$-multiplicative system} $\mathcal{R}$, is a collection 
$$\mathcal{R}=\{\mathcal{R}_{i,j}\subseteq A \ | \ i,j\in \mathscr{P} \ \text{ such that } i\leq j,\} , $$
of $\textsf{K}$-vector subspaces $\mathcal{R}_{i,j}$ of $A$, that satisfy the following conditions:
\begin{enumerate}[M.1]
\item For every $i,j,l\in \mathscr{P}$ with $i\leq j\leq l$, we have $\mathcal{R}_{i,j}\mathcal{R}_{j,l}\subseteq \mathcal{R}_{i,l}$.
\item For every $i\in \mathscr{P}$, the space $\mathcal{R}_i=\mathcal{R}_{i,i}$ is a $\textsf{K}$-algebra with unit $1_i$, such that for all $r\in \mathcal{R}_{i,j}$ it holds $1_ir=r=r1_j$, for each $j\in \mathscr{P}$, $i\leq j$.
\end{enumerate}
\end{definition}

The triple $(\mathscr{P},A,\mathcal{R})$ is and \textit{algebraically equipped poset}. We say that $\mathcal{R}$ is an \textit{equipment} of the poset $\mathscr{P}$, over $A$.

\begin{definition}\label{defRep}
A \textit{representation} $\widetilde{L}$, of an algebraically equipped poset $(\mathscr{P},A,\mathcal{R})$, is a collection $\widetilde{L}=(L, L_i:i\in \mathscr{P})$, where $L$ is a right $A$-module, and $L_i$ is a vector subspace of $L$, for every $i\in \mathscr{P}$, such that the following conditions hold;
\begin{enumerate}[L.1]
\item $L_i\mathcal{R}_{i,j}\subseteq L_j$, for each pair $i,j\in\mathscr{P}$, with $i\leq j$.
\item $a1_{i}=a$, for all $a\in L_{i}$.
\end{enumerate}
\end{definition}

A \textit{representation morphism} $h:\widetilde{L}\rightarrow \widetilde{M}$, is an right $A$-module morphism $h:L\rightarrow M$, such that $h(L_i)\subseteq M_i$, for every $i\in \mathscr{P}$.

The representations and the morphisms between them form a category called
$$\Rep (\mathscr{P},A,\mathcal{R}).$$


\subsection{Algebraically equipped posets associated to Generalized equipped posets}\label{genSonALg}

Representations of generalized equipped posets were studied over a Galois field extension  $\textsf{L}/\textsf{K}$, so we will construct an equipment $\mathcal{T}$ for a poset $\mathscr{P}$, over the algebra
$$A=(\End_{\textsf{K}}\textsf{L})^{op}.$$

Every $a\in \textsf{L}$, determines a \textsf{K}-endomorphism $\mu_a:\textsf{L}\rightarrow \textsf{L}$, given by multiplication, that is, 
\begin{equation}\label{defmu}
\mu_a(b)=ab\text{, for all }b\in \textsf{L}.
\end{equation}

Notice that, for every $\sigma$ in the Galois group $\Gamma$,
$$(\sigma \mu_a)(b)=\sigma(ab)= \sigma(a)\sigma(b) = \mu_{\sigma(a)}\sigma(b).$$

The \textsf{K}-algebra $\End_{\textsf{K}}\textsf{L}$ has a left \textsf{L}-vector space structure given by the identification of every $l\in \textsf{L}$ with $\mu_l$, then, for all $h\in \End_{\textsf{K}}\textsf{L}$ and $l\in \textsf{L}$, we identify $lh=\mu_lh$. This product gives to $A$ a right \textsf{L}-vector space structure.

Due to the linear independence of the Galois automorphisms over \textsf{L} (see \cite{Jacobson-64}), 
\[\dim_{\textsf{L}} \left\{ \sum _{\sigma \in \Gamma }\mu_{a_\sigma}\sigma |  a_\sigma \in \textsf{L} \right\} = [\textsf{L}:\textsf{K}].\]

And $\dim_{\textsf{L}} \End_{\textsf{K}}\textsf{L}= [\textsf{L}:\textsf{K}]$, then any $h\in \End_{\textsf{K}}\textsf{L}$ can be written in a unique way as 
\begin{equation}\label{baseDeA}
h=\sum _{\sigma \in \Gamma }\mu_{a_\sigma}\sigma,
\end{equation}
for some $a_\sigma \in \textsf{L}$.

For $h,h' \in \End_{\textsf{K}}\textsf{L}$, the multiplication in $A$ is denoted by $h*h'=h'h$.

Let $(\mathscr{P},\Gamma,S)$ be a generalized equipped poset with, $S=\{\Gamma_{xy} | x\leq y \text{ en } \mathscr{P} \}$. Consider the $\textsf{L}$-vector subspaces $\mathcal{T}_{x,y}\subseteq A$ spanned by $\Gamma _{x,y}$, for  $x\leq y$,
$$\mathcal{T}_{x,y} = \textsf{K} \left\langle \sum_{\sigma \in \Gamma_{xy}}  \sigma * \mu_{a_\sigma}\  |\ \  a_\sigma \in \textsf{L} \right\rangle = \textsf{K} \left\langle \sum_{\sigma \in \Gamma_{xy}} \mu_{a_\sigma} \sigma \  |\ \  a_\sigma \in \textsf{L} \right\rangle = \textsf{L} \left\langle \Gamma_{xy} \right\rangle.$$

If $x\leq y \leq z$ in $\mathscr{P}$, $\sigma \in \Gamma_{xy}$, $\tau \in \Gamma_{yz}$ and $a,b \in \textsf{L}$, we have
$$\mu_a \sigma \mu_b \tau = \mu_a \mu_{\sigma(b)} \sigma \tau = \mu_{a \sigma(b)}\sigma \tau.$$

By Condition (\ref{e-definition of an equipped poset}) in the definition of generalized equipped poset,
$$\mathcal{T}_{x,y}\mathcal{T}_{y,z}\subset \mathcal{T}_{x,z}.$$ 

For all $x\in \mathscr{P}$, the set $\Gamma  _{x}=\Gamma  _{x,x}$ is a subgroup of $\Gamma $, so we call $\textsf{K}(x)$ its fixed field.

If $\alpha \in \mathcal{T}_{x}=\mathcal{T}_{x,x}$, it can be written as $\alpha =\sum _{\sigma \in \Gamma  _{x}}\sigma * \mu_{a_\sigma}$, for some $a_\sigma \in \textsf{L}$. Therefore, every $\mathcal{T}_{x}$ is a $\textsf{K}$-subalgebra of $A$, and 
\begin{equation}\label{TxSubalgebra}
\mathcal{T}_{x}\cong (\End_{\textsf{K}(x)}\textsf{L})^{op}.
\end{equation}

Then, the collection $\mathcal{T}=\{\mathcal{T}_{x,y}\subseteq A \ | \ x,y\in \mathscr{P} \ \text{ such that } x\leq y,\}$ is a $\mathscr{P}$-multiplicative system.

\begin{theorem}\label{equivalenciaRepGen}
Using the notation of this section, the categories $\Rep (\mathscr{P},\Gamma,S)$ and $\Rep (\mathscr{P},A,\mathcal{T})$ are equivalent.
\end{theorem}

\begin{proof}
Consider $\widetilde{V}=(V;V_x : x \in \mathscr{P})\in \Rep (\mathscr{P},\Gamma,S)$, let us define the following action of $A$ over $V$;
\begin{equation}\label{accionSobreV}
va=\sum _{\sigma \in \Gamma}s_\sigma(v)a_\sigma;
\end{equation}
for every $v\in V$ y $a\in A$, of the form $a=\sum _{\sigma \in \Gamma }\sigma*\mu_{a_\sigma}$, for some $a_\sigma \in \textsf{L}$.

Notice that when $a$ is the unit of $A$, we have $va=v$ for all $v\in V$. And if $\sigma, \tau \in \Gamma$ and $a, b \in \textsf{L}$,
\begin{align*}
(v \sigma*\mu_a) \tau*\mu_b &= (s_\sigma(v)a)\tau*\mu_b = s_\tau(s_\sigma(v)a) b\\
 &= s_\tau(s_\sigma(v))\tau (a) b\\
 &= s_{\tau \sigma}(v)\tau (a) b\\
 &= v(\tau \sigma * \mu_{\tau (a) b}) = v (\tau \sigma * \mu_{b\tau (a)})\\
 &= v (\mu_{b\tau (a)} \tau \sigma) = v(\mu_b \mu_{\tau (a)} \tau \sigma)\\
 &= v(\mu_b (\mu_{\tau (a)} \tau) \sigma) = v (\mu_b \tau \mu_a \sigma) = v((\mu_a \sigma)*(\mu_b \tau))\\
 &= v((\sigma*\mu_a)*(\tau*\mu_b)).
\end{align*}

That is, for every $a,b \in A$
$$(va)b=v(ab);$$
so the action (\ref{accionSobreV}), gives a right $A$-module structure to $V$, and is such that,
$$V_x \mathcal{T}_{x,y} \subseteq V_y,$$
is equivalent to
$$\sum_{\sigma\in \Gamma_{xy}} s_\sigma(V_x) \subseteq V_y.$$

Therefore, $\widetilde{V}$ is an object of $\Rep (\mathscr{P},A,\mathcal{T})$.

If $\widetilde{V}=(V, V_x:x\in \mathscr{P}) \in \Rep (\mathscr{P},A,\mathcal{T})$, as $V$ is a right $A$-module, then it is a right $\textsf{L}$-module by using the identification 
$$\textsf{L}= \left\{ \mu_a |  a \in \textsf{L} \right\}.$$

So $V$ is an \textsf{L}-vector space. To give it a $\Gamma$-action, we define for each $\sigma\in \Gamma$
$$s_\sigma:V\rightarrow V: v \mapsto v\sigma.$$

Clearly, with this definition
$$V_x \mathcal{T}_{x,y} \subseteq V_y,$$
if and only if
$$\sum_{\sigma\in \Gamma_{xy}} s_\sigma(V_x) \subseteq V_y.$$

By the $A$-module structure of $V$, the condition T.1 holds, and for $v\in V, \sigma, \tau \in \Gamma$,
$$s_\tau(s_\sigma(v)) = (v\sigma)\tau = v (\sigma*\tau) = v(\tau \sigma) = s_{\tau \sigma}(v);$$
then, we have T.2.

The condition T.3 is satisfied because for $a\in \textsf{L}$,
$$s_\sigma(va) = va\sigma = v(\mu_a *\sigma) = v (\sigma \mu_a) = v (\mu_{\sigma(a)} \sigma) = v (\sigma* \mu_{\sigma(a)}) = s_\sigma(v) \sigma(a).$$

Hence, $V$ is an \textsf{L}-space with $\Gamma$-action, and $\widetilde{V}\in \Rep (\mathscr{P},\Gamma,S)$.

Let $\widetilde{V}=(V,V_x : x \in \mathscr{P})$ and $\widetilde{V'}=(V',V_x' : x \in \mathscr{P})$ be two representations of $(\mathscr{P},A,\mathcal{T})$. An application $h:V\rightarrow V'$ is a morphism in $\Rep (\mathscr{P},A,\mathcal{T})$, if and only if, $h$ is an $A$-module morphism and $h(V_x)\subseteq V_x'$, for all $x \in \mathscr{P}$.
That is, for every $v\in V$, $a\in A$, we have that $h(va) = h(v)a$, and by (\ref{baseDeA}), this is equivalent to have
\begin{align*}
h(v(\sigma* \mu_l)) &= h(v)(\sigma* \mu_l)\\
h(s_\sigma(v)l) &= s'_\sigma(h(v))l,
\end{align*}
for every $\sigma\in \Gamma, l\in \textsf{L}$.

And the last happens if and only if $h$ is \textsf{L}-linear, and commutes with every $s_\sigma$ and $s'_\sigma$.

We conclude that $h$ is a morphism in $\Rep (\mathscr{P},A,\mathcal{T})$, if and only if, it  is a morphism in $\Rep (\mathscr{P},\Gamma,S)$, and the theorem is proved.
\end{proof}


\subsection{Algebraically equipped posets associated to $p$-equipped posets}\label{pequipadosSonALg}


\subsubsection{Equipment over a field}\label{pequipadosSonALgCorep}

Let $(\textsf{F},\textsf{G})$ be a pair of fields as in the section \ref{defPequipados}.

\begin{definition}\label{sisQ}
Let $\mathscr{P}$ be a $p$-equipped poset. The field $\textsf{G}$, is an \textsf{F}-algebra, so we can define a $\mathscr{P}$-multiplicative system $\mathcal{Q}=\{\mathcal{Q}_{x,y}\subseteq \textsf{G} \ | \ x,y\in \mathscr{P} \ \text{ such that } x\leq y,\}$, by assigning to  each $x\leq^\ell y$ in $\mathscr{P}$, the \textsf{F}-space
\[\mathcal{Q}_{x,y}=\textsf{F} \langle 1, \xi, \xi^2, \ldots, \xi^{\ell-1} \rangle.\]
\end{definition}

For every $x, y, z \in\mathscr{P}$ such that $x\leq^\ell y\leq^m z$, we have 
$$\mathcal{Q}_{x,y}\mathcal{Q}_{y,z}=\textsf{F} \langle 1, \xi, \xi^2, \ldots, \xi^{\ell+m-2} \rangle.$$

If $\ell+m-1 \geq p$, then $\mathcal{Q}_{x,y}\mathcal{Q}_{y,z}=\textsf{F} \langle 1, \xi, \xi^2, \ldots, \xi^{p-1} \rangle = \textsf{G}$.

By the condition (\ref{definition}) in the definition of $p$-equipped poset, $x\leq^n z$, with $n\geq \min \{ \ell+m-1, p \}$, so 
$$\mathcal{Q}_{x,y}\mathcal{Q}_{y,z}\subseteq \textsf{F} \langle 1, \xi, \xi^2, \ldots, \xi^{n-1} \rangle = \mathcal{Q}_{x,z};$$ 
therefore $\mathcal{Q}$ satisfies M.1.

It is clear that $\mathcal{Q}_{x,y}=\textsf{F}$, when $x\leq^1y$, and $\mathcal{Q}_{x,y}=\textsf{G}$ if $x\leq^py$. In particular,
\[\mathcal{Q}_x= \begin{cases} 
\textsf{F}, & \text{ if } x \text{ is weak},\\
\textsf{G}, & \text{ si } x \text{ is strong},\\
\end{cases}\]
this is, $\mathcal{Q}$ satisfies the condition M.2, then it is a $\mathscr{P}$-multiplicative system.

We conclude that $(\mathscr{P},\textsf{G},\mathcal{Q})$ is an algebraically equipped poset.

For any collection $\widetilde{V}=(V;V_x : x \in \mathscr{P})$, where $V$ is a \textsf{G}-module, so it is a \textsf{G}-vector space, and each $V_x$ is an \textsf{F}-subespace of $V$, we have that, for all $x\leq^\ell y$ in $\mathscr{P}$,
the condition $V_x \mathcal{Q}_{x,y} \subseteq V_y$ is equivalent to $\sum_{i=0}^{\ell-1}V_x \xi^i \subseteq V_y$.

Hence, $\widetilde{V}\in \Cor \mathscr{P}$ if and only if $\widetilde{V}\in \Rep (\mathscr{P},\textsf{G},\mathcal{Q})$. 

Obviously a morphism in $\Cor \mathscr{P}$ is a morphism in $\Rep (\mathscr{P},\textsf{G},\mathcal{Q})$ and vice versa. Then, we have the following equivalence:
\[\Cor \mathscr{P} \cong \Rep (\mathscr{P},\textsf{G},\mathcal{Q}).\]


\subsubsection{Equipment over an endomorphism algebra}\label{pequipadosSonALgRep}

Let $\mathscr{P}$ be a $p$-equipped poset, and $\Rep \mathscr{P}$ its category of representations over the pair $(\textsf{F},\textsf{G})$. Consider the \textsf{F}-algebras
$$\End_{\textsf{F}}\textsf{G} \hspace{0.5cm}\text{ and }\hspace{0.5cm} A=(\End_{\textsf{F}}\textsf{G})^{op}.$$

As in Section \ref{genSonALg}, the multiplication in $A$ is denoted by $h*h'=h'h$, for all $h,h' \in \End_{\textsf{F}}\textsf{G}$.

From every $g\in \textsf{G}$, we obtain an \textsf{F}-endomorphism $\mu_g:\textsf{G}\rightarrow \textsf{G}$, given by, 
\begin{equation}\label{defmuPequipados}
\mu_g(a)=ga\text{, for all }a\in \textsf{G}.
\end{equation}

We denote $\vartheta= \begin{cases} 
\sigma, & \text{ if } \ch \textsf{F}\neq p,\\
\delta, & \text{ if } \ch \textsf{F}= p.
\end{cases}$

In the following result, we define an equipment for a  $p$-equipped poset $\mathscr{P}$, over $A$. If $\ch \textsf{F}\neq p$, by Remark \ref{pequipadosSonGen}, the poset $\mathscr{P}$ is a particular case of generalized equipped poset, then the $\mathscr{P}$-multiplicative system we will construct next, is the same as in the section \ref{genSonALg}.

\begin{proposition}\label{sisMultEnd}
Let $\mathscr{P}$ be a  $p$-equipped poset. The colection 
$$\mathcal{T}=\{\mathcal{T}_{x,y} = A_\ell \ | \ x\leq^\ell y\}_{x,y\in \mathscr{P}},$$  
where, to every $x\leq^\ell y$ in $\mathscr{P}$, it is assigned the following \textsf{F}-subspace of $A$,
\[A_\ell =\textsf{F} \left\langle \sum_{i=0}^{\ell-1}  \vartheta^i * \mu_{g_i}\  |\ \  g_i\in \textsf{G} \right\rangle = \textsf{F} \left\langle \sum_{i=0}^{\ell-1} \mu_{g_i} \vartheta^i \  |\ \  g_i\in \textsf{G} \right\rangle,\]
is a $\mathscr{P}$-multiplicative system that satisfies
\[\mathcal{T}_x \cong \textsf{G},  \text{ if } x \text{ is weak, and}\]
\[\mathcal{T}_x = A,  \text{ if } x \text{ is strong}.\]
\end{proposition}

\begin{proof}
Clearly, 
$$A_1 = \textsf{F} \left\langle \mu_{g} \  |\ \  g\in \textsf{G} \right\rangle \cong \textsf{G}.$$

For every weak point $x\in \mathscr{P}$, by definition $x\leq^1x$, then $\mathcal{T}_{x} \cong \textsf{G}.$

If $t\geq p$,
\[\textsf{F} \left\langle \sum_{i=0}^{t-1}  \vartheta^i * \mu_{g_i}\  |\ \  g_i\in \textsf{G} \right\rangle = \textsf{F} \left\langle \sum_{i=0}^{t-1} \mu_{g_i} \vartheta^i \  |\ \  g_i\in \textsf{G} \right\rangle = A_p,\]
because $\vartheta^p=\delta^p=0$, or $\vartheta^p=\sigma^p=id_{\textsf{G}}$.

When $\ch \textsf{F}= p$, we have that $\vartheta=\delta$ and for all $a, g \in \textsf{G}$, 
\[\delta \mu_g(a)=\delta(ga)=g\delta(a)+\delta(g)a=\mu_g\delta(a) + \mu_{\delta(g)}(a).\]

We suppose, as induction hypothesis, that   $k\in \mathbb{N}$,
\begin{equation}\label{derivarmultiplicarGeneral}
\delta^k \mu_g=\sum_{i=0}^k \binom{k}{i}\mu_{\delta^i(g)}\delta^{k-i};
\end{equation}
for $k\in \mathbb{N}$.

By composing with $\delta$,
\begin{align*}
\delta^{k+1} \mu_g &= \sum_{i=0}^k \binom{k}{i}\delta \mu_{\delta^i(g)}\delta^{k-i}\\
 &= \sum_{i=0}^k \binom{k}{i} (\mu_{\delta^i(g)} \delta + \mu_{\delta^{i+1}(g)}) \delta^{k-i}\\
 &= \mu_g \delta^{k+1} + \cdots + \binom{k}{i-1} \mu_{\delta^{i-1}(g)} \delta^{k-i+2}\\
 &\ \ \ + \binom{k}{i-1}  \mu_{\delta^i(g)} \delta^{k-i+1} + \binom{k}{i} \mu_{\delta^i(g)}\delta^{k-i+1}\\
 &\ \ \ + \binom{k}{i} \mu_{\delta^{i+1}(g)} \delta^{k-i} + \cdots + \mu_{\delta^{k+1}(g)}\\
 &= \sum_{i=0}^{k+1} \binom{k+1}{i} \mu_{\delta^i(g)}\delta^{k+1-i}.
\end{align*}

Then, we have (\ref{derivarmultiplicarGeneral}) for every $k$. 

If $\ch \textsf{F}\neq p$, as $\vartheta=\sigma$, then for all $a, g \in \textsf{G}$, 
$$(\sigma \mu_g)(a)=\sigma(ga)= \sigma(g)\sigma(a) = \mu_{\sigma(g)}\sigma(a).$$

For any natural number $k$, and any $g\in \textsf{G}$, let us give the induction hypothesis
\begin{equation}\label{sigmamultiplicar}
\sigma^{k} \mu_g= \mu_{\sigma^k(g)} \sigma^{k}.
\end{equation}

By composing with $\sigma$
\begin{align*}
\sigma^{k+1} \mu_g &= \sigma \mu_{\sigma^k(g)} \sigma^{k}\\
 &= \mu_{\sigma^{k+1}(g)} \sigma^{k+1}.
\end{align*}

Hence (\ref{sigmamultiplicar}) holds for every $k$. 

Consider $x\leq^\ell y\leq^m z$ in $\mathscr{P}$, for $a,b\in \textsf{G}$, when $\ch \textsf{F}= p$
\[\mu_a \delta^{\ell -1} \mu_b \delta^{m -1} =  \sum_{i=0}^{\ell -1} \binom{\ell -1}{i} \mu_{a\delta^i(b)}\delta^{\ell +m -2-i};\]
or, if $\ch \textsf{F}\neq p$
\[\mu_a \sigma^{\ell -1} \mu_b \sigma^{m -1} = \mu_{a\sigma^{\ell -1}(b)}\sigma^{\ell +m -2}.\]

Therefore $(\vartheta^{\ell -1} * \mu_a) (\vartheta^{m -1} * \mu_b) \in A_n$, where $n = \min \{ \ell+m-1, p \}$. That is,
$$A_\ell A_m \subseteq A_n.$$

Moreover, if $\ell+m-1 \geq p$, then $A_p \subseteq A_\ell A_m$, and if $n< p$, any element of the form $\vartheta^{n'} * \mu_b \in A_n$, can be written as $\vartheta^{\ell'} (\vartheta^{m'} * \mu_b)$, where $\ell' < \ell$ and $m' < n$, so $\vartheta^{n'} * \mu_b \in A_\ell A_m$. We have that
\begin{equation}\label{AlAm=An}
A_\ell A_m = A_n \text{, where } n = \min \{ \ell+m-1, p \}.
\end{equation}

By the definition of a $p$-equipped poset (see (\ref{definition})), $x\leq^t z$, with $t\geq n$, then
$$\mathcal{T}_{x,y}\mathcal{T}_{y,z}\subseteq \mathcal{T}_{x,z};$$ 
so $\mathcal{T}$ satisfies the condition M.1.

Notice that $A_\ell$ is an $A_1$-bimodule, for $1\leq \ell \leq p$, and, when $\ch \textsf{F}\neq p$, due to linear independence of that Galois automorphisms, 
\begin{equation}\label{dimAl}
\dim_{A_1}A_\ell = \ell.
\end{equation}

If $\ch \textsf{F}= p$, consider a non trivial equation in $A_\ell$, of the form
\[\mu_{a_1} + \delta * \mu_{a_2} + \cdots + \delta^{\ell-1} * \mu_{a_{\ell}} =0,\]
this is equivalent to 
\[\mu_{a_1} + \mu_{a_2}\delta + \cdots + \mu_{a_{\ell}} \delta^{\ell-1}=0.\]

Let $i$ be the minimum index such that $a_i\neq 0$, composing by $\mu_{a_i^{-1}}$ 
\[\delta^{i-1}+ \mu_{ b_{i+1}}\delta^{i} + \cdots + \mu_{ b_{\ell}} \delta^{l-1}=0,\]
where $b_j=a_ja_i^{-1}$ para $j\in \{i+1,\ldots, \ell-1\}$. Then
\[\delta^{i-1}=-\mu_{ b_{i+1}}\delta^{i} - \cdots - \mu_{ b_{\ell}} \delta^{l-1}.\]

Composing the previous expression from the right by $\delta^{p-i}$, we obtain $\delta^{\ell-1}=0$, which is a contradiction. Then, the equation (\ref{dimAl}) holds for any characteristic of \textsf{F}.

As $\dim_{A_1}A_p = p$, then $A_p=A$.

Therefore
\[\mathcal{T}_x= \begin{cases} 
A_1 \cong \textsf{G}, & \text{ if } x \text{ is weak},\\
A_p=A, & \text{ if } x \text{ is strong}.\\
\end{cases}\]
for each $x\in \mathscr{P}$,
 
So $\mathcal{T}$ also satisfies the condition M.2., which finishes the proof.
\end{proof}

For the categories of representations we have the next result.

\begin{theorem}\label{equivalenciaRepp}
If $\Rep \mathscr{P}$ is the category of representations of a $p$-equipped poset $\mathscr{P}$, over the pair $(\textsf{F},\textsf{G})$, we have the equivalence
\[\Rep \mathscr{P} \cong \Rep (\mathscr{P},A,\mathcal{T});\]
where $A=(\End_{\textsf{F}}\textsf{G})^{op}$ and $\mathcal{T}$ is the $\mathscr{P}$-multiplicative system defined in Proposition \ref{sisMultEnd}.
\end{theorem}

\begin{proof}
By Remark \ref{pequipadosSonGen}, when $\ch \textsf{F} \neq p$, we have the present result because of Theorem \ref{equivalenciaRepGen}. So we will deal in this proof with the case $\ch \textsf{F}= p$.

If $\widetilde{V}=(V,r;V_x : x \in \mathscr{P})\in \Rep \mathscr{P}$, we want to give to $V$ a right $A$-module structure.

Every $a \in A$ can be written, in a unique way, as
$$a=\sum _{i=0}^{p-1}\vartheta^{i}* \mu_{g_{i}}.$$
for some $g_{i}\in G$.

We define the following action of $A$ over $V$, for all $v\in V$:
$$va=\sum _{i=0}^{p-1}r^{i}(v)g_{i}.$$

If $a$ is the unit in $A$, it is clear that $va=v$ for any $v\in V$.

To prove that this action gives to $V$ a right $A$-module structure, it is enough to prove that for every $a,b\in A$, of the form $a=\vartheta^{i}*\mu_{g}=\mu_{g}\vartheta^{i}$ and $b=\vartheta^{j}*\mu_{h} =\mu_{h}\vartheta^{j}$, for some $g,h \in G$ we have:
\begin{equation}\label{VesAmodulo}
(va)b=v(a*b)=v(ba).
\end{equation}

If $\ch \textsf{F}= p$, we have that $\vartheta = \delta$ and $r=d$ is a $\delta$-derivation. By D.2., for all $v\in V$, $g\in G$ y $j<p$,
$$d^{j}(vg)=\sum _{s=0}^{j}{j \choose s}d^{j-s}(v)\delta ^{s}(g).$$

Then,
$$(v(\mu_{g}\delta ^{i}))(\mu_{h}\delta ^{j})=[d^{i}(v)g)](\mu_{h}\delta ^{j})$$
$$=d^{j}(d^{i}(v)g)b=\sum _{s=0}^{j}{j\choose s}d^{j+i-s}(v)\delta^{s}(g)h.$$

Furthermore,
$$(\mu_{h}\delta ^{j})(\mu_{g}\delta ^{i})=\sum _{s=0}^{j}{j \choose s}\mu_{h}(\mu_{\delta ^{s}(g)})\delta ^{j+i-s}
=\sum _{s=0}^{j}{j \choose s}(\mu_{\delta ^{s}(g)h})\delta ^{j+i-s}.$$

Therefore,
$$v[(\mu_{h}\delta ^{j})(\mu_{g}\delta ^{i})] = \sum _{s=0}^{j}{j\choose s}d^{j+i-s}(v)\delta^{s}(g)h.$$

Consequently (\ref{VesAmodulo}) holds.

We conclude that $V$ is a right $A$-module. 

Notice that, for $x\leq^\ell y$ in $\mathscr{P}$, the condition 
$$V_x \mathcal{T}_{x,y} \subseteq V_y,$$
is equivalent to
$$\sum_{i=0}^{\ell-1} r^i(V_x)\subseteq V_y.$$

Hence $(V, V_x:x\in \mathscr{P})$ is an object of $\Rep (\mathscr{P},A,\mathcal{T})$.

Conversely, if $\widetilde{V}=(V, V_x:x\in \mathscr{P}) \in \Rep (\mathscr{P},A,\mathcal{T})$. As $V$ is a right $A$-module, in particular it is is a right $\textsf{G}$-module. We define $r:V\rightarrow V$ for $v\in V$ , 
\begin{equation}\label{derivacionAccion}
r(v)=v\vartheta .
\end{equation}

When $\ch \textsf{F}= p$, we have $r^{p}(v)=v\delta ^{p}=0$. And, for all $v\in V$ y $g\in G$, 
$$r(vg)=(vg)(\delta )=(v\mu_{g})(\delta)=v(\mu_{g}*\delta )=v(\delta \mu_{g})$$
$$=v(\mu_{g}\delta +\mu_{\delta (g)})=r(v)g+v\delta (g).$$
So $r$ is a $\delta $-derivation.

With this $r$, we have again
\[ V_x \mathcal{T}_{x,y} \subseteq V_y \Leftrightarrow \sum_{i=0}^{\ell-1} r^i(V_x)\subseteq V_y;\]
for every $x\leq^\ell y$ in $\mathscr{P}$.

Therefore $(V,r; V_x:x\in \mathscr{P}) \in \Rep \mathscr{P}$.

For two representations $\widetilde{V}=(V,V_x : x \in \mathscr{P})$ and $\widetilde{V'}=(V',V_x' : x \in \mathscr{P})$ of $(\mathscr{P},A,\mathcal{T})$, an application $h:V\rightarrow V'$ is a morphism in $\Rep (\mathscr{P},A,\mathcal{T})$, if and only if it is an $A$-module morphism, and $h(V_x)\subseteq V_x'$, for every $x \in \mathscr{P}$.
Then, for $v\in V$, $a=\sum _{i=0}^{p-1}\vartheta^{i}* \mu_{g_{i}}\in A$,
\begin{align*}
h(va) &= h(v)a;\\
h \left( v \sum _{i=0}^{p-1}\vartheta^{i}* \mu_{g_{i}} \right) &= h(v) \sum _{i=0}^{p-1}\vartheta^{i}* \mu_{g_{i}};\\
\intertext{by using (\ref{derivacionAccion}),}
h \left( \sum _{i=0}^{p-1} r^i(v) \mu_{g_{i}} \right) &= \sum _{i=0}^{p-1} (r')^i h(v) \mu_{g_{i}}.\\
\end{align*}

The last equation holds if and only if $h$ is \textsf{G}-linear, commutes with $r$ and $r'$, and $h(V_x)\subseteq V_x'$, for all $x \in \mathscr{P}$. That is, $h$ is a morphism in $\Rep (\mathscr{P},A,\mathcal{T})$, if and only if it is a morphism in $\Rep \mathscr{P}$. The theorem is proved.
\end{proof}


\section{Incidence algebras}\label{algebras}

Given the algebra $\mathscr{M}_m(\textsf{K})$\index{$\mathscr{M}_m(\textsf{K})$}, of matrices of size $m$ over a field $\textsf{K}$, its standard basis elements are the matrices $e_{i,j}$, with 1 at the place $(i,j)$ and 0 otherwise. Recall that the incidence algebra $A(\mathscr{P})$, of $\mathscr{P}$ over $\textsf{K}$, is a subalgebra of $\mathcal{M}_m(\textsf{K})$ spanned by the elements $e_{i,j}$ with $i\leq j$ in $\mathscr{P}$ (see for instance \cite{ASS}), i. e.,
\[\Lambda(\mathscr{P})=\bigoplus_{\substack{i\leq j \\ i,j\in \mathscr{P}}}e_{i,j}\textsf{K}.\]

Let $(\mathscr{P},A,\mathcal{R})$ be an algebraically equipped poset.

Notice that $\mathcal{R}_{i,j}$ is a $\mathcal{R}_i$-$\mathcal{R}_j$-bimodule, for all $i,j\in \mathscr{P}$.

We define 
\[\Lambda=\Lambda(\mathcal{R})=\bigoplus_{\substack{i\leq j \\ i,j\in \mathscr{P}}}e_{i,j}\mathcal{R}_{i,j};\]
which is a subset of the $\textsf{K}$-algebra $\bigoplus_{\substack{i\leq j \\ i,j\in \mathscr{P}}}e_{i,j}A\subset \mathscr{M}_m(A)$\index{$\Lambda=\Lambda(\mathcal{R})$}. 

Consider $a=\sum_{\substack{i\leq j \\ i,j\in \mathscr{P}}}e_{i,j}a_{i,j}, b=\sum_{\substack{k\leq l \\ k,l\in \mathscr{P}}}e_{k,l}b_{k,l} \in \Lambda$, for some $a_{i,j},b_{i,j}\in \mathcal{R}_{i,j}$. The sum and product 
\[a+b=\sum_{\substack{i\leq j \\ i,j\in \mathscr{P}}}e_{i,j}a_{i,j}+b_{i,j} \in \Lambda;\]
\[ab=\sum_{\substack{i\leq j\leq l \\ i,j,l\in \mathscr{P}}}e_{i,l}a_{i,j}b_{j,l} \in \Lambda;\]
and, for $z\in \textsf{K}$
\[za=\sum_{\substack{i\leq j \\ i,j\in \mathscr{P}}}e_{i,j}za_{i,j}; \hspace{2cm} bz=\sum_{\substack{k\leq l \\ i,j\in \mathscr{P}}}e_{k,l}b_{k,l}z;\]
every $\mathcal{R}_{i,j}$ is a \textsf{K}-vector space so, $za_{i,j}, b_{i,j}z\in \mathcal{R}_{i,j}$, for all $i,j\in \mathscr{P}$, therefore $za, bz  \in \Lambda$.

Then $\Lambda$ has a $\textsf{K}$-algebra structure, but it is not always a subalgebra of $\mathscr{M}_m(A)$, because their units are  not necessarily the same.

We denote $e_i=e_{i,i}1_i$, for all $i\in \mathscr{P}$. If 1 is the unit of $\Lambda$, then $1=\sum_{i\in \mathscr{P}}e_i$ is a decomposition in a sum of ortogonal idempotents.

The subalgebra $S=\sum_{i\in \mathscr{P}}e_{i,i}\mathcal{R}_i\subset \Lambda$, determines a decomposition of $\Lambda$ as a direct sum of $S$-$S$-bimodules:
\begin{equation}\label{semisimpleRadical}
\Lambda=S\oplus J, \ \ \text{ where } \ \ J=\bigoplus_{\substack{i< j \\ i,j\in \mathscr{P}}}e_{i,j}\mathcal{R}_{i,j}.
\end{equation}

It is easy to see that $J$ is a bilateral ideal of $\Lambda$, and $J^m=0$, moreover if $\mathcal{R}_i$ is a semisimple algebra, for all $i\in \mathscr{P}$, then $S$ is semisimple and $J$ is the radical of $\Lambda$.

For a representation $\widetilde{L}$ of $(\mathscr{P},A,\mathcal{R})$ (Definition \ref{defRep}), we define a right $\Lambda$-module, as follows:
\begin{equation}\label{ModdeSis}
\text{\Large\Fontamici u}(\widetilde{L})=\bigoplus_{i\in \mathscr{P}}L_i,
\end{equation}
if $x\in \bigoplus_{i\in \mathscr{P}}L_i$ and $a=\sum_{\substack{i\leq j \\ i,j\in \mathscr{P}}}e_{i,j}r_{i,j}\in \Lambda$, then $xa=\sum_{i=1}^m\textsf{\emph{i}}_i\left(\sum_{j=1}^m\pi_j(x)r_{j,i}\right)$, where $\textsf{\emph{i}}_i:L_i\rightarrow \text{\Large\Fontamici u}(\widetilde{L})$ and $\pi_j:\text{\Large\Fontamici u}(\widetilde{L})\rightarrow L_j$ are the $i$-th canonical inclusion and the $j$-th canonical projection.

Let $h:\widetilde{L} \rightarrow \widetilde{M}$ be a morphism of representations, we define 
$$\text{\Large\Fontamici u}(h):\text{\Large\Fontamici u}(\widetilde{L}) \rightarrow \text{\Large\Fontamici u}(\widetilde{M})$$ 
as follows: If $x=(x_1,x_2,\ldots, x_m)\in \text{\Large\Fontamici u}(\widetilde{L})$, then 
$$\text{\Large\Fontamici u}(h)(x)=(h(x_1),h(x_2),\ldots, h(x_m)).$$

For $a=\sum_{\substack{i\leq j \\ i,j\in \mathscr{P}}}e_{i,j}r_{i,j}\in \Lambda$ and $x=(x_1,x_2,\ldots, x_m)\in \text{\Large\Fontamici u}(\widetilde{L})$, we have
\begin{align*}
\text{\Large\Fontamici u}(h)(xa) &= (h(\pi_1(xa)),h(\pi_2(xa)),\ldots,h(\pi_m(xa)))\\
 &= \left(h\left(\sum_{j=1}^mx_jr_{j,1}\right),h\left(\sum_{j=1}^mx_jr_{j,2}\right),\ldots,h\left(\sum_{j=1}^mx_jr_{j,m}\right)\right)\\
 &= \left(\sum_{j=1}^mh(x_jr_{j,1}),\sum_{j=1}^mh(x_jr_{j,2}),\ldots,\sum_{j=1}^mh(x_jr_{j,m})\right)\\
 &= \left(\sum_{j=1}^mh(x_j)r_{j,1},\sum_{j=1}^mh(x_j)r_{j,2},\ldots,\sum_{j=1}^mh(x_j)r_{j,m}\right)\\
 &= \text{\Large\Fontamici u}(h)(x)a.
\end{align*}
In this way, $\text{\Large\Fontamici u}(h)$ is a right $\Lambda$-module morphism. Clearly, if $h$ is a monomorphism, $\text{\Large\Fontamici u}(h)$ is a monomorphism, too.

Consider the case where the units $1_{i}\in \mathcal{R}_i$ are equal to the unit 1 of $A$, for every $i\in \mathscr{P}$. An example of representation of $(\mathscr{P},A,\mathcal{R})$ is the \textit{constant representation} $\widetilde{M}_c$, associated to a right $A$-module $M$, in which $\widetilde{M}_c=(M,M_i:i\in \mathscr{P})$ with $M_i=M$, for all $i\in \mathscr{P}$.

We denote by $D$, the duality with respect to the field $\textsf{K}$. If $L$ is a left $A$-module, $D(L)$ is a right $A$-module, and we can consider the dual notion of \textit{left constant representation} $\widetilde{L}_c$, associated to $L$, and construct the left $\Lambda$-module $\text{\Large\Fontamici u}(\widetilde{L}_c)=L^m$. With this notation, we have the following result.

\begin{lemma}\label{cteDual}
Let $L$ be a left $A$-module, then
\[D\left(\text{\Large\Fontamici u}\left(\widetilde{L}_c\right)\right)\cong \text{\Large\Fontamici u}\left(\widetilde{D(L)}_c\right).\]
\end{lemma}

\begin{proof}
We have $\text{\Large\Fontamici u}\left(\widetilde{L}_c\right)=L\oplus \cdots \oplus L=L^m$. There is a $\textsf{K}$-vector spaces isomorphism
\[\psi :D\left(\text{\Large\Fontamici u}\left(\widetilde{L}_c\right)\right) \rightarrow D(L)^m=D(L)\oplus \cdots \oplus D(L)=\text{\Large\Fontamici u}\left(\widetilde{D(L)}_c\right).\]

Let $f:\text{\Large\Fontamici u}\left(\widetilde{L}_c\right) \rightarrow \textsf{K}$ be a linear transformation in $D\left(\text{\Large\Fontamici u}\left(\widetilde{L}_c\right)\right)$ and $a=\sum_{\substack{i\leq j \\ i,j\in \mathscr{P}}}e_{i,j}r_{i,j}\in A$, then
\[\psi(f)=(f\textsf{\emph{i}}_1,f\textsf{\emph{i}}_2,\ldots,f\textsf{\emph{i}}_m),\]
besides
\[\psi(fa)=(fa\textsf{\emph{i}}a_1,fa\textsf{\emph{i}}_2,\ldots,fa\textsf{\emph{i}}_m).\]

For $i\in \{1,2,\ldots,m\}$ and $y\in L$
\begin{multline*}
\pi_i(\psi(fa))(y)=(fa)\textsf{\emph{i}}_i(y)=f(a\textsf{\emph{i}}_i(y))=f(a_{1,i}y,a_{2,i}y,\ldots,a_{m,i}y)\\
=\sum_{j=1}^mf(\textsf{\emph{i}}_j(a_{j,i}y))=\sum_{j=1}^mf\textsf{\emph{i}}_ja_{j,i}(y)=(\psi(f)a)(\textsf{\emph{i}}_i(y))=\pi_i(\psi(f)a)(y).
\end{multline*}

Therefore $\psi(fa)=\psi(f)a$, and the lemma is proved.
\end{proof}

The following proposition allows us to construct some other representations.

\begin{proposition}\label{espacio}
Let $V$ be a \textsf{K}-vector space and $\widetilde{L}=(L, L_i:i\in \mathscr{P})$ be a representation of $(\mathscr{P},A,\mathcal{R})$. Then $(V\otimes_{\textsf{K}}L,V\otimes_{\textsf{K}}L_i:i\in \mathscr{P})$ is another representation, denoted by $V\otimes_{\textsf{K}}\widetilde{L}$.
Moreover,
\[V\otimes_{\textsf{K}}\text{\Large\Fontamici u}(\widetilde{L}) \cong \text{\Large\Fontamici u}(V\otimes_{\textsf{K}}\widetilde{L}).\]
\end{proposition}

\begin{proof}
For $i\leq j$ in $\mathscr{P}$, we have 
\[(V\otimes_{\textsf{K}}L_i)\mathcal{R}_{i,j} = V\otimes_{\textsf{K}}L_i\mathcal{R}_{i,j}  \subseteq V\otimes_{\textsf{K}}L_j.\]

There is a $\textsf{K}$-vector spaces isomorphism
\[\psi: V\otimes_{\textsf{K}}\text{\Large\Fontamici u}(\widetilde{L})=V\otimes_{\textsf{K}}\left(\bigoplus_{i\in \mathscr{P}}L_i\right) \rightarrow \text{\Large\Fontamici u}(V\otimes_{\textsf{K}}\widetilde{L})=\bigoplus_{i\in \mathscr{P}}(V\otimes_{\textsf{K}}L_i),\]
such that, for all $v\in V, x=(x_1,x_2,\ldots , x_m)\in \text{\Large\Fontamici u}(\widetilde{L})$,
\[\psi(v\otimes x)=\psi(v\otimes (x_1,x_2,\ldots , x_m))=(v\otimes x_1,v\otimes x_2,\ldots ,v\otimes x_m).\]

If $a=(a_{i,j})\in \Lambda$, then
\begin{align*}
\psi([v\otimes x]a) &= \psi(v\otimes xa) = \psi \left(v\otimes \left(\sum_{j=1}^mx_ja_{j,1},\sum_{j=1}^mx_ja_{j,2},\ldots ,\sum_{j=1}^mx_ja_{j,m} \right)\right) \\
 &= \left(v\otimes \sum_{j=1}^mx_ja_{j,1},v\otimes \sum_{j=1}^mx_ja_{j,2},\ldots ,v\otimes \sum_{j=1}^mx_ja_{j,m} \right) \\
 &= (v\otimes x_1,v\otimes x_2,\ldots ,v\otimes x_m)a = \psi(v\otimes x)a.
\end{align*}

This proves that $\psi$ is a right $\Lambda$-module isomorphism, i. e., 
\[V\otimes_{\textsf{K}}\text{\Large\Fontamici u}(\widetilde{L}) \cong \text{\Large\Fontamici u}(V\otimes_{\textsf{K}}\widetilde{L}).\]
\end{proof}

Let $(\mathscr{P},A,\mathcal{R})$ be an algebraically equipped poset. We denote by $\mathscr{P}^\text{\large \Fontlukas m}$, the poset obtained by adding a maximal point $\text{\Large \Fontlukas m}\,$, to $\mathscr{P}$.

A $\mathscr{P}^\text{\large \Fontlukas m}$-multiplicative system $\mathcal{R}^\text{\large \Fontlukas m}$, is the collection
\begin{equation}\label{sisMultiMaximo}
\begin{split}
& \mathcal{R}_{i,j}, \text{ for every } i\leq j \text{ in } \mathscr{P},\\
& \mathcal{R}_{i,\text{\large \Fontlukas m}}=A, \text{ for every } i\in \mathscr{P}^\text{\large \Fontlukas m}.
\end{split}
\end{equation}

Any representation of $(\mathscr{P}^\text{\large \Fontlukas m},A,\mathcal{R}^\text{\large \Fontlukas m})$ is of the form $\widetilde{L}=(L, L_i:i\in \mathscr{P},L_\text{\large \Fontlukas m})$, notice that $L_\text{\large \Fontlukas m}\mathcal{R}_{\text{\large \Fontlukas m}} = L_\text{\large \Fontlukas m}A\subseteq L_m$, then $L_\text{\large \Fontlukas m}$ is a right $A$-module.

We call $\mathcal{S}$ the full subcategory of $\Rep (\mathscr{P}^\text{\large \Fontlukas m},A,\mathcal{R}^\text{\large \Fontlukas m})$ whose objects are such that $L=L_\text{\large \Fontlukas m}$.

If to each $(L, L_i:i\in \mathscr{P}) \in \Rep (\mathscr{P},A,\mathcal{R})$, we assign the representation $(L, L_i:i\in \mathscr{P},L_\text{\large \Fontlukas m}=L)$ of $(\mathscr{P}^\text{\large \Fontlukas m},A,\mathcal{R}^\text{\large \Fontlukas m})$, we have an equivalence between $\Rep (\mathscr{P},A,\mathcal{R})$ and $\mathcal{S}$.

Moreover, when $A$ is semisimple, if an object $\widetilde{L} \in \Rep (\mathscr{P}^\text{\large \Fontlukas m},A,\mathcal{R}^\text{\large \Fontlukas m})$ is indecomposable, then $\widetilde{L} \in \mathcal{S}$, or $L_i=0$, for every $i\in \mathscr{P}$. More generally, in this case, any $\widetilde{L} \in \Rep (\mathscr{P}^\text{\large \Fontlukas m},A,\mathcal{R}^\text{\large \Fontlukas m})$, can be written $\widetilde{L}=\widetilde{L}^1 \oplus \widetilde{L}^2$, where $\widetilde{L}^1 \in \mathcal{S}$ and $L^2_i=0$, for all $i\in \mathscr{P}$.

Let us construct the algebra $\Lambda=\Lambda(\mathscr{P}^\text{\large \Fontlukas m})$. The equation (\ref{ModdeSis}) determines a functor $\text{\Large\Fontamici u}:\mathcal{S}\rightarrow \Mod \Lambda$, with the following properties.

\begin{proposition}\label{uFielyPleno}
With the previous notation, the functor $\text{\Large\Fontamici u}:\mathcal{S}\rightarrow \Mod \Lambda$ is full and faithful.
\end{proposition}

\begin{proof}
Let $\widetilde{L}=(L, L_i:i\in \mathscr{P}^\text{\large \Fontlukas m})$ and $\widetilde{M}=(M, M_i:i\in \mathscr{P}^\text{\large \Fontlukas m})$ be objects in $\mathcal{S}$ and $f:\widetilde{L}\rightarrow \widetilde{M}$ be a morphism between them, such that $\text{\Large\Fontamici u}(f)=0$. That is
\[\text{\Large\Fontamici u}(f)= (f(L_1),f(L_2),\ldots,f(L_\text{\large \Fontlukas m})) =0,\]
so, $f(L_\text{\large \Fontlukas m})=f(L)=0$, this is, $f=0$, then the functor is faithful.

Let $g:\text{\Large\Fontamici u}(\widetilde{L}) \rightarrow \text{\Large\Fontamici u}(\widetilde{M})$ be a $\Lambda$-module morphism. For $i\in \mathscr{P}$, we choose an element $y\in L_i\subseteq L_\text{\large \Fontlukas m} = L$. Consider $g(\textsf{\emph{i}}_i(y))\in \text{\Large\Fontamici u}(\widetilde{M})$, and its $i$-th projection $\pi_ig(\textsf{\emph{i}}_i(y))\in M_i$.

By defining
\[f=\pi_\text{\large \Fontlukas m}g\textsf{\emph{i}}_\text{\large \Fontlukas m}:L_\text{\large \Fontlukas m}=L \rightarrow M_\text{\large \Fontlukas m}=M\]
we have that
\begin{align*}
f(y) &= (\pi_\text{\large \Fontlukas m} g\textsf{\emph{i}}_\text{\large \Fontlukas m})(y)=(\pi_\text{\large \Fontlukas m}g)(\textsf{\emph{i}}_\text{\large \Fontlukas m}(y))\\
 &= (\pi_\text{\large \Fontlukas m}g)(\textsf{\emph{i}}_i(y)e_{i,\text{\large \Fontlukas m}}) = \pi_\text{\large \Fontlukas m}(g(\textsf{\emph{i}}_i(y))e_{i,\text{\large \Fontlukas m}}) = \pi_i(g(\textsf{\emph{i}}_i(y)))\\
 &= (\pi_ig\textsf{\emph{i}}_i)(y) \in M_i.
\end{align*}

Then $f(L_i)\subseteq M_i$, therefore $f$ is a morphism in $\mathcal{S}$ such that
\[\text{\Large\Fontamici u}(f)=(g\textsf{\emph{i}}_1,g\textsf{\emph{i}}_2,\ldots,g\textsf{\emph{i}}_\text{\large \Fontlukas m})=g.\]

We conclude that $\text{\Large\Fontamici u}$ is a full functor.
\end{proof}

\begin{proposition}\label{submodulo}
Suppose that $X\in \Mod A$ is of the form $X=\text{\Large\Fontamici u}(\widetilde{L})$, for some $\widetilde{L}=(L, L_i:i\in \mathscr{P}^\text{\large \Fontlukas m})\in \mathcal{S}$. If $Y$ is a submodule of $X$, then there exists $\widetilde{N} \in \mathcal{S}$, such that $Y=\text{\Large\Fontamici u}(\widetilde{N})$.
\end{proposition}

\begin{proof}
For every $i\in \mathscr{P}$,
\[Ye_i\subseteq Xe_i=\text{\Large\Fontamici u}(\widetilde{L})e_i=\textsf{\emph{i}}_i(L_i).\]

The  $e_iAe_i$-module $Ye_i$ has a \textsf{K}-vector space structure because $e_iAe_i\cong \mathcal{R}_i$ is a \textsf{K}-algebra. Besides, as  $L_i$ is a \textsf{K}-subspace of $L$, there exists a \textsf{K}-subspacie $N_i\subseteq L_i$ such that $\textsf{\emph{i}}_i(N_i)=Ye_i$.

We define the collection 
\[\widetilde{N}=(N_\text{\large \Fontlukas m}, N_i:i\in \mathscr{P}^\text{\large \Fontlukas m}).\]

Clearly, $\widetilde{N}$ satisfies L.2.

Consider $i,j\in \mathscr{P}$ such that $i\leq j$, $y\in N_i$ and $a\in \mathcal{R}_{i,j} \cong e_iAe_j$. We have that $\textsf{\emph{i}}_i(y)\in Ye_i$, then $\textsf{\emph{i}}_i(y)(e_iae_j)\in Ye_j$. Besides,
\[\textsf{\emph{i}}_i(y)(e_iae_j) = \textsf{\emph{i}}_i(ya)e_{i,j} = \textsf{\emph{i}}_j(ya),\]
so $\textsf{\emph{i}}_j(ya) \in \textsf{\emph{i}}_j(N_j)$, therefore $ya=\pi_j(\textsf{\emph{i}}_j(ya))\in \pi_j\textsf{\emph{i}}_j(N_j)=N_j$. This proves that $N_i \mathcal{R}_{i,j} \subseteq N_j$.

In particular, $N_\text{\large \Fontlukas m} \mathcal{R}_{\text{\large \Fontlukas m}}=N_\text{\large \Fontlukas m}A \subseteq N_\text{\large \Fontlukas m}$, then $N_\text{\large \Fontlukas m}$ is an $A$-module containing $N_i$, for all $i\in \mathscr{P}$.

We conclude that $\widetilde{N}$ is an object of $\mathcal{S}$, such that
\[Y=\text{\Large\Fontamici u}(\widetilde{N})=\bigoplus_{i\in \mathscr{P}^\text{\large \Fontlukas m}}N_i \subseteq \bigoplus_{i\in \mathscr{P}^\text{\large \Fontlukas m}}L_i,\]
which finishes the proof.
\end{proof}


\section{Right peak algebras}\label{PicoDerechas}

\begin{definition}\label{sisAdm}
Let $\mathscr{P}$ be a partially ordered set with a maximal element {\large \Fontlukas m}. An \textit{admissible system} $\mathcal{R}=\{\mathcal{R}_{i,j}\}_{i\leq j}$, is a $\mathscr{P}$-multiplicative system that satisfies the following additional conditions
\begin{enumerate}[\text{A}.1]
\item $\mathcal{R}_i=\mathcal{R}_{i,i}$ is a division $\textsf{K}$-ring, for all $i\in \mathscr{P}$.
\item If $j< \text{\large \Fontlukas m}$ and $x\in \mathcal{R}_{i,j}$ is different from 0, then there exists $y\in \mathcal{R}_{j,l}$, with $l\neq j$, such that $xy\neq 0$.
\end{enumerate}
\end{definition}

Let us construct $\Lambda=\Lambda(\mathcal{R})$, if $\mathcal{R}$ is an admissible system, the projective $\Lambda$-module $e_\text{\large \Fontlukas m}\Lambda$ is simple.

To a $p$-equipped or generalized equipped poset $\mathscr{P}$, we assign some $\mathscr{P}$-multiplicative systems $\mathcal{Q}$ and $\mathcal{T}$ (see Sections \ref{genSonALg} and \ref{pequipadosSonALg}). These systems are extended to $\mathcal{Q}^\text{\large \Fontlukas m}$ and $\mathcal{T}^\text{\large \Fontlukas m}$ following (\ref{sisMultiMaximo}). Recall that
\[\mathcal{Q}_x^\text{\large \Fontlukas m}= \begin{cases} 
\textsf{F}, & \text{ if } x \text{ is weak},\\
\textsf{G}, & \text{ if } x \text{ is strong},\\
\end{cases}\]
so A.1. holds for $\mathcal{Q}^\text{\large \Fontlukas m}$. For any $j\in \mathscr{P}$ and  $x\in \mathcal{Q}_{i,j}$ a no-null element, it is enough to take the unit $1\in \textsf{G}= \mathcal{Q}_{j,\text{\large \Fontlukas m}}$, to satisfy A.2. Therefore $\mathcal{Q}^\text{\large \Fontlukas m}$ is an admissible system.

The system $\mathcal{T}^\text{\large \Fontlukas m}$ satisfies A.2, because if $x\in \mathcal{T}_{i,j}$ is such that $x\neq 0$, for some $j\in \mathscr{P}$, there is the identity $id\in A=\mathcal{T}_{j,\text{\large \Fontlukas m}}$ such that $x=x id\neq 0$. However, $\mathcal{T}_{\text{\large \Fontlukas m}}=A$ is not a division ring, then A.1. does not hold.

We will describe next, some interesting properties of admissible systems. As $\mathcal{T}^\text{\large \Fontlukas m}$ does not have those properties, we will construct an admissible system $\mathcal{R}$, in such a way that the algebras $\Lambda(\mathcal{T}^\text{\large \Fontlukas m})$ and $\Lambda(\mathcal{R})$ are Morita equivalent.

\begin{lemma}\label{Mem}
Let $\mathcal{R}$ be an admissible system, and $\Lambda=\Lambda(\mathcal{R})$. If $M$ is a right $\Lambda$-module, then $Me_\text{\large \Fontlukas m}$ is a submodule of $M$, isomorphic to $(e_\text{\large \Fontlukas m}\Lambda)^\nu$, for some cardinal $\nu$. Furthermore, $\nu=\dim_{\mathcal{R}_\text{\large \Fontlukas m}}Me_\text{\large \Fontlukas m}$.
\end{lemma}

\begin{proof}
We have $Me_\text{\large \Fontlukas m}\Lambda =Me_\text{\large \Fontlukas m}\Lambda e_m\subset Me_\text{\large \Fontlukas m}$, therefore $Me_\text{\large \Fontlukas m}$ is a submodule of $M$.

For any $x\in Me_\text{\large \Fontlukas m}$, $x\neq 0$, we define $f_x:e_\text{\large \Fontlukas m}\Lambda \rightarrow Me_\text{\large \Fontlukas m}: a\mapsto xa$, for $a\in e_\text{\large \Fontlukas m}\Lambda$. Notice that $f_x$ is a right  $\Lambda$-module morphism such that $f(e_\text{\large \Fontlukas m})=xe_\text{\large \Fontlukas m}=x$, then $\I f_x$ is a submodule of $Me_\text{\large \Fontlukas m}$ isomorphic to $e_\text{\large \Fontlukas m}\Lambda$, hence, $Me_\text{\large \Fontlukas m}$ is a sum of simple modules isomorphic to $e_\text{\large \Fontlukas m}\Lambda$, that is why $Me_\text{\large \Fontlukas m}$ is isomorphic to a direct sum of the form $(e_\text{\large \Fontlukas m}\Lambda)^\nu$, for some cardinal $\nu$. It is clear that $\nu=\dim_{\mathcal{R}_\text{\large \Fontlukas m}}Me_\text{\large \Fontlukas m}$.
\end{proof}

\begin{proposition}
Let $\mathcal{R}$ be an admissible system. For every $i\in \mathscr{P}$, we have
\[\soc (e_i\Lambda)\cong (e_\text{\large \Fontlukas m}\Lambda)^{\nu(i)},\]
with $\nu(i)=\dim_{\mathcal{R}_\text{\large \Fontlukas m}}\mathcal{R}_{i,\text{\large \Fontlukas m}}$.
\end{proposition}

\begin{proof}
For $j\neq \text{\large \Fontlukas m}$, consider $x\in \soc (e_i\Lambda)e_j\subset e_i\Lambda e_j$, $x\neq 0$, then $x=e_{i,j}x_0$, for some $x_0\in \mathcal{R}_{i,j}$, with $x_0\neq 0$. By definition \ref{sisAdm}, A.2, there exists $y_0\in \mathcal{R}_{j,l}$, for some $l\neq j$, such that $x_0y_0\neq 0$.

Then $y=e_{j,l}y_0$ belongs to the radical of $\Lambda$, and $xy=x_0e_{i,j}e_{j,l}y_0=x_0y_0e_{i,l}\neq 0$, which contradicts $x\in \soc (e_i\Lambda)$. Hence,
\[\soc (e_i\Lambda)=e_i\Lambda e_\text{\large \Fontlukas m}.\]
By Lemma \ref{Mem}, $e_i\Lambda e_\text{\large \Fontlukas m}$ is a submodule of $e_i\Lambda$, such that $e_i\Lambda e_\text{\large \Fontlukas m}=\soc (e_i\Lambda)\cong (e_\text{\large \Fontlukas m}\Lambda)^{\nu(i)}$, with $\nu(i)=\dim_{\mathcal{R}_\text{\large \Fontlukas m}}\mathcal{R}_{i,\text{\large \Fontlukas m}}$.
\end{proof}

Following  \cite{Simson}, we call an algebra \textit{right peak}, if its socle is a direct sum of finite copies of a simple projective module. From the previous result we have:

\begin{proposition}
If $\mathcal{R}$ is an admissible system, then $\Lambda=\Lambda(\mathcal{R})$ is a right peak algebra.
\end{proposition}

\begin{proof}
We have that 
\[\soc \Lambda_\Lambda=\bigoplus_{i\in \mathscr{P}}\soc(e_i\Lambda)=\bigoplus_{i\in \mathscr{P}}(e_\text{\large \Fontlukas m}\Lambda)^{\nu(i)}=(e_\text{\large \Fontlukas m}\Lambda)^{\nu},\]
besides, $\nu=\dim_{\mathcal{R}_\text{\large \Fontlukas m}}\Lambda e_\text{\large \Fontlukas m}$.
\end{proof}


\subsection{Morita equivalence}\label{Morita}

Consider a poset $(\mathscr{P},A,\mathcal{T})$ as in Section \ref{genSonALg} or \ref{pequipadosSonALgRep}, and its extension to $(\mathscr{P}^\text{\large \Fontlukas m},A,\mathcal{T}^\text{\large \Fontlukas m})$. 

By Proposition \ref{sisMultEnd} and Equation (\ref{TxSubalgebra}), we have that $\mathcal{T}_{x}$ is semisimple with one simple, up to isomorphism, then, for all $x\in \mathscr{P}^\text{\large \Fontlukas m}$, there exist idempotents $\varepsilon_{x}^{1},\ldots , \varepsilon_{x}^{n(x)}\in \mathcal{T}_{x}$ such that
$$\mathcal{T}_{x}=\varepsilon_{x}^{1}\mathcal{T}_{x}\oplus \cdots \oplus \varepsilon_{x}^{n(x)}\mathcal{T}_{x};$$
with $\varepsilon_{x}^{1}\mathcal{T}_{x}\cong \cdots \cong \varepsilon_{x}^{n(x)}\mathcal{T}_{x}$.

For every $x\leq y$ in $\mathscr{P}^\text{\large \Fontlukas m}$, we define the following subspace of $A$,
$$\mathcal{R}_{x,y}=\varepsilon_{x}^{1}\mathcal{T}_{x,y}\varepsilon_{y}^{1}.$$

The collection $\mathcal{R}=\{\mathcal{R}_{x,y}\}_{\substack{x\leq y \\ x,y\in \mathscr{P}^\text{\large \Fontlukas m}}}$ is a $\mathscr{P}^\text{\large \Fontlukas m}$-multiplicative system.

For every $x\in \mathscr{P}^\text{\large \Fontlukas m}$, 
$$\mathcal{R}_{x} = \varepsilon_{x}^{1}\mathcal{T}_{x}\varepsilon_{x}^{1} \cong \End_{\mathcal{T}_{x}}(\varepsilon_{x}^{1}\mathcal{T}_{x});$$
with $\varepsilon_{x}^{1}\mathcal{T}_{x}$ a simple module, then $\mathcal{R}_x$ is a field, and A.1 holds.

For $j\in \mathscr{P}$ if there is $x\in \mathcal{R}_{i,j}$ with $x\neq 0$, we choose $id\in \mathcal{T}_{j,\text{\large \Fontlukas m}}=A$, 
$$x=xid=x\varepsilon_{j}^{1}id=x\varepsilon_{j}^{1} (\varepsilon_{\text{\large \Fontlukas m}}^{1} + \cdots + \varepsilon_{\text{\large \Fontlukas m}}^{n(\text{\large \Fontlukas m})});$$
then, for some idempotent $\varepsilon_{\text{\large \Fontlukas m}}^{\ell}$ we have $x\varepsilon_{j}^{1} \varepsilon_{\text{\large \Fontlukas m}}^{\ell}\neq 0$.

Every isomorphism $\psi: \varepsilon_{\text{\large \Fontlukas m}}^{1}A \rightarrow \varepsilon_{\text{\large \Fontlukas m}}^{\ell}A$ is such that
$$\psi(a) = \psi(\varepsilon_{\text{\large \Fontlukas m}}^{1}a) = \psi(\varepsilon_{\text{\large \Fontlukas m}}^{1})a = \psi(\varepsilon_{\text{\large \Fontlukas m}}^{1})\varepsilon_{\text{\large \Fontlukas m}}^{1} a;$$
for all $a\in \varepsilon_{\text{\large \Fontlukas m}}^{1}A$, in particular, $\varepsilon_{\text{\large \Fontlukas m}}^{\ell} = \psi(a_0)$, for some $a_0\in \varepsilon_{\text{\large \Fontlukas m}}^{1}A$.

We have
$$0 \neq x\varepsilon_{j}^{1} \varepsilon_{\text{\large \Fontlukas m}}^{\ell} = x\varepsilon_{j}^{1} \psi(a_0) = x\varepsilon_{j}^{1} \psi(\varepsilon_{\text{\large \Fontlukas m}}^{1})\varepsilon_{\text{\large \Fontlukas m}}^{1} a_0.$$

Then $\varepsilon_{j}^{1} \psi(\varepsilon_{\text{\large \Fontlukas m}}^{1})\varepsilon_{\text{\large \Fontlukas m}}^{1}\in \mathcal{R}_{j,\text{\large \Fontlukas m}}$, and is such that $x \varepsilon_{j}^{1} \psi(\varepsilon_{\text{\large \Fontlukas m}}^{1})\varepsilon_{\text{\large \Fontlukas m}}^{1} \neq 0$.

Hence, $\mathcal{R}$ satisfies A.2, therefore, it is an admissible system.

Calling $\Lambda = \Lambda(\mathcal{T}^\text{\large \Fontlukas m})$, and 
\begin{equation}\label{idempotente}
\varepsilon = \sum_{i \in \mathscr{P}^\text{\large \Fontlukas m}} e_i \varepsilon_{i}^{1},
\end{equation}
then, $\Lambda(\mathcal{R}) = \varepsilon \Lambda \varepsilon$.

As $\End_{\Lambda}(e_i \varepsilon_{i}^{\ell} \Lambda) \cong e_i \varepsilon_{i}^{\ell} \Lambda e_i \varepsilon_{i}^{\ell} \cong \mathcal{R}_i$ is a field, the indecomposable projective $\Lambda$-modules are of the form $e_i \varepsilon_{i}^{\ell} \Lambda$, for $i\in \mathscr{P}^\text{\large \Fontlukas m}$ and $\ell \in \{1,\ldots , n(i)\}$.

If $m\in \{1,2,\dots,n(i)\}$, with $\ell \neq m$, for each $e_{i}\varepsilon_{i}^{\ell}, e_{i}\varepsilon_{i}^{m}\in \Lambda$, we have isomorphisms $\psi: \varepsilon_{i}^{\ell}A \rightarrow e_{i}\varepsilon_{i}^{m}A$ and $\varphi: \varepsilon_{i}^{m}A \rightarrow e_{i}\varepsilon_{i}^{\ell}A$ such that $\psi(\varepsilon_{i}^{\ell}) = \varepsilon_{i}^{m}$ and $\varphi(\varepsilon_{i}^{m}) = \varepsilon_{i}^{\ell}$. Then
\[e_{i}\varepsilon_{i}^{\ell}= e_i\varphi(\varepsilon_{i}^{m}) = e_i\varphi(\varepsilon_{i}^{m})\varepsilon_{i}^{m} = e_i\varphi(\varepsilon_{i}^{m})\psi(\varepsilon_{i}^{\ell}) = (e_i\varphi(\varepsilon_{i}^{m}))(e_{i}\psi(\varepsilon_{i}^{\ell}))\]
and 
\[e_{i}\varepsilon_{i}^{m}= e_i\psi(\varepsilon_{i}^{\ell}) = e_i\psi(\varepsilon_{i}^{\ell})\varepsilon_{i}^{\ell} = e_i\psi(\varepsilon_{i}^{\ell})\varphi(\varepsilon_{i}^{m}) = (e_i\psi(\varepsilon_{i}^{\ell}))(e_{i}\varphi(\varepsilon_{i}^{m}));\]
Hence,
\[e_{i}\varepsilon_{i}^{\ell}\Lambda \rightarrow e_{i}\varepsilon_{i}^{m}\Lambda : a \mapsto (e_{i}\psi(\varepsilon_{i}^{\ell}))a\]
is a   $\Lambda$-module isomorphism.

If $j\in \mathscr{P}^\text{\large \Fontlukas m}$ is such that $i\neq j$, or
\[\Hom_\Lambda(e_{j}\varepsilon_{j}^{1}\Lambda, e_{i}\varepsilon_{i}^{1}\Lambda )\cong e_{i}\varepsilon_{i}^{1}\Lambda e_{j}\varepsilon_{j}^{1}=0,\] 
or
\[\Hom_\Lambda(e_{i}\varepsilon_{i}^{1}\Lambda , e_{j}\varepsilon_{j}^{1}\Lambda)\cong e_{j}\varepsilon_{j}^{1}\Lambda e_{i}\varepsilon_{i}^{1}=0.\] 
Therefore $e_{i}\varepsilon_{i}^{1}\Lambda$ is not isomorphic to $e_{j}\varepsilon_{j}^{1}\Lambda$.

The projective $\Lambda$-modules of the form $e_i \varepsilon_{i}^{\ell} \Lambda$, for each $i\in \mathscr{P}^\text{\large \Fontlukas m}$, are a representative system of the indecomposable projective $\Lambda$-modules.

We conclude that $\Lambda$ and $\Lambda(\mathcal{R})$ are Morita-equivalent.


\section{Socle-projective modules}

Let $\Lambda(\mathcal{R})$ be the algebra corresponding to an admissible system $\mathcal{R}$. The right $\Lambda(\mathcal{R})$-modules whose socle is projective, were called, in \cite{Simson}, \textit{socle-projective modules}, naturally.  

The only simple projective $\Lambda(\mathcal{R})$-module is $e_\text{\large \Fontlukas m}\Lambda(\mathcal{R})$, by Lema \ref{Mem}, a right $\Lambda(\mathcal{R})$-module $M$ is socle-projective, if and only if
\[\soc M=Me_\text{\large \Fontlukas m} \cong (e_\text{\large \Fontlukas m}\Lambda(\mathcal{R}))^{\nu},\]
with $\nu=\dim_{\mathcal{R}_\text{\large \Fontlukas m}}Me_\text{\large \Fontlukas m}$.

We denote by $E$ the injective envelope of $e_\text{\large \Fontlukas m}\Lambda(\mathcal{R})$.

\begin{proposition}\label{subE}
A right $\Lambda(\mathcal{R})$-module $M$, is socle-projective if and only if it is a submodule of a coproduct of copies of the injective envelope $E$.
\end{proposition}

\begin{proof}
Suppose that $M$ is socle-projective, $\soc M$ is a direct sum of copies of $e_\text{\large \Fontlukas m}\Lambda(\mathcal{R})$, therefore the injective envelope of $\soc M$ is a direct sum of copies of $E$. As the injective envelopes of $M$ and $\soc M$ coincide, and $M$ is a submodule of its injective envelope, we have that $M$ is a submodule of a direct sum of copies of $E$.

Conversely, suppose that $M$ is submodule of $L$, which is a direct sum of copies of $E$. We have $\soc M = \soc L$, and $\soc L$ is a direct sum of copies of $e_\text{\large \Fontlukas m}\Lambda(\mathcal{R})$, therefore $M$ is socle-projective.
\end{proof}

We introduce next a definition that involves the algebraically equipped posets we have considered.

\begin{definition}\label{extensible}
An algebraically equipped poset $(\mathscr{P},A,\mathcal{T})$ is called \textit{extendable} with admissible system $\mathcal{R}$, if $A$ is a semisimple \textsf{K}-algebra with only one simple module $\varepsilon_\text{\large \Fontlukas m} A$, up to isomorphism, $\mathcal{T}_i$ is a subalgebra of $A$, for every $i\in \mathscr{P}$, and there exist idempotents $\varepsilon_i \in \mathcal{T}_i$ such that, when we construct $\mathcal{T}^\text{\large \Fontlukas m}$, the collection $\mathcal{R}= \left\{ \mathcal{R}_{i,j}=\varepsilon_{i}\mathcal{T}_{i,j}\varepsilon_{j}\ | \ i,j \in \mathscr{P}^\text{\large \Fontlukas m} \right\}$ is an admissible system and $\Lambda(\mathcal{R})$ is Morita equivalent to $\Lambda = \Lambda(\mathcal{T}^\text{\large \Fontlukas m})$.
\end{definition}

\begin{lemma}\label{epsAAeps}
For an extendable poset $(\mathscr{P},A,\mathcal{T})$, we have the $A$-module isomorphism
$$\varepsilon_\text{\large \Fontlukas m} A \cong D\left(A \varepsilon_\text{\large \Fontlukas m}\right).$$
\end{lemma}

\begin{proof}
If $\mathcal{R}$ is the admissible system of $(\mathscr{P},A,\mathcal{T})$, then $\mathcal{R}_\text{\large \Fontlukas m} = \varepsilon_\text{\large \Fontlukas m} A \varepsilon_\text{\large \Fontlukas m}$ is a division ring, and a \textsf{K}-algebra. We choose a \textsf{K}-linear transformation $h_0: \varepsilon_\text{\large \Fontlukas m} A \varepsilon_\text{\large \Fontlukas m} \rightarrow \textsf{K}$ such that $h_0(\varepsilon_\text{\large \Fontlukas m}) \neq 0$.

For every $x\in \varepsilon_\text{\large \Fontlukas m} A$ we define 
$$\psi_x: A \varepsilon_\text{\large \Fontlukas m} \rightarrow \varepsilon_\text{\large \Fontlukas m} A \varepsilon_\text{\large \Fontlukas m}: y \mapsto xy.$$

We have a \textsf{K}-vector space morphism
$$\psi: \varepsilon_\text{\large \Fontlukas m} A \rightarrow D\left(A \varepsilon_\text{\large \Fontlukas m}\right) : x \mapsto h_0 \psi_x.$$

For every $a\in A$,
$$h_0 \psi_{xa}(y) = h_0 \psi_x(ay);$$
therefore $\psi$ is a right $A$-module morphism.

The $A$-module $\varepsilon_\text{\large \Fontlukas m} A$ is simple, then $\psi$ is a monomorphism, besides $\varepsilon_\text{\large \Fontlukas m} A$ and $D\left(A \varepsilon_\text{\large \Fontlukas m}\right)$ have the same dimension over \textsf{K}. Hence $\psi$ is an isomorphism, and the lemma is proved.
\end{proof}

Let us recall the functor $\text{\Large\Fontamici u}$ and the constant representations (Section \ref{algebras}) in the following result.

\begin{proposition}\label{Econstante}
If $(\mathscr{P},A,\mathcal{T})$ is an extendable poset with admissible system $\mathcal{R}$, then the injective envelope $E$, of the simple module $e_\text{\large \Fontlukas m}\Lambda(\mathcal{R})$, has the form
\[E \cong \text{\Large\Fontamici u}\left(\widetilde{\varepsilon_\text{\large \Fontlukas m}A}_c\right)\varepsilon;\]
where $\varepsilon = \sum_{i \in \mathscr{P}^\text{\large \Fontlukas m}} e_i \varepsilon_{i}.$
\end{proposition}

\begin{proof}
Naturally, $\Lambda(\mathcal{R})=\varepsilon \Lambda \varepsilon$.

Notice that the injective envelope
$$E=D(\Lambda(\mathcal{R})e_\text{\large \Fontlukas m})=D(\varepsilon \Lambda \varepsilon e_\text{\large \Fontlukas m}) = D(\Lambda \varepsilon_\text{\large \Fontlukas m} e_\text{\large \Fontlukas m})\varepsilon.$$

We have $\mathcal{T}_{i,\text{\large \Fontlukas m}}=A, \text{ for all } i\in \mathscr{P}^\text{\large \Fontlukas m}$, and $A \varepsilon_\text{\large \Fontlukas m}$ is a left $A$-module, then by Lemma \ref{cteDual},
\[E = D\left(\text{\Large\Fontamici u}\left(\widetilde{A \varepsilon_\text{\large \Fontlukas m}}_c\right)\right)\varepsilon \cong \text{\Large\Fontamici u}\left(\widetilde{D(A \varepsilon_\text{\large \Fontlukas m})}_c\right)\varepsilon.\]

Then, by Lemma \ref{epsAAeps},
\[E \cong \text{\Large\Fontamici u}\left(\widetilde{\varepsilon_\text{\large \Fontlukas m}A}_c\right)\varepsilon.\]
\end{proof}

\begin{theorem}\label{equivalenciaRepU}
Let $(\mathscr{P},A,\mathcal{T})$ be an extendable poset with admissible system $\mathcal{R}$. The category $\mathcal{U}$, of socle-projective $\Lambda(\mathcal{R})$-modules is equivalent to $\Rep (\mathscr{P},A,\mathcal{T})$.
\end{theorem}

\begin{proof}
Due to the Morita equivalence, $M$ is a $\Lambda$-module if and only if $M\varepsilon$ is a $\Lambda(\mathcal{R})$-module. 

The category $\Rep (\mathscr{P},A,\mathcal{T})$ is equivalent to $\mathcal{S}$, and the functor $\text{\Large\Fontamici u}:\mathcal{S} \rightarrow \Mod \Lambda$ is full and faithful (proposition \ref{uFielyPleno}). So, there is a full and faithful functor:
\begin{align*}
\hspace{3cm} \mathcal{S} & \rightarrow \mathcal{U}\\
\widetilde{L} & \mapsto \text{\Large\Fontamici u}(\widetilde{L})\varepsilon\ \  \text{ for the objects};\\
h  & \mapsto \text{\Large\Fontamici u}(h)\varepsilon\ \  \text{ for the morphisms}.
\end{align*}

To obtain an equivalence $\mathcal{S} \cong \mathcal{U}$, it is enough to prove that a $\Lambda(\mathcal{R})$-module $M$ is socle-projective if and only if $M \cong \text{\Large\Fontamici u}(\widetilde{L})\varepsilon$, for some $\widetilde{L} \in \mathcal{S}$.

By Proposition \ref{subE}, if $M$ is a socle-projective $\Lambda(\mathcal{R})$-module, then it is a submodule of $E^{\nu}$, for some cardinal $\nu$. 

There is a \textsf{K}-vector space $V$ with dimension $\nu$, such that
\begin{align*}
E^{\nu} &\cong V\otimes_{\textsf{K}}E,\\
 &\cong V\otimes_{\textsf{K}} \text{\Large\Fontamici u}(\widetilde{\varepsilon_\text{\large \Fontlukas m}A}_c) \varepsilon, \text{ (proposition \ref{Econstante})}\\
 &\cong \text{\Large\Fontamici u} (V\otimes_{\textsf{K}}\widetilde{\varepsilon_\text{\large \Fontlukas m}A}_c) \varepsilon, \text{ (proposition \ref{espacio})}.
\end{align*}

Then, $M$ is a submodule of $\text{\Large\Fontamici u}(V\otimes_{\textsf{K}}\widetilde{\varepsilon_\text{\large \Fontlukas m}A}_c)\varepsilon$, with $V\otimes_{\textsf{K}}\widetilde{\varepsilon_\text{\large \Fontlukas m}A}_c \in \mathcal{S}$. So, there is a $\Lambda$-module $M'$, submodule of $\text{\Large\Fontamici u}(V\otimes_{\textsf{K}}\widetilde{\varepsilon_\text{\large \Fontlukas m}A}_c)$, such that $M' \varepsilon =M$. By Proposition \ref{submodulo}, there exists $\widetilde{L} \in \mathcal{S}$ such that $\text{\Large\Fontamici u}(\widetilde{L})\cong M'$, that is,
\[\text{\Large\Fontamici u}(\widetilde{L})\varepsilon \cong M.\]

Conversely, for $M \in \Mod \Lambda(\mathcal{R})$ of the form $M=\text{\Large\Fontamici u}(\widetilde{L})\varepsilon$, with $\widetilde{L}=(L, L_i:i\in \mathscr{P}^\text{\large \Fontlukas m}) \in \mathcal{S}$. As $L_\text{\large \Fontlukas m}=L$ is a right $A$-module, it has the form $L=V\otimes_{\textsf{K}}\varepsilon_\text{\large \Fontlukas m}A$, for some \textsf{K}-vector space $V$. 

Hence $\widetilde{L} \subseteq V\otimes_{\textsf{K}}\widetilde{\varepsilon_\text{\large \Fontlukas m}A}_c$, and then $\text{\Large\Fontamici u}(\widetilde{L})$ is a submodule of 
$$\text{\Large\Fontamici u}(V\otimes_{\textsf{K}}\widetilde{\varepsilon_\text{\large \Fontlukas m}A}_c) \cong V\otimes_{\textsf{K}}\text{\Large\Fontamici u}(\widetilde{\varepsilon_\text{\large \Fontlukas m}A}_c).$$

Therefore $M=\text{\Large\Fontamici u}(\widetilde{L})\varepsilon$ is a submodule of $V\otimes_{\textsf{K}}\text{\Large\Fontamici u}(\widetilde{\varepsilon_\text{\large \Fontlukas m}A}_c)\varepsilon \cong E^{\nu}$. By Proposition \ref{subE}, the module $M$ is socle-projective, and the proof is finished.
\end{proof}


\section{1-Gorenstein algebras}\label{unoG}

Let $\Lambda=\Lambda(\mathcal{R})$ be that algebra associated to an admissible system $\mathcal{R}$, with the notation of the previous sections, we construct the algebra
\begin{equation}\label{algGor}
\widetilde{\Lambda}=
 \begin{pmatrix} \mathcal{R}_\text{\large \Fontlukas m} & E \\
 0 & \Lambda \end{pmatrix}.
 \end{equation}
 
Consider the idempotents 
\[e_0=  \begin{pmatrix} 1_\text{\large \Fontlukas m} & 0 \\ 0&0  \end{pmatrix},\]
by abuse of notation, for $i\in \mathscr{P}$
\[e_i=  \begin{pmatrix} 0 & 0 \\ 0& e_i  \end{pmatrix},\]
and $\epsilon = \sum_{i\in \mathscr{P}^\text{\large \Fontlukas m}}e_i.$

We have a $\mathcal{R}_\text{\large \Fontlukas m}-\Lambda$-bimodule isomorphism $e_0\widetilde{\Lambda}\epsilon =  \begin{pmatrix} 0 & E \\ 0&0  \end{pmatrix} \cong E$.

Let us describe the right and left $\widetilde{\Lambda}$-modules.

For every right $\widetilde{\Lambda}$-module $M$, we have that $Me_0$ is a right $\mathcal{R}_\text{\large \Fontlukas m}$-module, $M\epsilon$ is a right $\Lambda$-module, and the product $Me_0 \times e_0\widetilde{\Lambda}\epsilon \rightarrow M\epsilon$ induces a morphism $\mu : Me_0 \otimes_{\mathcal{R}_\text{\large \Fontlukas m}} e_0\widetilde{\Lambda}\epsilon \rightarrow M\epsilon$. Through the descomposition $M=Me_0 \oplus M\epsilon$, the right $\widetilde{\Lambda}$-module structure of $M$ is given by
\begin{equation}\label{modGor}
(m_1,m_2) \begin{pmatrix} r & x \\ 0 & a \end{pmatrix} = (m_1r, \mu (m_1\otimes x)+m_2a),
\end{equation}
for all $m_1\in Me_0, m_2\in M\epsilon, r\in \mathcal{R}_\text{\large \Fontlukas m}, x\in E, a\in \Lambda$.

Conversely, if $M_1$ is a right $\mathcal{R}_\text{\large \Fontlukas m}$-module, $M_2$  is a right $\Lambda$-module, and there is a right $\Lambda$-module morphism $\mu : M_1\otimes_{\mathcal{R}_\text{\large \Fontlukas m}} E \rightarrow M_2$, the equation (\ref{modGor}) induces a right $\widetilde{\Lambda}$-module structure in $M=M_1\oplus M_2$.

For any triple $(M'_1,M'_2,\mu')$ where $M'_1$ is a right $\mathcal{R}_\text{\large \Fontlukas m}$-module, $M'_2$  is a right $\Lambda$-module, and $\mu' : M'_1\otimes_{\mathcal{R}_\text{\large \Fontlukas m}} E \rightarrow M'_2$ is a right $\Lambda$-module morphism, by (\ref{modGor}), we have a right $\widetilde{\Lambda}$-module $M'=M'_1\oplus M'_2$.
 
A right $\widetilde{\Lambda}$-module morphism $h=(h_1,h_2):M \rightarrow M'$ is determined by an $\mathcal{R}_\text{\large \Fontlukas m}$-module morphism $h_1:M_1\rightarrow M'_1$ and a $\Lambda$-module morphism $h_2:M_2\rightarrow M'_2$, such that the following diagram commutes
\begin{equation}\label{diagISO}
\begin{CD}
M_1\otimes_{\mathcal{R}_\text{\large \Fontlukas m}} E @>{\mu}>> M_2 \\
@V{h_1\otimes id}VV @VV{h_2}V \\
M'_1\otimes_{\mathcal{R}_\text{\large \Fontlukas m}} E @>>{\mu'}> M'_2
\end{CD}
\end{equation}
clearly $h$ is an isomorphism if and only if $h_1$ and $h_2$ are isomorphisms.

A left $\widetilde{\Lambda}$-module $N$, is given by a decomposition into $\textsf{K}$-vector spaces $N=N_1\oplus N_2$, with $N_1=e_0N$ a left $\mathcal{R}_\text{\large \Fontlukas m}$-module, $N_2=\epsilon N$ a left $\Lambda$-module, and a left $\mathcal{R}_\text{\large \Fontlukas m}$-module morphism $\nu : E\otimes_\Lambda N_2 \rightarrow N_1$. The left $\widetilde{\Lambda}$-module structure of $N$ is as follows
\begin{equation}\label{modGorIZ}
\begin{pmatrix} r & x \\ 0 & a \end{pmatrix} \begin{pmatrix} n_1 \\ n_2 \end{pmatrix} = \begin{pmatrix} rn_1 + \nu(x\otimes n_2) \\ an_2 \end{pmatrix}
\end{equation}

Consider a right $\widetilde{\Lambda}$-module $D(N)$. We have $D(N)e_0=D(e_0N)=D(N_1)$ and $D(N)\epsilon = D(\epsilon N)=D(N_2)$. There exists a morphism $\mu : D(N_1) \otimes_{\mathcal{R}_\text{\large \Fontlukas m}} E \rightarrow D(N_2)$ given by
\begin{equation}\label{morDual}
\mu(\eta_1 \otimes x)(n_2) = \eta_1(\nu(x\otimes n_2)),
\end{equation}
where $\eta_1:N_1\rightarrow \textsf{K}$ belongs to $D(N_1)$, $x\in E$ and $n_2\in N_2$.

\begin{proposition}\label{Eproy}
Let $\widetilde{\Lambda}$ be as in the equation (\ref{algGor}). There is an isomorphism:
\[D(\widetilde{\Lambda}e_\text{\large \Fontlukas m})\cong e_0\widetilde{\Lambda}.\]
\end{proposition}

\begin{proof}
The $\widetilde{\Lambda}$-module $e_0\widetilde{\Lambda}$ is determined by the triple $(\mathcal{R}_\text{\large \Fontlukas m},E,\mu)$, where $\mu$ is induced by the multiplication of $\mathcal{R}_\text{\large \Fontlukas m}$ times $E$, i. e.
\[\mu : \mathcal{R}_\text{\large \Fontlukas m} \otimes_{\mathcal{R}_\text{\large \Fontlukas m}} E \rightarrow E : r\otimes x \mapsto rx.\]

We have 
\[e_0\widetilde{\Lambda}e_\text{\large \Fontlukas m}\cong Ee_\text{\large \Fontlukas m}\]
as left $\mathcal{R}_\text{\large \Fontlukas m}$-modules,
\[\epsilon \widetilde{\Lambda}e_\text{\large \Fontlukas m}\cong \Lambda e_\text{\large \Fontlukas m}\]
as left $\Lambda$-modules, and the left $\widetilde{\Lambda}$-module structure of $\widetilde{\Lambda}e_\text{\large \Fontlukas m}$, induces the morphism
\[\nu : E\otimes_\Lambda \Lambda e_\text{\large \Fontlukas m}\rightarrow Ee_\text{\large \Fontlukas m}\cong D(\mathcal{R}_\text{\large \Fontlukas m})\]
given by the action of $\Lambda e_\text{\large \Fontlukas m}$ on $E$.

Therefore $\widetilde{\Lambda}e_\text{\large \Fontlukas m}$ is isomorphic to the left $\widetilde{\Lambda}$-module determined by the triple
\[(Ee_\text{\large \Fontlukas m}, \Lambda e_\text{\large \Fontlukas m}, \nu ).\]

Then the dual $D(\widetilde{\Lambda}e_\text{\large \Fontlukas m})$, is determined by the triple $(D(Ee_\text{\large \Fontlukas m}), D(\Lambda e_\text{\large \Fontlukas m}), \rho )$, where $\rho : D(Ee_\text{\large \Fontlukas m})\otimes_{\mathcal{R}_\text{\large \Fontlukas m}} E \rightarrow D(\Lambda e_\text{\large \Fontlukas m})$ is a morphism given by the equation (\ref{morDual}).

We have $D(Ee_\text{\large \Fontlukas m})\cong DD(e_\text{\large \Fontlukas m}\Lambda e_\text{\large \Fontlukas m})\cong DD(\mathcal{R}_\text{\large \Fontlukas m})$ and $D(\Lambda e_\text{\large \Fontlukas m})=E$, and we want to see that the pair of morphisms $\ev : DD(\mathcal{R}_\text{\large \Fontlukas m})\rightarrow \mathcal{R}_\text{\large \Fontlukas m}$ and $id:D(\Lambda e_\text{\large \Fontlukas m})\rightarrow E$ determine a $\widetilde{\Lambda}$-module isomorphism between $D(\widetilde{\Lambda}e_\text{\large \Fontlukas m})$ and $e_0\widetilde{\Lambda}$. Following the diagram (\ref{diagISO}), it is enough to prove
\[\rho(\ev \otimes id)=\mu .\]

For every $r\in \mathcal{R}_\text{\large \Fontlukas m}, x\in E=D(\Lambda e_\text{\large \Fontlukas m}), a\in \Lambda e_\text{\large \Fontlukas m}$, it holds
\[\rho(\ev(r) \otimes x)(a)=\ev(r)(xa)=xa(r)=x(ar),\]
on the other hand
\[\mu(r\otimes x)= rx \text{ y } rx(a)=x(ar).\]

That proves
\[D(\widetilde{\Lambda}e_\text{\large \Fontlukas m})\cong e_0\widetilde{\Lambda}.\]
\end{proof}

Notice that $\widetilde{\Lambda}$ has a decomposition similar to that of the equation (\ref{semisimpleRadical}):
\[\widetilde{\Lambda}=\widetilde{S}\oplus \widetilde{J}, \ \ \text{where} \ \ \widetilde{S}=e_0\mathcal{R}_\text{\large \Fontlukas m}e_0\oplus \sum_{i\in \mathscr{P}^\text{\large \Fontlukas m}}e_i\mathcal{R}_ie_i\ \ \text{y} \ \ \widetilde{J}=E \oplus \bigoplus_{\substack{i< j \\ i,j\in \mathscr{P}^\text{\large \Fontlukas m}}}e_{i,j}\mathcal{R}_{i,j}.\]

\begin{proposition}\label{picoDI}
If $\widetilde{\Lambda}$ has the form (\ref{algGor}), then it is a right and left peak  algebra.
\end{proposition}

\begin{proof}
To see that $\widetilde{\Lambda}$ is right peak, let us consider the socle of the right projective $\widetilde{\Lambda}$-modules.

We have that $e_\text{\large \Fontlukas m}\widetilde{\Lambda}\cong e_\text{\large \Fontlukas m}\Lambda\cong \mathcal{R}_\text{\large \Fontlukas m}$ is a division ring, then it is a simple right projective $\widetilde{\Lambda}$-module. Its injective envelope $E(e_\text{\large \Fontlukas m}\widetilde{\Lambda})=D(\widetilde{\Lambda}e_\text{\large \Fontlukas m})$ is isomorphic to the projective $\widetilde{\Lambda}$-module $e_0\widetilde{\Lambda}$. Therefore
\[\soc e_0\widetilde{\Lambda}\cong e_\text{\large \Fontlukas m}\widetilde{\Lambda}.\]

For $i,j\in \mathscr{P}^\text{\large \Fontlukas m}$ with $j\neq \text{\large \Fontlukas m}$, and $x\in \soc (e_i \widetilde{\Lambda})e_j\subset e_i \widetilde{\Lambda}e_j$ such that $x\neq 0$, then $x=e_{i}x_0e_j$, for some $x_0\in \mathcal{R}_{i,j}$, with $x_0\neq 0$. By Definition \ref{sisAdm}, A.2, there is $y_0\in \mathcal{R}_{j,l}$, for some $l\neq j$, such that $x_0y_0\neq 0$.

Hence $y=e_{j}y_0e_l$ belongs to the radical of $\Lambda$, and $xy=e_ix_0y_0e_l\neq 0$, so $x$ cannot belong to $\soc (e_i \widetilde{\Lambda})$. Therefore
\[\soc (e_i \widetilde{\Lambda})=e_i \widetilde{\Lambda}e_\text{\large \Fontlukas m} = \begin{pmatrix} 0 & 0 \\ 0 & e_i\Lambda e_\text{\large \Fontlukas m} \end{pmatrix},\]
then
\[\soc (e_i \widetilde{\Lambda})\cong (e_\text{\large \Fontlukas m}\Lambda)^{\nu(i)}\cong (e_\text{\large \Fontlukas m}\widetilde{\Lambda})^{\nu(i)}, \ \ \text{where} \ \nu(i)\in \mathbb{N}.\]

When we consider $\widetilde{\Lambda}$ as a right $\widetilde{\Lambda}$-module
\[\soc (\widetilde{\Lambda})\cong (e_\text{\large \Fontlukas m}\widetilde{\Lambda})^{\nu}, \ \ \text{for some} \ \nu \in \mathbb{N},\]
that is, $\widetilde{\Lambda}$ is right peak.

Now, let us compute the socle of the left projective $\widetilde{\Lambda}$-modules.

As $\widetilde{\Lambda}e_0\cong \mathcal{R}_\text{\large \Fontlukas m}$, it is a simple  $\widetilde{\Lambda}$-module.

For $i,j\in \mathscr{P}^\text{\large \Fontlukas m}$ and a non-null element $e_jxe_i\in e_j\widetilde{\Lambda}e_i$, we have $x\in \mathcal{R}_{j,i}$, and there is some $y\in \mathcal{R}_{i,\text{\large \Fontlukas m}}$, such that $xy\neq 0$. The existence of $y$ is given by Definition \ref{sisAdm}, A.2, if $i< \text{\large \Fontlukas m}$, or because $\mathcal{R}_\text{\large \Fontlukas m}$ is a division ring, when $i= \text{\large \Fontlukas m}$.

We choose a linear transformation $h: e_j\widetilde{\Lambda}e_\text{\large \Fontlukas m} \rightarrow \textsf{K}$, such that $h(xy)\neq 0$, then $hx\neq 0$.

The element  $e_0he_j$ is in the radical of $\widetilde{\Lambda}$, and it is such that $e_0he_je_jxe_i=e_0hxe_i\neq 0$, so, $e_jxe_i$ cannot belong to $\soc \widetilde{\Lambda}e_i$. Hence,
\[\soc \widetilde{\Lambda}e_i=e_0 \widetilde{\Lambda}e_i.\]

For any $x\in e_0 \widetilde{\Lambda}e_i$, $x\neq 0$, we define $f_x:\widetilde{\Lambda}e_0 \rightarrow e_0 \widetilde{\Lambda}e_i:a\mapsto ax$, for all $a\in \widetilde{\Lambda}e_0$. As $f_x$ is a left $\widetilde{\Lambda}$-module morphism such that $f(e_0)=e_0x=x$, then $\I f_x$ is a submodule of $e_0 \widetilde{\Lambda}e_i$ isomorphic to $\widetilde{\Lambda}e_0$, therefore $e_0 \widetilde{\Lambda}e_i$ is a sum of simple modules isomorphic to $\widetilde{\Lambda}e_0$, then $e_0 \widetilde{\Lambda}e_i$ is isomorphic to a direct sum $(\widetilde{\Lambda}e_0)^{\mu(i)}$, for some $\mu(i) \in \mathbb{N}$. 

The socle of the left $\widetilde{\Lambda}$-module $\widetilde{\Lambda}$ is
\[\soc _{\widetilde{\Lambda}}\widetilde{\Lambda}=\bigoplus_{i\in \mathscr{P}\cup \{0\}}\soc(\widetilde{\Lambda}e_i)=\bigoplus_{i\in \mathscr{P}}(\widetilde{\Lambda}e_0)^{\mu(i)}=(\widetilde{\Lambda}e_0)^{\mu},\]
for some $\mu \in \mathbb{N}$. Therefore $\widetilde{\Lambda}$ is a left peak algebra, and the proposition is proved.
\end{proof}

\begin{proposition}\label{AesGor}
The algebra $\widetilde{\Lambda}$, constructed in (\ref{algGor}), is 1-Gorenstein.
\end{proposition}

\begin{proof}
Considering $\widetilde{\Lambda}$ as a right $\widetilde{\Lambda}$-module 
\[\soc \widetilde{\Lambda}_{\widetilde{\Lambda}}\cong (e_\text{\large \Fontlukas m} \widetilde{\Lambda})^{\nu}, \ \ \text{for some} \ \nu \in \mathbb{N}.\]

Its injective envelope 
\[E \left( \widetilde{\Lambda}_{\widetilde{\Lambda}} \right)=E(e_\text{\large \Fontlukas m} \widetilde{\Lambda})^{\nu},\]
by Proposition \ref{Eproy} we have $E(e_\text{\large \Fontlukas m} \widetilde{\Lambda})\cong D(\widetilde{\Lambda}e_\text{\large \Fontlukas m})\cong e_0\widetilde{\Lambda}$, then the injective envelope $E\left( \widetilde{\Lambda}_{\widetilde{\Lambda}} \right)$ is projective, therefore $\widetilde{\Lambda}$ is 1-Gorenstein.
\end{proof}

\begin{proposition}\label{admGor}
Let $(\mathscr{P},A,\mathcal{R})$ be an algebraically equipped poset, where $\mathscr{P}$ has a maximal point $\text{\large \Fontlukas m}$ and a minimal point $0$, and $\mathcal{R}$ is an admissible system satisfying 
\begin{enumerate}[$\mathrm{G.1}$]
\item $\mathcal{R}_0=\mathcal{R}_\text{\large \Fontlukas m}$,
\item $\mathcal{R}_{0,i}\cong D(\mathcal{R}_{i,\text{\large \Fontlukas m}})$, for all $i\in \mathscr{P}$,
\end{enumerate}
then the algebra $\Lambda=\Lambda(\mathcal{R})$ is 1-Gorenstein.
\end{proposition}

\begin{proof}
We denote by $\mathscr{P}^{>0}$ the poset obtained from $\mathscr{P}$\index{$\mathscr{P}$} without its minimal point. Clearly, the collection $\mathcal{R}^{>0}=\{\mathcal{R}_{i,j} | i,j\in \mathscr{P}, 0<i,j\}$ is an admissible system for $\mathscr{P}^{>0}$, so we construct the algebra $\Gamma=\Gamma(\mathcal{R}^{>0})$.

Let $E_{\Gamma}$ be the injective envelope of the projective $\Gamma$-module $e_m\Gamma$, we have 
\[E_\Gamma=D(\Gamma e_\text{\large \Fontlukas m})=\bigoplus_{\substack{i>0 \\ i\in \mathscr{P}}}D(\Gamma e_\text{\large \Fontlukas m})e_i=\bigoplus_{\substack{i>0 \\ i\in \mathscr{P}}}D(e_i\Gamma e_\text{\large \Fontlukas m})\cong \bigoplus_{\substack{i>0 \\ i\in \mathscr{P}}}D(\mathcal{R}_{i,\text{\large \Fontlukas m}}) \cong \bigoplus_{\substack{i>0 \\ i\in \mathscr{P}}} \mathcal{R}_{0,i}.\]

Therefore the algebra $\Lambda$ has the form
\[\Lambda=
 \begin{pmatrix} \mathcal{R}_\text{\large \Fontlukas m} & E_{\Gamma} \\
 0 & \Gamma \end{pmatrix},\]
as in the equation (\ref{algGor}). Then, by Proposition \ref{AesGor}, $\Lambda$ is 1-Gorenstein.
\end{proof}

Let $(\mathscr{P},A,\mathcal{T})$ be an extendable poset with admissible system $\mathcal{R}$. We construct an algebraically equipped poset $(\underline{\mathscr{P}}^\text{\large \Fontlukas m} ,A, \underline{\mathcal{T}}^\text{\large \Fontlukas m})$, by adding to $\mathscr{P}^\text{\large \Fontlukas m}$ a minimal point 0, and with the $\underline{\mathscr{P}}^\text{\large \Fontlukas m}$-multiplicative system $\underline{\mathcal{T}}^\text{\large \Fontlukas m}$ given by the collection
\begin{equation}\label{sisMulti1gor}
\begin{split}
& \mathcal{T}_{i,j}, \text{ for every } i\leq j \text{ en } \mathscr{P}^\text{\large \Fontlukas m},\\
& \mathcal{T}_{0,i}=A, \text{ for every } i\in \underline{\mathscr{P}}^\text{\large \Fontlukas m}.
\end{split}
\end{equation}

Notice that $(\underline{\mathscr{P}}^\text{\large \Fontlukas m} ,A, \underline{\mathcal{T}}^\text{\large \Fontlukas m})$ has an admissible system 
\begin{equation}\label{sisAdm1gor}
\underline{\mathcal{R}}= \left\{ \mathcal{R}_{i,j}=\varepsilon_{i}\mathcal{T}_{i,j}\varepsilon_{j}\ | \ i,j \in \underline{\mathscr{P}}^\text{\large \Fontlukas m} \right\}
\end{equation}
which satisfies the hypothesis of Proposition \ref{admGor}, therefore $\Lambda(\underline{\mathcal{R}})$ is 1-Gorenstein.


\section{Categories of modules, morphisms and equivalences}\label{sucAR}

Let $\mathscr{C}$ be a full subcategory of $\mathrm{Mod}\textrm{-}\Lambda$, where $\Lambda$ is a finite dimensional $\textsf{K}$-algebra. Suppose that for every $X, Y, Z\in \mathrm{Mod}\textrm{-}\Lambda$ with $X\in \mathscr{C}$, the following conditions hold:
\begin{enumerate}[C.1]
\item If $X\cong Y$ then $Y\in \mathscr{C}$.
\item If $Z$ is a submodule of $X$, then $Z\in \mathscr{C}$.
\end{enumerate}

\begin{definition}\label{parExacto}
If $M\stackrel{u}{\rightarrow }E\stackrel{v}{\rightarrow }N$\index{$u$} is a sequence of morphisms in $\mathscr{C}$, we say that $(u,v)$ is an \textit{exact pair} if for every $W\in \mathscr{C}$ the following sequences are exact:
$$ 0\rightarrow \mathrm{Hom}_{\Lambda}(W,M)\xrightarrow{\mathrm{Hom}(id_{W},u)} \mathrm{Hom}_{\Lambda}(W,E) \xrightarrow{\mathrm{Hom}(id_{W},v)} \mathrm{Hom}_{\Lambda}(W,N)$$
$$ 0\rightarrow \mathrm{Hom}_{\Lambda}(N,W)\xrightarrow{\mathrm{Hom}(v,id_{W})} \mathrm{Hom}_{\Lambda}(E,W)\xrightarrow{\mathrm{Hom}(u,id_{W})} \mathrm{Hom}_{\Lambda}(M,W)$$
\end{definition}

\begin{proposition}\label{paresExactos}
If $\mathscr{C}$ is a full subcategory of $\mathrm{Mod}\textrm{-}\Lambda$ satisfying $\mathrm{C.1}$ and $\mathrm{C.2}$, a pair of morphisms $M\stackrel{u}{\rightarrow }E\stackrel{v}{\rightarrow }N$ is an exact pair if and only if the sequence
$0\rightarrow M\stackrel{u}{\rightarrow }E\stackrel{v}{\rightarrow }N\rightarrow 0$ is an exact sequence of $\Lambda$-modules, or equivalently the pair $(u,v)$ is exact in $\mathrm{Mod}\textrm{-}\Lambda$.
\end{proposition}

\begin{proof}
It is clear that if $(u,v)$ is exact in $\mathrm{Mod}\textrm{-}\Lambda$, then $(u,v)$ is an exact pair in $\mathscr{C}$.

Suppose that $(u,v)$ is an exact pair in $\mathscr{C}$, let $i:X\rightarrow E$ be the kernel of $v$, by $\mathrm{C.2}$, we have $X\in \mathscr{C}$.

Notice that $i:X\rightarrow E$ is in the kernel of $\mathrm{Hom}(id_{X},v)$ (because $vi=0$),  by the first condition in the definition of exact pair, there exists a unique $s:X\rightarrow M$ such that $us=i$. On the other hand, $i$ is the kernel of $v$ and $vu=0$, so there is a unique $t:X\rightarrow M$ such that $it=u$.

Therefore $its=us=i$, and $i$ is monomorphism, so, $ts=id_{X}$.

Besides $ust=it=u$, as $\mathrm{Hom}_{A}(id_{X},u)$ is a monomorphism, $st=id_{M}$, then $u:M\rightarrow E$ is a kernel of $v:E\rightarrow N$.

Let $\rho :E\rightarrow \mathrm{Coker}u$\index{$p$} be a cokernel of $u$. As $vu=0$, there is a morphism $\lambda :\mathrm{Coker}u\rightarrow N$ such that $v=\lambda \rho$. For every $x\in \Ker \lambda$, there is some $y\in E$ such that $\rho(y)=x$, because $\rho$ is an epimorphism. We have $\lambda(\rho(y))=v(y)=0$, and $\Ker v=\I u=\Ker \rho$, then $\rho(y)=x=0$. Therefore $\lambda $ is a monomorphism. That is $\mathrm{Coker}u$ is isomorphic to a submodule of $N$ which belongs to $\mathscr{C}$, by using C.1 and C.2, we conclude that $\mathrm{Coker}(u)$ is an object of $\mathscr{C}$.

The morphism $\rho$ is in the kernel of $\mathrm{Hom}(u,id_{\mathrm{Coker}u})$, because $pu=0$. By using the second condition of the definition \ref{parExacto}, there exists $\nu :N\rightarrow \mathrm{Coker}(u)$ such that $\nu v=p$, then $\lambda \nu v=\lambda p = v$, therefore $\lambda \nu =id_{N}$. Hence, $\lambda $ is an epimorphism, then it is an isomorphism $\mathrm{Coker}(u)\cong N$. 

Consequently, the sequence
$$0\rightarrow M\stackrel{u}{\rightarrow }E\stackrel{v}{\rightarrow }N\rightarrow 0$$
is exact.
\end{proof}

With dual arguments, we have the following result:

\begin{proposition}
Let $\mathscr{C}$ be a full subactegory of $\mathrm{Mod}\textrm{-}\Lambda$ satisfying $\mathrm{C.1}$, and such that if $M\in \mathscr{C}$ and $N$ is a quotient of $M$, then $N\in \mathscr{C}$.

A pair of morphisms $M\stackrel{u}{\rightarrow }E\stackrel{v}{\rightarrow }N$ is an exact pair if and only if the sequence
$0\rightarrow M\stackrel{u}{\rightarrow }E\stackrel{v}{\rightarrow }N\rightarrow 0$ is an exact sequence of $\Lambda$-modules, or equivalently the pair $(u,v)$ is exact in $\mathrm{Mod}\textrm{-}\Lambda$.
\end{proposition}

Consider an extendable poset $(\mathscr{P},A,\mathcal{T})$, we construct its admissible system $\underline{\mathcal{R}}$ as in \ref{sisAdm1gor}, and its incidence algebra $\Lambda=\Lambda(\underline{\mathcal{R}})$ which is 1-Gorenstein.

Following the previous section, the injective envelope of $e_\text{\large \Fontlukas m}\Lambda$ is $e_0\Lambda$, besides
\[\soc \Lambda \cong (e_\text{\large \Fontlukas m} \Lambda)^{\nu}\ \ \text{and} \ \mathrm{top}D(\Lambda)\cong D(\Lambda e_{0})^{\mu } \ \ \text{for some} \ \nu , \mu \in \mathbb{N}.\]

Let $\mathcal{U}$ be the full subcategory of $\mathrm{Mod}\textrm{-}\Lambda$ whose objects are the right socle-projective $\Lambda$-modules which do not have $e_{0}\Lambda$ as a direct summand. We denote by $\mathcal{V}$ the full subcategory of $\mathrm{Mod}\textrm{-}\Lambda$ whose objects have injective top and do not have $e_{0}\Lambda$ as a direct summand.

\begin{proposition}\label{noMorfismos}
If $M\in \mathcal{U}$, then $\mathrm{Hom}_{\Lambda}(e_{0}\Lambda,M)=0$. If $N\in \mathcal{V}$, then $\mathrm{Hom}_{\Lambda}(N,e_{0}\Lambda)=0$.
\end{proposition}

\begin{proof}
We have $\mathrm{soc}M=e_{\text{\large \Fontlukas m}}\Lambda^{L}$ for some set $L$, then $E(M)=E(\mathrm{soc}M)\cong (e_{0}\Lambda)^{L}$. There is a monomorphism:
$i_{M}:M\rightarrow (e_{0}\Lambda)^{L}$. If there exists a non-null morphism $s:e_{0}\Lambda\rightarrow M$, there is a projection $\pi :(e_{0}\Lambda)^{L}\rightarrow e_{0}\Lambda$ such that $\pi i_{M}s\neq 0$, this is an endomorphism of $e_{0}\Lambda$ and the endomorphism ring of $e_{0}\Lambda$ is a field, then $\pi i_{M}s$ is an isomorphism. Then $e_{0}\Lambda$ is a direct summand of $M$, which contradicts the hypothesis. Therefore $\mathrm{Hom}_{\Lambda}(e_{0}\Lambda,M)=0$.

If $N\in \mathcal{V}$, then $\mathrm{top}N\cong D(\Lambda e_{0})^{L}$, and $\mathrm{top}(e_{0}\Lambda)\cong D(\Lambda e_{0})$. So the projective cover of $N$ has the form $\eta :(e_{0}\Lambda)^{L}\rightarrow N$. 

If there exists a non-zero morphism $t:N\rightarrow e_{0}\Lambda$, then $t\eta \neq 0$, so there is an inclusion $i:e_{0}\Lambda\rightarrow (e_{0}\Lambda)^{L}$ such that $t\eta i$ is a non-zero endomorphism of $e_{0}\Lambda$, and therefore an isomorphism. This implies that $e_{0}\Lambda$ is a direct summand of $N$, which is a contradiction, then $\mathrm{Hom}_{\Lambda}(N,e_{0}\Lambda)=0$, and the proposition is proved
\end{proof}

Let $(L, L_i:i\in \mathscr{P})$ be a representation of $(\mathscr{P},A,\mathcal{T})$, then $\widetilde{L}=(L, 0, L_i:i\in \mathscr{P}, L)$ is a representation of $(\underline{\mathscr{P}}^\text{\large \Fontlukas m} ,A, \underline{\mathcal{R}})$. By Theorem \ref{equivalenciaRepU}, $M=\text{\Large\Fontamici u}(\widetilde{L})\varepsilon$ is a socle-projective $\Lambda$-module such that $Me_0=0$, then $M\in \mathcal{U}$. Conversely, if $M\in \mathcal{U}$, we have $Me_0=0$ because if there exists $m\in Me_0$, $m\neq 0$, we can define a morphism $e_{0}\Lambda\rightarrow M: x \mapsto mx$, for all $x\in e_{0}\Lambda$, which contradicts the previous proposition. Therefore $M=\text{\Large\Fontamici u}(\widetilde{L})\varepsilon$, where $\widetilde{L}=(L, L_0, L_i:i\in \mathscr{P}, L)$, with $L_0=0$, hence $(L, L_i:i\in \mathscr{P})\in \Rep (\mathscr{P},A,\mathcal{T})$. We have an equivalence
\[\Rep (\mathscr{P},A,\mathcal{T}) \cong \mathcal{U}.\]

\begin{proposition}
There exists an equivalence $F:\mathcal{U}\rightarrow \mathcal{V}$, where $F$\index{$F$} is an exact functor.
\end{proposition}

\begin{proof}
For every $M\in \mathcal{U}$ we choose an injective envelope
$$i_{M}:M\rightarrow (e_{0}\Lambda)^{\nu (M)}.$$\index{$i$}

We have an isomorphism
\[(e_{0}\Lambda)^{\nu (M)}/\mathrm{rad}(e_{0}\Lambda)^{\nu (M)}\cong [(e_{0}\Lambda)/\mathrm{rad}(e_{0}\Lambda)]^{\nu (M)},\] 
if $\mathrm{Im}i_{M}$ is not contained in $\mathrm{rad}(e_{0}\Lambda)^{\nu (M)}$, there is a non-null morphism $s:M\rightarrow e_{0}\Lambda/\mathrm{rad}(e_{0}\Lambda)$, this implies the existence of a non-zero morphism from $e_{0}\Lambda$ to $M$, which contradicts Proposition \ref{noMorfismos}. Therefore $\mathrm{Im}i_{M}\subset \mathrm{rad}(e_{0}\Lambda)^{\nu (M)}$.

If $t: (e_{0}\Lambda)^{\nu (M)} \rightarrow (e_{0}\Lambda)^{\nu (M)}/\mathrm{rad}(e_{0}\Lambda)^{\nu (M)}$ is the canonical epimorphism, $ti_M=0$. Considering the cokernel $\eta _{M}:(e_{0}\Lambda)^{\nu (M)} \rightarrow \mathrm{Coker}i_M$, there is a morphism $\mathrm{Coker}i_M \rightarrow (e_{0}\Lambda)^{\nu (M)}/\mathrm{rad}(e_{0}\Lambda)^{\nu (M)}$ which, with $\eta _{M}$, induces an isomorphism $(e_{0}\Lambda)^{\nu (M)}/\mathrm{rad}(e_{0}\Lambda)^{\nu (M)}\cong \mathrm{Coker}i_{M}/\mathrm{rad}\mathrm{Coker}i_{M}$. Therefore $\mathrm{Coker}i_{M}\in \mathcal{V}$.

We define $F:\mathcal{U}\rightarrow \mathcal{V}$, in objects, as $F(M)=\mathrm{Coker}i_{M}$, for every $M \in \mathcal{U}$.

There is an exact sequence of $\Lambda$-modules:
$$0\rightarrow M \xrightarrow {i_{M}} (e_{0}\Lambda)^{\nu (M)} \xrightarrow {j_{M}}F(M)\rightarrow 0.$$\index{$j$}

A morphism $f:M\rightarrow N$ in $\mathcal{U}$ induces a morphism $\hat{f}:(e_{0}\Lambda)^{\nu (M)}\rightarrow (e_{0}\Lambda)^{\nu (N)}$ such that $\hat{f}i_{M}=i_{N}f$. Suppose there is another morphism $\hat{f}_{1}:(e_{0}\Lambda)^{\nu (M)}\rightarrow (e_{0}\Lambda)^{\nu (N)}$ such that $\hat{f}_{1}i_{M}=i_{N}f$, then $(\hat{f}-\hat{f}_{1})i_{M}=0$. There exists $s:F(M)\rightarrow (e_{0}\Lambda)^{\nu (N)}$ such that $\hat{f}-\hat{f}_{1}=s\rho $, if $s\neq 0$ there exists a non-null morphism $F(M)\rightarrow e_{0}\Lambda$, by Proposition \ref{noMorfismos}, we conclude $s=0$, and $\hat{f}=\hat{f}_{1}$. 

Each $f:M\rightarrow N$ in $\mathcal{U}$, induces a unique $\hat{f}:(e_{0}\Lambda)^{\nu (M)}\rightarrow (e_{0}\Lambda)^{\nu (N)}$ such that $\hat{f}i_{M}=i_{N}f$. If $\eta _{N}$ is a cokernel of $i_N$, then $\eta _{N}\hat{f}i_{M}=\eta _{N}i_{N}f=0$. We define $F(f)=\underline{f}$, which finishes the definition of the functor $F$.

Suppose $F(f)=0$, then $\eta _{N}\hat{f}=0$. So there exists a morphism $t:(e_{0}\Lambda)^{\nu (M)}\rightarrow N$ such that $i_{N}t=\hat{f}$. By Proposition \ref{noMorfismos}, $t=0$, hence $\hat{f}=0$, and therefore $f=0$, consequently $F$ is a faithful functor.

Consider a morphism $g:F(M)\rightarrow F(N)$. There exists some $\tilde{g}:(e_{0}\Lambda)^{\nu (M)}\rightarrow (e_{0}\Lambda)^{\nu (N)}$ such that $\eta _{N}\tilde{g}=g\eta _{M}$, because $(e_{0}\Lambda)^{\nu (M)}$ is projective and $\eta _{N}$ is epic. Then $\eta _{N}(\tilde{g}i_{M})=g\eta _{M}i_{M}=0$. The injective envelope $i_N$ is a kernel of $\eta_N$, so there exists a unique $h:M\rightarrow N$ such that $i_{N}h=\tilde{g}i_{M}$, then $F(h)=g$. Therefore $F$ is a full functor.

If $N\in \mathcal{V}$, by definition $\mathrm{top}N\cong [(e_{0}\Lambda)/\mathrm{rad}(e_{0}\Lambda)]^{\nu (N)}$. A minimal projective cover of $N$ has the form $(e_{0}\Lambda)^{\nu (N)}\stackrel{\eta _{N}}{\rightarrow }N$. By Proposition \ref{subE}, the submodule $M=\mathrm{Ker}\eta _{N}$ de $(e_{0}\Lambda)^{\nu (N)}$ is an object of $\mathcal{U}$, such that $F(M)\cong N$. We conclude that $F$ is dense, and an equivalence. Hence, there is a functor $G:\mathcal{V}\rightarrow \mathcal{U}$ such that $FG\cong id_{\mathcal{V}}$ and $GF\cong id_{\mathcal{U}}$.

Let us show that $F$ is an exact functor.

Let $0\rightarrow M_{1}\stackrel{u}{\rightarrow }M_{2}\stackrel{v}{\rightarrow }M_{3}\rightarrow 0$ be an exact sequence in $\mathcal{U}$, then $(u,v)$ is an exact pair in $\mathcal{U}$, so, $(F(u),F(v))$ is an exact pair in $\mathcal{V}$, because $F$ is an equivalence. The hypothesis of Proposition \ref{paresExactos} hold for $\mathcal{V}$, then the following sequence is exact
$$0\rightarrow F(M_{1})\stackrel{F(u)}{\rightarrow}F(M_{2})\stackrel{F(v)}{\rightarrow }F(M_{3})\rightarrow 0.$$
That finishes the proof.
\end{proof}

With the algebra $\Lambda$, we construct the ditalgebra $\mathcal{D}$, as in Example \ref{ditDrozd}. Using the elements $w_1$ y $w_2$, as in equation (\ref{w}), we define the idempotents 
\begin{equation}
e=\sum _{i=1}^{\text{\large \Fontlukas m}}w_1 e_i + w_2 e_0\,; \hspace{1cm} e'=1-e.
\end{equation}

Following the notation of Example \ref{ditDrozd}, we have a surjective $\textsf{K}$-algebra morphism $R\rightarrow eRe$ and a $R$-$R$-bimodule epimorphism $W\rightarrow eWe$. These morphisms induce an epimorphism of $\textsf{K}$-algebras
$$\eta :T_{R}(W)\rightarrow T_{eRe}(eWe).$$

Clearly, $d$ induces a differential $d_{e}$ in $T_{eRe}(eWe)$ such that $d_{e}\eta =\eta d$.

We have a ditalgebra $\mathcal{D}^{e}=(T_{eRe}(eWe),d_{e})$ and a ditalgebra morphism $\eta :\mathcal{D}\rightarrow \mathcal{D}^{e}$. This morphism induces a full and faithful functor
$$F_{\eta }:\mathrm{Mod}\, \mathcal{D}^{e}\rightarrow \mathrm{Mod}\, \mathcal{D}.$$

An object $M\in \mathrm{Mod}\, \mathcal{D}$ has the form $F_{\eta }(N)$ for some $N\in \mathrm{Mod}\, \mathcal{D}^{e}$ if and only if $Me'=0$. So, we say that $\mathrm{Mod}\, \mathcal{D}^{e}$ is a full subcategory of $\mathrm{Mod}\, \mathcal{D}$ whose objects are the right $T_R(W)_0$-modules $M$, such that $Me'=0$.

The category $\mathrm{Mod}\, \mathcal{D}$ is equivalent to the category $\mathcal{M}^1(\Lambda)$ of Definition \ref{catMorfismosProj}. Naturally, there is a subcategory of $\mathcal{M}^1(\Lambda)$ equivalent to $\mathrm{Mod}\, \mathcal{D}^e$, as follows.

\begin{definition}\label{catProblema}
We denote by $\mathcal{M}$ the full subcategory of $\mathcal{M}^1(\Lambda)$ whose objects are morphisms of the form $P\stackrel{\phi }{\rightarrow }(e_{0}\Lambda)^{\nu }$, with $\nu$ a set, and $P$ a projective $\Lambda$-module such that $Pe_{0}=0$. 
\end{definition}

\begin{lemma}\label{equvInducida}
The equivalence $\Xi :\mathrm{Mod}\, \mathcal{D}\rightarrow \mathcal{M}^1(\Lambda)$ induces an equivalence $\Xi _{e}:\mathrm{Mod}\, \mathcal{D}^{e}\rightarrow \mathcal{M}.$
\end{lemma}

\begin{proof}
Using the functors $F$ and $G$ built in the proofs of Propositions \ref{projEquivDitalg} and \ref{DitalgEquivDrozd}, for every $M\in \mathrm{Mod}\, \mathcal{D}$,
$$\Xi (M):M _{1}\otimes _{S}\Lambda \rightarrow M_{2}\otimes _{S}\Lambda.$$

Recall that $M_i=Mw_i1_S$, with $i=1,2$.

If $M\in \mathrm{Mod}\, \mathcal{D}^{e}$, then $Me'=0$.

We have $(M _{1}\otimes _{S}\Lambda)e_0=M _{1}\otimes _{S}\Lambda e_0 \cong M _{1}\otimes _{S} e_0\Lambda e_0 = M _{1}e_0 \otimes _{S}\Lambda e_0 = Mw_1e_0\otimes _{S}\Lambda e_0 = Me'w_1e_0\otimes _{S}\Lambda e_0 = 0$.

Besides, $M_2\otimes _{S}\Lambda = M(e+e')w_21_S\otimes _{S}\Lambda = Mw_2e \otimes _{S}\Lambda = Mw_2e_0 \otimes _{S}\Lambda = M_2 \otimes _{S}e_0\Lambda$.

That is, $\Xi (M)\in \mathcal{M}$, therefore $\Xi $ induces a full and faithful functor
$$\Xi _{e}:\mathrm{Mod}\, \mathcal{D}^{e}\rightarrow \mathcal{M}.$$

If $P\stackrel{\phi }{\rightarrow }(e_{0}\Lambda)^{\nu }\in \mathcal{M}$, then $\phi = \Xi (M)$ for some $M\in \mathrm{Mod}\, \mathcal{D}$ such that $Pe_0=(M _{1}\otimes _{S}\Lambda)e_0= 0$.

Therefore:
\begin{align*}
M_1e' &= Mw_1e_0 \cong Mw_1e_0 \otimes _{e_{0}\Lambda e_{0}}e_{0}\Lambda e_{0} \cong Mw_1e_0\otimes _{S}e_{0}\Lambda e_{0}\\
 &= M_1 \otimes _{S}e_{0}\Lambda e_{0} \cong M_1\otimes _{S}\Lambda e_{0} = Pe_0 =0.
\end{align*}

On the other hand $(e_{0}\Lambda)^{\nu }=M_2\otimes _{S}\Lambda$, hence, $M_2\otimes _{S}\Lambda = M_2\otimes _{S}e_0\Lambda = M_2e_0\otimes _{S}\Lambda,$
by the equivalence built in Proposition \ref{projEquivDitalg}, $M_2\cong M_2e_0$, then $M_2e'\cong M_2e_0e'=0$.

We have that $Me'=M_1e'+M_2e'=0$, that is $M\in \mathrm{Mod}\, \mathcal{D}^{e}$, which proves that $\Xi _{e}$ is dense, and finishes the proof of the proposition.
\end{proof}

\begin{proposition}
The functor $\mathrm{Cok}:\mathcal{M}\rightarrow \mathcal{V}$ induces an equivalence:
$$\mathrm{Cok}:\overline{\mathcal{M}}\rightarrow \mathcal{V}$$
where $\overline{\mathcal{M}}$ is the category $\mathcal{M}$ modulo the morphisms which have a factorization through objects of the form $P\rightarrow 0$ with $Pe_{0}=0$.
\end{proposition}

\begin{proof}
We have that $\mathrm{Cok}$ is a full and dense functor, so it is enough to prove that if $f:\phi_1\rightarrow \phi_2$ is a morphism in $\mathcal{M}$ such that $\mathrm{Cok}(f)=0$, then $f$ has a factorization through an object of the form $P\rightarrow 0$ with $Pe_{0}=0$.

Let $\phi_{1}: Q_1\rightarrow (e_{0}\Lambda)^{\nu_1}$, $\phi_{2}: Q_2\rightarrow (e_{0}\Lambda)^{\nu_2}$ be two objects of $\mathcal{M}$ and $f=(f_{1},f_{2})$ a morphism between them, where $f_{1}: Q_{1}\rightarrow Q_2$ and $f_{2}: (e_{0}\Lambda)^{\nu_1} \rightarrow (e_{0}\Lambda)^{\nu_2}$ such that $f_{2}\phi_1=\phi_2f_{1}$. Suppose $\mathrm{Cok}(f)=0$, we have the following exact sequences
$$Q_{1}\stackrel{\phi_1}{\rightarrow }(e_{0}\Lambda)^{\nu_1}\stackrel{\eta _{1}}{\rightarrow }\mathrm{Cok}\phi_1\rightarrow 0$$
$$Q_2\stackrel{\phi_2}{\rightarrow }(e_{0}\Lambda)^{\nu_2}\stackrel{\eta _{2}}{\rightarrow }\mathrm{Cok}\phi_2\rightarrow 0,$$
and $\eta _{2}f_{2}=\mathrm{Cok}(f)\eta _{1}=0$. Hence, there is a morphism $s:(e_{0}\Lambda)^{\nu_1}\rightarrow Q_2$ such that $\phi_2 s=f_{2}$, by Proposition \ref{noMorfismos}, we have $s=0$, therefore $f_{2}=0$. That is $f=(f_{1},0)$ is the composition of the following morphisms in $\mathcal{M}$: 
$$(id_{Q_{1}},0):(Q_{1}\stackrel{\phi_1}{\rightarrow }(e_{0}\Lambda)^{\nu_1})\rightarrow (Q_{1}\rightarrow 0)$$
$$(f_{1},0):(Q_{1}\rightarrow 0)\rightarrow (Q_2\stackrel{\phi_2}{\rightarrow }(e_{0}\Lambda)^{\nu_2}).$$
The proof is complete.
\end{proof}

We are now going to study $\mathrm{mod}\, \mathcal{D}^e$, the full subcategory of $\mathrm{Mod}\, \mathcal{D}^e$ whose objects are those having finite dimension over \textsf{K}.


\section{Getting the matrix problem}

Let $(\mathscr{P},A,\mathcal{T})$ be an extendable poset with an admissible system $\underline{\mathcal{R}}$, as in \ref{sisAdm1gor}. We know that the algebra $\Lambda=\Lambda(\underline{\mathcal{R}})$ is 1-Gorenstein and has a decomposition $\Lambda=S \oplus J$, as in \ref{semisimpleRadical}.

Consider the ditalgebra $\mathcal{D}^{e}$, constructed in the previous section. Let us write $e=w_1 \hat{e} + w_2 e_0$, where $\hat{e}=1-e_0$, and 1 is the unit in $\Lambda$.

We have
\[eRe= \left( \begin{array}{cc}\hat{e}S\hat{e} &0 \\ 0&e_0S e_0 \end{array} \right) = \left( \begin{array}{cc}\hat{e}S &0 \\ 0& e_0\mathcal{R}_0 \end{array} \right),\]
\[W_{0}=\left( \begin{array}{cc}0&\hat{e}\,^{*}J e_0\\0&0 \end{array} \right) = w_{1,2}\hat{e}\,^{*}J e_0 ,\ \ \ W_{1}=\left( \begin{array}{cc}\hat{e}\,^{*}J\hat{e}&0\\0&0 \end{array} \right) = w_1\hat{e}\,^{*}J\hat{e}.\]

An object $M\in \mathrm{mod}\, \mathcal{D}^e$ is a right module over $\begin{pmatrix} \hat{e}S & \hat{e}\,^{*}J e_0 \\ 0 & e_0\mathcal{R}_0 \end{pmatrix}$, then  $M_1=Mw_1\hat{e}$ is a right $\hat{e}S$-module, and $M_0=Mw_2e_0$ is an $e_0\mathcal{R}_0$-module, so it is an $\mathcal{R}_0$-vector space.

For every $x \in \underline{\mathscr{P}}^\text{\large \Fontlukas m}$, we denote $M_x=Mw_1 e_x$.

We say that $\mathscr{B}(M)\subset M$ is a \textit{local basis} of $M$ if $\mathscr{B}(M) = \bigcup_{x \in \underline{\mathscr{P}}^\text{\large \Fontlukas m}} \mathscr{B}_x(M)$, where $\mathscr{B}_x(M)$ is an $\mathcal{R}_x$-basis of $M_x$.

The \textit{dimension} of $M$ is an integer vector $\underline{d}(M)=(d_x=\dim_{\mathcal{R}_x}M_x )_{x \in \underline{\mathscr{P}}^\text{\large \Fontlukas m}}$. When $M_x=0$, obviously $d_x=0$ and we will obtain formal matrices with 0 rows or 0 columns. By definition, the product for a formal matrix behaves as the zero matrix.

In order to obtain certain matrix problems determined by some algebraically equipped posets, we will use the following remark.

\begin{remark}\label{morfismoGeneral}
Let $S$ and $T$ be rings, and $V$ be a projective $S$-$T$-bimodule finitely generated over $S$ with a dual basis $\{\lambda _{i},p_{i}\}_{i\in I}$, $\lambda _{i}\in \,^{*}V$ and $p_{i}\in V$.

Recall that if $M$ is a right $S$-module, and $N$ is a right $T$-module, there is an isomorphism:
$$\Xi: \mathrm{Hom}_{S}(M \otimes _{T}\,^{*}V,N)\rightarrow \mathrm{Hom}_{T}(M,N\otimes _{S}V);$$
for $h:M\otimes _{T}\,^{*}V\rightarrow N$, we denote $\tilde{h}= \Xi(h) :M\rightarrow N\otimes _{S}V$, and if $g:M\rightarrow N\otimes _{S}V$ is a $T$-module morphism, we call $\underline{g}$ the morphism in $\mathrm{Hom}_{S}(M \otimes _{T}\,^{*}V,N)$ such that $\Xi (\underline{g})=g$.

Then, for $m\in M$,  
$$\tilde{h}(m)=\sum _{i\in I}h(m\otimes \lambda _{i})\otimes p_{i},$$
and, for $m\in M, \lambda \in \,^{*}V$ we have
$$\underline{g}(m\otimes \lambda )=\varsigma (id\otimes \lambda )g(m),$$
where $\varsigma :N\otimes _{S} V \rightarrow N$ is given by $\varsigma (n\otimes s)=ns$.
\end{remark}

In the situations we will consider in the rest of this section, we suppose that $\mathcal{R}_x$ is a field, for all $x \in \underline{\mathscr{P}}^\text{\large \Fontlukas m}$.


\subsection{Matrices over $\mathcal{R}_0$}

For every $x \in \underline{\mathscr{P}}^\text{\large \Fontlukas m}$, suppose that $\mathcal{R}_{0,x}$ is one-dimensional over $\mathcal{R}_0$ (recall that $\mathcal{R}_{x,y}$ is an $\mathcal{R}_x$-$\mathcal{R}_y$-bimodule). Then $e_0Je_x = e_0 \mathcal{R}_{0,x}e_x = e_{0,x}  \mathcal{R}_{0,x} = e_{0,x} \mathcal{R}_0 v_{0,x}$, for some $v_{0,x}\in \mathcal{R}_{0,x}$.

For $y<x$ in $\mathscr{P}$, and $r\in \mathcal{R}_{y,x}$, we have $v_{0,y}r\in \mathcal{R}_{0,y}\mathcal{R}_{y,x} \subseteq \mathcal{R}_{0,x}$, then 
\begin{equation}\label{chi}
v_{0,y}r=\chi_{y,x}(r)v_{0,x},
\end{equation}
for some $\chi _{y,x}(r)\in \mathcal{R}_0$. Hence, there are $\textsf{K}$-linear functions
$$\chi_{y,x}: \mathcal{R}_{y,x} \rightarrow \mathcal{R}_0: r \mapsto \chi _{y,x}(r),$$
satisfying (\ref{chi}).

With the notation $u_{0,x}=e_0v_{0,x}e_x=e_{0,x}v_{0,x}$, in the dual space $e_x\,^{*}Je_0$, we identify $\hat{u}_{0,x}=\widehat{e_{0,x}v_{0,x}}= e_x\hat{v}_{0,x}e_0$. For all 
$z\in  \mathcal{R}_x$, we have $z\hat{u}_{0,x}= e_x z\hat{v}_{0,x}e_0$, then, for each $x\in \mathscr{P}$ there is a $K$-linear function $\phi _{x}: \mathcal{R}_x \rightarrow \mathcal{R}_0$ such that
$$z\hat{u}_{0,x}= e_x\hat{v}_{0,x}\phi _{x}(z)e_0 =\hat{u}_{0,x}\phi _{x}(z).$$

The multiplication of elements in $M\in \mathrm{Mod}\, \mathcal{D}^e$ by elements in $eRe$, determines, for every $\gamma \in \hat{e}\,^{*}Je_0$, a $\textsf{K}$-transformation
\[\psi_M(\gamma): M_1 \rightarrow M_0: m\mapsto m w_{1,2} \gamma;\]
in such a way that $\psi_M: \hat{e}\,^{*}Je_0 \rightarrow \mathrm{Hom}_{\textsf{K}}(M_1,M_0)$ is an $\hat{e}S$-$Se_0$-bimodule morphism (see the proof of Proposition \ref{DitalgEquivDrozd} in the appendix). Notice that $\hat{e}\,^{*}Je_0$ is determined by the elements $\hat{u}_{0,x}$, generators of each $e_x\,^{*}Je_0$, for all $x \in \mathscr{P}^\text{\large \Fontlukas m}$, then $\psi_M$ is determined by 
$$\psi_M(\hat{u}_{0,x}): M_x \rightarrow M_0: m\mapsto m w_{1,2} \hat{u}_{0,x}.$$

Given a local basis $\mathscr{B}(M)$ of $M$, with $\mathscr{B}_x(M)=\{m_i^x : 1\leq i \leq d_x \}$ and $\mathscr{B}_0(M)=\{m_j^0 : 1\leq j \leq d_0 \}$, then
\[\psi_M(\hat{u}_{0,x})(m_i^x) = \sum_{j=1}^{d_0} m_j^0 g_{j,i}^x;\]
for some $g_{j,i}^x \in \mathcal{R}_0$.

We denote $M_{j,i}^x=g_{j,i}^x$, to obtain the matrix $\mathscr{M}(M)_x = \left( M_{j,i}^x \right)$. If $d_0=0$ or $d_x=0$, we consider a formal matrix $\mathscr{M}(M)_x$ with 0 rows or 0 columns. 

The matrix $\mathscr{M}(M)$, divided into vertical stripes $\mathscr{M}(M)_x$, for every $x \in \underline{\mathscr{P}}^\text{\large \Fontlukas m}$, is called \textit{matrix representation} of $M$.

\begin{theorem}\label{MatricesGeneralCoreps}
Let $M$ and $N$ be two objects of $\mathrm{mod}\, \mathcal{D}^e$, with local basis $\mathscr{B}(M)$ and $\mathscr{B}(N)$, respectively. Then, $M\cong N$ if and only if there are non-singular square matrices $T_0\in \mathscr{M}_{d_0} (\mathcal{R}_0)$, $T_x\in \mathscr{M}_{d_x} (\mathcal{R}_x)$, for all $x\in \mathscr{P}^\text{\large \Fontlukas m}$, and for each $y < x$ in $\mathscr{P}^\text{\large \Fontlukas m}$, a matrix $T_{yx} \in \mathscr{M}_{d_x\times d_y} (\mathcal{R}_{y,x})$, such that
\[\mathscr{M}(M)_x= T_0 \mathscr{M}(N)_x \phi_{x}(T_x) + \sum_{y<x} T_0 \mathscr{M}(N)_y \chi_{y,x}(T_{yx}).\]
\end{theorem}

\begin{proof}
If $M$ and $N$ are isomorphic, their dimension vectors are equal.

The pair $(f^0,f^1):M \rightarrow N$ is an isomorphism in $\mathrm{mod}\, \mathcal{D}^e$, if and only if $f^0$ is an $eRe$-module isomorphism that induces isomorphisms $f_0^0:M_0 \rightarrow N_0$ and $f_x^0:M_x \rightarrow N_x$, for all $x\in \mathscr{P}^\text{\large \Fontlukas m}$, such that, for every $m_i^x \in \mathscr{B}_x(M)$,
\begin{align*}
f_x^0(m_i^x) w_{1,2} \hat{u}_{0,x} &= f^0_0 (m_i^x w_{1,2} \hat{u}_{0,x}) + f^1(d_e(w_{1,2} \hat{u}_{0,x})) (m_i^x);\\
\sum_{t=1}^{d_x} n_t^x h_{t,i}^x w_{1,2} \hat{u}_{0,x} &= f^0_0 \left( \sum_{j=1}^{d_0} m_j^0 M_{j,i}^x \right) + f^1(d_e(w_{1,2} \hat{u}_{0,x})) (m_i^x);
\end{align*}
where $n_t^x\in \mathscr{B}_x(N), m_j^0\in \mathscr{B}_0(M), h_{t,i}^x \in \mathcal{R}_x$ and $M_{j,i}^x \in \mathcal{R}_0$, then
\begin{align*}
\sum_{t=1}^{d_x} n_t^x w_{1,2} h_{t,i}^x \hat{u}_{0,x} &= \sum_{j=1}^{d_0}  f^0_0 \left( m_j^0 M_{j,i}^x \right) + f^1(d_e(w_{1,2} \hat{u}_{0,x})) (m_i^x);\\
\sum_{t=1}^{d_x} n_t^x w_{1,2} \hat{u}_{0,x} \phi _{x}(h_{t,i}^x) &= \sum_{j=1}^{d_0}  f^0_0 \left( m_j^0 \right) M_{j,i}^x + f^1(d_e(w_{1,2} \hat{u}_{0,x})) (m_i^x);\\
\sum_{t=1}^{d_x} \sum_{s=1}^{d_0} n_s^0 N_{s,t}^x \phi _{x}(h_{t,i}^x) &= \sum_{j=1}^{d_0} \sum_{s=1}^{d_0} n_s^0 \mu_{s,j} M_{j,i}^x + f^1(d_e(w_{1,2} \hat{u}_{0,x})) (m_i^x);
\end{align*}
with $n_s^0\in \mathscr{B}_0(N)$ and $N_{s,t}^x, \mu_{s,j} \in \mathcal{R}_0$.

To calculate $d_e(w_{1,2} \hat{u}_{0,x})$, consider a basis $\mathscr{B}_0(y,x)$ for every $\mathcal{R}_{y,x}$, as a left $\mathcal{R}_y$-vector space. The elements $e_{y,x}r$ with $r\in \mathscr{B}_0(y,x)$, form a basis $\mathscr{B}_{y,x}$ for $e_y Je_x$. For every $z\in \mathscr{B}= \bigcup_{y,x \in \underline{\mathscr{P}}^\text{\large \Fontlukas m}} \mathscr{B}_{y,x}$, its corresponding dual element $\hat{z} \in \,^*J$, is such that $\hat{z}(w)=0$, if $z\neq w$, and $\hat{z}(z)=e_x$. We have,
\begin{align*}
d_e(w_{1,2} \hat{u}_{0,x}) &= - \sum_{z, v \in \mathscr{B}} (w_1 \hat{z} \otimes w_{1,2} \hat{v}) \hat{u}_{0,x} (vz).
\end{align*}

To have $\hat{u}_{0,x} (vz) \neq 0$, we should have $v\in \mathscr{B}_{0,y}$ and $z\in \mathscr{B}_{y,x}$, for each $y<x$ in $\mathscr{P}^\text{\large \Fontlukas m}$, hence, $v=u_{0,y}$ and $z=e_{y,x}r$ with $r\in \mathscr{B}_0(y,x)$, then,
\begin{align*}
d_e(w_{1,2} \hat{u}_{0,x}) &= - \sum_{y<x} \sum_{r \in \mathscr{B}_0(y,x)} (w_1 \widehat{e_{y,x}r} \otimes w_{1,2} \widehat{u}_{0,y}) \hat{u}_{0,x}(u_{0,y}e_{y,x}r);\\
 &= - \sum_{y<x} \sum_{r \in \mathscr{B}_0(y,x)} (w_1 \widehat{e_{y,x}r} \otimes w_{1,2} \widehat{u}_{0,y}) \hat{u}_{0,x}(e_{0,y}v_{0,y}e_{y,x}r);\\
 &= - \sum_{y<x} \sum_{r \in \mathscr{B}_0(y,x)} (w_1 \widehat{e_{y,x}r} \otimes w_{1,2} \widehat{u}_{0,y}) \hat{u}_{0,x}(e_{0,x}v_{0,y}r);\\
 &= - \sum_{y<x} \sum_{r \in \mathscr{B}_0(y,x)} (w_1 \widehat{e_{y,x}r} \otimes w_{1,2} \widehat{u}_{0,y}) \hat{u}_{0,x}(e_{0,x}\chi_{y,x}(r)v_{0,x});\\
 &= - \sum_{y<x} \sum_{r \in \mathscr{B}_0(y,x)} (w_1 \widehat{e_{y,x}r} \otimes w_{1,2} \widehat{u}_{0,y}) \hat{u}_{0,x}(u_{0,x}\chi_{y,x}(r));\\
  &= - \sum_{y<x} \sum_{r \in \mathscr{B}_0(y,x)} (w_1 \widehat{e_{y,x}r} \otimes w_{1,2} \widehat{u}_{0,y}) \chi_{y,x}(r).
\end{align*}

For $\hat{z}\in e_x\,^*Je_y$, we have $f^1(\hat{z})\in \mathrm{Hom}_{K}(M_x,N_y)$, so it can be consider as $f^1: M_1\otimes_{\hat{e}S} \hat{e}\,^*J\hat{e} \rightarrow N_1 : m\otimes \hat{z} \mapsto f^1(\hat{z})(m)$, for all $m\in M_x$. This determines $\mathcal{R}_y$-morphisms $f^1_{x,y}:M_x \rightarrow N_y \otimes_{\mathcal{R}_y} e_y J e_x$ (see Remark \ref{morfismoGeneral}) such that, if $f^1_{x,y}(m)=\sum _{s=1}^{d_y}n_s\otimes \gamma_s$, for some $n_s \in N_y, \gamma_s \in e_y J e_x$, then $f^1(\hat{z})(m)=\sum _{s=1}^{d_y}n_s\hat{z}(\gamma_s)$. Therefore,
\begin{align*}
f^1(d_e(w_{1,2} \hat{u}_{0,x})) (m_i^x) &= - \sum_{y<x} \sum_{r \in \mathscr{B}_0(y,x)} f^1[(w_1 \widehat{e_{y,x}r} \otimes w_{1,2} \widehat{u}_{0,y}) \chi_{y,x}(r)] (m_i^x)\\
 &= - \sum_{y<x} \sum_{r \in \mathscr{B}_0(y,x)}  f^1(w_1 \widehat{e_{y,x}r}) (m_i^x) (w_{1,2} \widehat{u}_{0,y}) \chi_{y,x}(r);\\
 &= - \sum_{y<x} \sum_{t=1}^{d_y}  \sum_{r \in \mathscr{B}_0(y,x)} (n_t^y \widehat{e_{y,x}r}(\tau_{t,i}^{y,x}))  w_{1,2} \widehat{u}_{0,y} \chi_{y,x}(r);
\end{align*}
for some $\tau_{t,i}^{y,x} \in e_y J e_x$.

We have $\widehat{e_{y,x}r}(\tau_{t,i}^{y,x})\in e_yS= e_{y,y}\mathcal{R}_y$, and every $\tau_{t,i}^{y,x}$ has the form $\tau_{t,i}^{y,x}=e_{y,x}r_{t,i}^{y,x}$, with $r_{t,i}^{y,x}\in \mathcal{R}_{y,x}$, then $\widehat{e_{y,x}r}(\tau_{t,i}^{y,x}) = e_{y,y} \hat{r}(r_{t,i}^{y,x})$, where $\hat{r}(r_{t,i}^{y,x}) \in \mathcal{R}_y$.

In this way,
\begin{align*}
f^1(d_e(w_{1,2} \hat{u}_{0,x})) (m_i^x) &= - \sum_{y<x} \sum_{t=1}^{d_y} \sum_{r \in \mathscr{B}_0(y,x)} n_t^y (e_{y,y} \hat{r}(r_{t,i}^{y,x})) w_{1,2} \widehat{u}_{0,y} \chi_{y,x}(r);\\
 &= - \sum_{y<x} \sum_{t=1}^{d_y} \sum_{r \in \mathscr{B}_0(y,x)} n_t^y w_{1,2}  \hat{r}(r_{t,i}^{y,x}) \widehat{u}_{0,y} \chi_{y,x}(r);\\
 &= - \sum_{y<x} \sum_{t=1}^{d_y} \sum_{r \in \mathscr{B}_0(y,x)} n_t^y w_{1,2} \widehat{u}_{0,y} \phi_y(\hat{r}(r_{t,i}^{y,x})) \chi_{y,x}(r);\\
 &= - \sum_{y<x} \sum_{t=1}^{d_y} \sum_{r \in \mathscr{B}_0(y,x)} n_t^y w_{1,2} \widehat{u}_{0,y}  \chi_{y,x}(\hat{r}(r_{t,i}^{y,x})r);\\
 &= - \sum_{y<x} \sum_{t=1}^{d_y} n_t^y w_{1,2} \widehat{u}_{0,y}  \chi_{y,x} \left( \sum_{r \in \mathscr{B}_0(y,x)} \hat{r}(r_{t,i}^{y,x})r \right);\\
 &= - \sum_{y<x} \sum_{t=1}^{d_y} n_t^y w_{1,2} \widehat{u}_{0,y}  \chi_{y,x}(r_{t,i}^{y,x});\\
 &= - \sum_{y<x} \sum_{t=1}^{d_y} \sum_{s=1}^{d_0} n^0_s N_{s,t}^y \chi_{y,x}(r_{t,i}^{y,x});
\end{align*}
with $N_{j,t}^y \in \mathcal{R}_0$.

Then, for every $x\in \mathscr{P}^\text{\large \Fontlukas m}, m_i^x \in \mathscr{B}_x(M)$,
\[\sum_{t=1}^{d_x} \sum_{s=1}^{d_0} n_s^0 N_{s,t}^x \phi _{x}(h_{t,i}^x) = \sum_{j=1}^{d_0} \sum_{s=1}^{d_0} n_s^0 \mu_{s,j} M_{j,i}^x - \sum_{y<x} \sum_{t=1}^{d_y} \sum_{s=1}^{d_0} n^0_s N_{s,t}^y \chi_{y,x}(r_{t,i}^{y,x})\]

Writing the matrices:
\[\mathscr{M}(N)_x \phi _{x}(h_{t,i}^x) = (\mu_{s,j}) \mathscr{M}(M)_x  - \sum_{y<x}  \mathscr{M}(N)_y \chi_{y,x}(r_{t,i}^{y,x}).\]

We denote
$$T_0^{-1}= (\mu_{s,j}),\hspace{1cm} T_x=(h_{t,i}^x),\hspace{1cm} T_{y,x}=(r_{t,i}^{y,x}),$$
to obtain the desired equation:
\[\mathscr{M}(M)_x= T_0 \mathscr{M}(N)_x \phi_{x}(T_x) + \sum_{y<x} T_0 \mathscr{M}(N)_y \chi_{y,x}(T_{yx}).\]

Conversely, if we have suitable matrices satisfying the previous equation, the dimension of $M$ and $N$ are equal. We define some morphisms $f^{0}_{0}:M_{0}\rightarrow N_{0}$ by using the matrix $L_{0}$ and $f^{0}_{x}:M_{x}\rightarrow N_{x}$, with that matrix $T_{x}$, for all $x\in \mathscr{P}^\text{\large \Fontlukas m}$. They induce an $eSe$-module isomorphism $f^0:M\rightarrow N$, because matrices are invertible.

By using each matrix $T_{x,y}$, we define the morphism $f^1_{x,y}:M_x \rightarrow N_y \otimes_{\mathcal{R}_y} e_y J e_x$ which determine $f^1: M_1\otimes_{\hat{e}S} \hat{e}\,^*J\hat{e} \rightarrow N_1$.

We have the following matrix equation
\[\mathscr{M}(N)_x \phi _{x}(T_x) = T_0^{-1} \mathscr{M}(M)_x  - \sum_{y<x}  \mathscr{M}(N)_y \chi_{y,x}(T_{y,x});\]
then, for all $m_i^x\in \mathscr{B}_x(M)$
\[f_x^0(m_i^x) w_{1,2} \hat{u}_{0,x} = f^0_0 (m_i^x w_{1,2} \hat{u}_{0,x}) + f^1(d_e(w_{1,2} \hat{u}_{0,x})) (m_i^x);\]
where $\hat{u}_{0,x}$ is a generator of $e_x\,^*Je_0$ over $e_0S$.

Hence, the pair $f=(f^{0},f^{1})$ is a morphism from $M$ to $N$ in $\mathrm{Mod}\, \mathcal{D}^e$. Moreover, it is an isomorphism, because $f^{0}$ is an isomorphism. The theorem is proved.
\end{proof}

We will show next, an application for this result.


\subsubsection{Matrix problem for corepresentations of $p$-equipped posets}

To a $p$-equipped poset $\mathscr{P}_0$, we associate an algebraically equipped poset $(\mathscr{P}_{0},\textsf{G}, \mathcal{Q}_{0})$ as in Section \ref{pequipadosSonALgCorep}. We construct $(\mathscr{P},\textsf{G}, \mathcal{Q})$, where $\mathscr{P}=\underline{\mathscr{P}_0}^\text{\large \Fontlukas m}$ is obtained by adding to $\mathscr{P}_0$ a maximal point $\text{\large \Fontlukas m}$ and a minimal point $0$, and $\mathcal{Q}=\underline{\mathcal{Q}_0}^\text{\large \Fontlukas m}$ is obtained from $\mathcal{Q}_0$, as in (\ref{sisMulti1gor}). We know that $\mathcal{Q}$ is an admissible system and $\Lambda=\Lambda(\mathcal{Q})$ is 1-Gorenstein.

Due to the construction, $\mathcal{Q}_{0,x}=\textsf{G}$, for every $x\in \mathscr{P}$, so clearly, $\mathcal{Q}_{0,x}$ is one-dimensional over $\mathcal{Q}_0$, for all $x>0$. We can apply to $(\mathscr{P},\textsf{G}, \mathcal{Q})$ the results of this section.

Any $\mathcal{D}^{e}$-module $M$ with local base, has a matrix representation $\mathscr{M}(M)$, with coefficients in $\mathcal{Q}_0=\textsf{G}$. 

We have $\mathcal{Q}_{x,y}\subseteq \mathcal{Q}_0$, for all $x,y\in \mathscr{P}$, then $\chi_{x,y}$ and $\phi_x$ are the identity transformation. 

By Theorem \ref{MatricesGeneralCoreps}, we have the following result.

\begin{corollary}\label{MatricesCoreps}
Let $M$ and $N$ be two objects of $\mathrm{mod}\, \mathcal{D}^e$, with local basis $\mathscr{B}(M)$ and $\mathscr{B}(N)$, respectively. Then $M\cong N$ if and only if there exist:
\begin{enumerate}[·]
\item a non-singular square matrix $T_0\in \mathscr{M}_{d_0} (\textsf{G})$;
\item for $x>0$ in $\mathscr{P}$, a non-singular square matrix $T_x\in \mathscr{M}_{d_x} (\textsf{G})$ if $x$ is strong, or $T_x\in \mathscr{M}_{d_x} (\textsf{F})$ if $x$ is weak;
\item for each $0< y <^\ell x$ in $\mathscr{P}$ a matrix $T_{yx}$ over $\textsf{F} \langle 1, \xi, \xi^2, \ldots, \xi^{\ell-1} \rangle$;
\end{enumerate}
such that, for every $x\in \mathscr{P}$, $x>0$,
\[\mathscr{M}(M)_x= T_0 \mathscr{M}(N)_x T_x + \sum_{y<x} T_0 \mathscr{M}(N)_y T_{yx}.\]
\end{corollary}

The category $\Cor \mathscr{P}_0$ is equivalent to a subcategory of $\mathrm{Mod}\, \mathcal{D}^e$, this equivalence induces an equivalence between $\mathrm{corep}\,\mathscr{P}_0$, the full subcategory of $\Cor \mathscr{P}_0$ whose objects are of the form $(L, L_i:i\in \mathscr{P}_0)$ with $L$ a finite dimensional \textsf{G}-vector space, and a full subcategory of $\mathrm{mod}\, \mathcal{D}^e$. So the previous result determines a matrix problem for corepresentations of $p$-equipped posets.

\subsection{Matrices over $\mathcal{R}_x$, for all $x \in \underline{\mathscr{P}}^\text{\large \Fontlukas m}$}

Let $f=(f^{0},f^{1}):M\rightarrow N$ be a morphism in $\mathrm{mod}\, \mathcal{D}^e$. Recall that $f^{1}\in \mathrm{Hom}_{\hat{e}S-\hat{e}S}(T_{eRe}(eWe)_{1},\mathrm{Hom}_{K}(M,N))$ is determined by its restriction to $w_{1}\hat{e}\,^*J\hat{e}$, and, in turn, this morphism is completely determined by $\underline{f}^{1}:M_1\otimes \hat{e}\,^{*}J\hat{e}\rightarrow N_{1}$ where $\underline{f}^{1}(m\otimes \gamma )=f^{1}(w_{1}\gamma )(m)$, para $m\in M_{1}$, $\gamma \in \hat{e}\,*J\hat{e}$. 

We denote by $\underline{\varsigma }:e_{0}J\hat{e} \otimes _{\hat{e}S}\hat{e}J\hat{e} \rightarrow e_{0}J\hat{e}$ the morphism, such that $\underline{\varsigma }(a\otimes b)=ab$ for all $a \in e_{0}J\hat{e}, b\in \hat{e}J\hat{e}$.

The $\mathcal{D}^{e}$-module structure of $M$ is given by a morphism
$$h_{M}:M_{1}\otimes \hat{e}\,^*Je_{0}\rightarrow M_{0}: m\otimes \gamma \mapsto mw_{1,2}\gamma,$$ 
for $m\in M_{1}$ y $\gamma \in \hat{e}\,*Je_{0}$. 

\begin{proposition} 
A pair of morphisms
$$(f^{0},f^{1})\in \mathrm{Hom}_{S}(M,N)\times \mathrm{Hom}_{\hat{e}S-\hat{e}S}(T_{eRe}(eWe)_{1},\mathrm{Hom}_{K}(M,N))$$ 
is a morphism from $M$ to $N$ in $\mathrm{Mod}\, \mathcal{D}^e$ if and only if 
$$\tilde{h}_{N}f^{0}_{1}=(f^{0}_{0}\otimes id )\tilde{h}_{M}-(id\otimes \underline{\varsigma })(\tilde{h}_{N}\otimes id )\tilde{\underline{f}^{1}},$$
where $f^{0}_{0}:M_{0}\rightarrow N_{0}$ and $f^{0}_{1}:M_{1}\rightarrow N_{1}$ are morphisms induced by $f^{0}$, and $\tilde{h}_{N}, \tilde{h}_{M}, \tilde{\underline{f}^{1}}$ correspond to $h_N, h_M, \underline{f}^{1}$, respectively, following Remark \ref{morfismoGeneral}.
\end{proposition}

\begin{proof}
Consider some $\mathcal{R}_{0}$-space morphisms 
$$\sigma _{0},\sigma _{1},\sigma _{2}:M_{1}\otimes \hat{e}\,*Je_{0}\rightarrow N_{0},$$ 
defined as follows, for $m\in M_{1},\lambda \in \hat{e}\,^{*}J e_{0}$,
$$\sigma _{0}(m\otimes \lambda )=f^{0}_{1}(m)w_{1,2}\lambda ,$$
 $$ \sigma _{1}(m\otimes \lambda )=f_{0}^{0}(mw_{1,2}\lambda ),$$ 
 $$ \sigma _{2}(m\otimes \lambda )=f^{1}(d(w_{1,2}\lambda ))(m).$$
 
The pair $(f^{0},f^{1})$ is a morphism in $\mathrm{Mod}\, \mathcal{D}^e$ if and only if $\sigma _{0}=\sigma _{1}+\sigma _{2}$ and, by Remark \ref{morfismoGeneral}, this happens if and only if $\tilde{\sigma }_{0}=\tilde{\sigma }_{1}+\tilde{\sigma }_{2}$. 
 
We have $\sigma _{0}=h_{N}(f_{1}^{0}\otimes id )$ and $\sigma _{1}=f_{0}^{0}h_{M}$, therefore,
 $$\tilde{\sigma }_{0}=\Xi _{M_{1},N_{0}}(h_{N}(f_{1}^{0}\otimes id ))=\Xi _{N_{1},N_{0}}(h_{N})f_{1}^{0}=\tilde{h}_{N}f^{0}_{1};$$
 $$ \tilde{\sigma }_{1}=\Xi _{M_{1},N_{0}}(f_{0}^{0}h_{M})=(f_{0}^{0}\otimes id )\Xi _{M_{1},M_{0}}(h_{M})=(f^{0}_{0}\otimes id )\tilde{h}_{M}.$$
 
Let $\mathcal{B}_{0}$ be a basis for $e_0J$, and $\mathcal{B}_{1}$ be a basis for $\hat{e}J$, as left $eS$-modules. If $p\in \mathcal{B}_0 \cup \mathcal{B}_1$, its corresponding element in the dual space $\hat{p}\in e\,^*Je$, is such that $\hat{p}(w)=0$, if $w\in \mathcal{B}_0 \cup \mathcal{B}_1$ with $w\neq p$, and $\hat{p}(p)=e_0$, if $p\in \mathcal{B}_0$ or $\hat{p}(p)=\hat{e}$, if $p\in \mathcal{B}_1$.
 
We have
$$\tilde{\sigma }_{2}(m)=\sum _{p\in \mathcal{B}_{0}}f^{1}(d(w_{1,2}\hat{p}))(m)\otimes p$$
$$=-\sum _{\substack{p, \nu \in \mathcal{B}_{0} \\ \mu \in \mathcal{B}_{1}}}f^{1}(w_{1}\hat{\mu}\otimes w_{1,2}\hat{\nu }\hat{p}(\nu \mu ))(m)\otimes p$$
$$=-\sum _{\substack{p, \nu \in \mathcal{B}_{0} \\ \mu \in \mathcal{B}_{1}}}f^{1}(w_{1}\hat{\mu})(m)\otimes w_{1,2}\hat{\nu }\hat{p}(\nu \mu ))\otimes p$$
$$=-\sum _{\substack{\nu \in \mathcal{B}_{0} \\ \mu \in \mathcal{B}_{1}}}f^{1}(w_{1}\hat{\mu})(m)w_{1,2}\hat{\nu }\otimes \nu \mu $$
    $$=-\sum _{\substack{\nu \in \mathcal{B}_{0} \\ \mu \in \mathcal{B}_{1}}}h_{N}(f^{1}(w_{1}\hat{\mu})(m)\otimes\hat{\nu }) \otimes \nu \mu $$
 $$=-(id\otimes \underline{\varsigma })\sum _{\substack{\nu \in \mathcal{B}_{0} \\ \mu \in \mathcal{B}_{1}}}h_{N}(f^{1}(w_{1}\hat{\mu})(m)\otimes\hat{\nu }) \otimes \nu \otimes  \mu$$
      $$=-(id\otimes \underline{\varsigma })\sum _{\mu \in \mathcal{B}_{1}}\tilde{h}_{N}(f^{1}(w_{1}\hat{\mu})(m)) \otimes  \mu $$
       $$=-(id\otimes \underline{\varsigma })(\tilde{h}_{N}\otimes id) \left( \sum _{\mu \in \mathcal{B}_{1}}f^{1}(w_{1}\hat{\mu})(m) \otimes  \mu \right)$$
 $$=-(id \otimes \underline{\varsigma })(\tilde{h}_{N}\otimes id )\tilde{\underline{f}}^{1}(m).$$
 
It follows that $(f^{0},f^{1})$ is a morphism in $\mathrm{Mod}\, \mathcal{D}^e$ if and only if
$$\tilde{h}_{N}f^{0}_{1}=(f^{0}_{0}\otimes id )\tilde{h}_{M}-(id\otimes \underline{\varsigma })(\tilde{h}_{N}\otimes id )\tilde{\underline{f}^{1}},$$
as we wanted to prove. 
\end{proof}
 
Suppose that $\mathcal{R}_{0,x}$ is one-dimensional over $\mathcal{R}_{x}$, for all $x\in \mathscr{P}^\text{\large \Fontlukas m}$. Let $v_{0,x}$ be a generator for $\mathcal{R}_{0,x}$, as $\mathcal{R}_{x}$-space.
 
For every $a\in \mathcal{R}_{0}$, there exists some $\rho_{x}(a) \in \mathcal{R}_{x}$ such that $av_{0,x}=v_{0,x}\rho_{x}(a)$, then there is a field morphism $\rho _{x}:\mathcal{R}_{0} \rightarrow \mathcal{R}_{x} : a\mapsto \rho_{x}(a)$. 

Besides,  $v_{0,y}r \in \mathcal{R}_{0,x}$ for every $r\in \mathcal{R}_{y,x}$, so there is a $\mathcal{R}_x$-linear function $\chi _{y,x}: \mathcal{R}_{y,x} \rightarrow \mathcal{R}_{x}$ such that $v_{0,y}r=v_{0,x}\chi _{y,x}(r)$, for all $r\in \mathcal{R}_{y,x}$. 

If $\mathcal{B}(M)$ is a finite local basis for a $\mathcal{D}^{e}$-module $M$, we associate to it a matrix $\mathscr{M}(M)$, divided into stripes $\mathscr{M}(M)_{x}$, for every $x\in \mathscr{P}^\text{\large \Fontlukas m}$, where each $\mathscr{M}(M)_{x}$ has size $d_{0}\times d_{x}$, and coefficients in $\mathcal{R}_{x}$, which we denote $M^{x}_{j,i}$, such that 
$$\tilde{h}_{M}(m^{x}_{i})=\sum _{j=1}^{d_0}m_{j}^{0}\otimes u_{0,x}M^{x}_{j,i}, $$
for $u_{0,x}=e_{0,x}v_{0,x}$,  $m^{x}_{i}\in \mathcal{B}(M)\cap M_{x}$ and $m_{j}^{0}\in \mathcal{B}(M)\cap M_{0}$.

With this notation we have the following result.

\begin{theorem}\label{RepsMatriciales}
Two objects $M,N \in \mathrm{mod}\, \mathcal{D}^e$, with local basis $\mathscr{B}(M)$ and $\mathscr{B}(N)$, respectively, are isomorphic  if and only if there are non-singular square matrices $T_{0}\in \mathscr{M}_{d_0} (\mathcal{R}_0)$, $T_x\in \mathscr{M}_{d_x} (\mathcal{R}_x)$, for all $x\in \mathscr{P}^\text{\large \Fontlukas m}$, and for each $y < x$ in $\mathscr{P}^\text{\large \Fontlukas m}$, a matrix $T_{yx} \in \mathscr{M}_{d_x\times d_y} (\mathcal{R}_{y,x})$, such that
\[\mathscr{M}(M)_{x}=\rho _{x}(T_{0})\mathscr{M}(N)_{x} T_x +\sum _{y<x}\rho _{x}(T_{0})\chi _{y,x}(\mathscr{M}(N)_{y}T_{y,x}).\]
\end{theorem}

\begin{proof}
Consider an arbitrary pair  
$$(f^{0},f^{1})\in \mathrm{Hom}_{S}(M,N)\times \mathrm{Hom}_{e^{\prime }S-e^{\prime }S}(T_{eRe}(eWe)_{1},\mathrm{Hom}_{K}(M,N)),$$
the morphism $f^{0}$ induces an $\mathcal{R}_{0}$-linear morphism $f_{0}^{0}:M_{0}\rightarrow N_{0}$, and an $\mathcal{R}_{x}$-linear morphism, $f^{0}_{x}:M_{x}\rightarrow N_{x}$,  for each $x\in \mathscr{P}^\text{\large \Fontlukas m}$. 

Let $L_{0}=(c^{0}_{j,i})$ and $T_{x}=(r^{x}_{j,i})$ be the matrices of $f^{0}_{0}$ and $f^{0}_{x}$, with respect to $\mathcal{B}(M)$ and $\mathcal{B}(N)$. 

The vectors $\underline{d}(M)=(d^M_x : x \in \underline{\mathscr{P}}^\text{\large \Fontlukas m})$, $\underline{d}(N)=(d^N_x : x \in \underline{\mathscr{P}}^\text{\large \Fontlukas m})$ are the dimension of $M$ and $N$, respectively.

For $m^{x}_{i}\in \mathcal{B}(M)\cap M_{x}$ we have
$$\tilde{\underline{f}}^{1}(m_{i}^{x})=\sum _{y<x}\sum _{j=1}^{d_y^N}n_{j}^{y}\otimes e_{y,x}t_{j,i}^{y,x}$$
with $t_{j,i}^{y,x}\in \mathcal{R}_{y,x}$.

Hence, $(f^{0},f^{1})$ determine some matrices $L_{0}$, $T_{x}$ and $T_{y,x}=(t_{j,i}^{y,x})$, as we wanted.

Conversely, if there exist some matrices $L_{0}$, $T_{x}$ for all $x\in \mathscr{P}^\text{\large \Fontlukas m}$, and $T_{yx}$ for each $y<x$, as in the hypothesis, they define a pair $(f^{0},f^{1})$. Moreover $\underline{d}(M)=\underline{d}(N)$.

By the previous proposition, the pair $(f^{0},f^{1})$ is a morphism from $M$ to $N$, if and only if, for every $m^{x}_{i}\in \mathcal{B}(M)\cap M_{x}$,
$$\tilde{h}_{N}f^{0}_{x} (m_{i}^{x}) =(f^{0}_{0}\otimes id )\tilde{h}_{M} (m_{i}^{x}) - (id\otimes \underline{\varsigma })(\tilde{h}_{N}\otimes id )\tilde{\underline{f}^{1}} (m_{i}^{x}).$$

We have,
\begin{align*}
\tilde{h}_{N}f_{x}^{0}(m_{i}^{x}) &= \sum _{j=1}^{d^N_x}\tilde{h}_{N}(n_j^{x}r_{j,i}^{x})=\sum _{j=1}^{d^N_y}\tilde{h}_{N}(n_j^{x})r_{j,i}^{x}\\
 &= \sum _{j=1}^{d^N_x}\sum _{s=1}^{d^N_0}n_{s}^{0}\otimes u_{0,x}N^{x}_{s,j}r_{j,i}^{x};
\end{align*}
besides,
\begin{align*}
(f^{0}_{0}\otimes id)\tilde{h}_{M}(m_{i}^{x}) &= \sum _{j=1}^{d^M_0}(f^{0}_{0}\otimes id)(m_{j}^{0}\otimes u_{0,x}M_{j,i}^{x})\\
 &= \sum _{j=1}^{d^M_0}\sum _{s=1}^{d^N_0}n_{s}^{0}c_{s,j}\otimes u_{0,x}M^{x}_{j,i}\\
 &= \sum _{j=1}^{d^M_0}\sum _{s=1}^{d^N_0}n_{s}^{0}\otimes e_{0,x}c_{s,j} v_{0,x}M^{x}_{j,i}\\
 &= \sum _{j=1}^{d^M_0}\sum _{s=1}^{d^N_0} n_{s}^{0}\otimes e_{0,x}v_{0,x}\rho _{x}(c_{s,j}) M^{x}_{j,i}\\
 &= \sum _{j=1}^{d^M_0}\sum _{s=1}^{d^N_0} n_{s}^{0}\otimes u_{0,x}\rho _{x}(c_{s,j}) M^{x}_{j,i}.
\end{align*}

To write the whole equation in terms of the basis of $N_0$,
\begin{align*}
(id\otimes \underline{\varsigma })(\tilde{h}_{N}\otimes id )\tilde{\underline{f}^{1}}(m^{x}_{i}) &= \sum _{y<x}\sum _{j=1}^{d^N_y} (id\otimes \underline{\varsigma })(\tilde{h}_{M}\otimes id ) \left( n_{j}^{y}\otimes e_{y,x}t_{j,i}^{y,x} \right)\\
 &= \sum _{y<x}\sum _{j=1}^{d^N_y} \sum _{s=1}^{d^N_0}  (id\otimes \underline{\varsigma }) \left( n^{0}_{s}\otimes e_{0,y}v_{0,y}N^{y}_{s,j} \otimes e_{y,x}t_{j,i}^{y,x} \right)\\
 &= \sum _{y<x}\sum _{j=1}^{d^N_y} \sum _{s=1}^{d^N_0} n^{0}_{s}\otimes e_{0,x}v_{0,y}N_{s,j}^{y}t^{y,x}_{j,i}\\
 &= \sum _{y<x}\sum _{j=1}^{d^N_y} \sum _{s=1}^{d^N_0} n^{0}_{s}\otimes e_{0,x}v_{0,x}\chi _{y,x}(N_{s,j}^{y}t^{y,x}_{j,i})\\
  &= \sum _{y<x}\sum _{j=1}^{d^N_y} \sum _{s=1}^{d^N_0} n^{0}_{s}\otimes u_{0,x}\chi _{y,x}(N_{s,j}^{y}t^{y,x}_{j,i}).
\end{align*}

The pair $(f^{0},f^{1})$ given by the matrices $L_{0}$, $T_{x}$, $T_{y,x}$ determines a morphism from $M$ to $N$ if and only if:
$$\mathscr{M}(N)_{x}T_{x}=\rho _{x}(L_{0})\mathscr{M}_{x}-\sum _{y<x}\chi _{y,x}(\mathscr{M}(N)_{y}T_{y,x});$$
or equivalently:
\begin{equation}\label{ParEsMorfismo}
\rho _{x}(L_{0})\mathscr{M}_{x} = \mathscr{M}(N)_{x}T_{x} + \sum _{y<x}\chi _{y,x}(\mathscr{M}(N)_{y}T_{y,x}).
\end{equation}

Moreover, $(f^{0},f^{1}):M\rightarrow N$ is an isomorphism if and only if (\ref{ParEsMorfismo}) holds, with non-singular square matrices $L_{0}$ and $T_{x}$. By denoting $T_0=L_0^{-1}$, this is equivalent to
$$\mathscr{M}(M)_{x}=\rho _{x}(T_{0})\mathscr{M}(N)_{x} T_x +\sum _{y<x}\rho _{x}(T_{0})\chi _{y,x}(\mathscr{M}(N)_{y}T_{y,x}),$$
and the theorem is proved.
\end{proof}

We will apply this theorem to obtain matrix problems to classify representations of generalized equipped posets and $p$-equipped posets.

\subsubsection{Matrix problem for representations of generalized equipped, and $p$-equipped posets}

Consider an algebraically equipped poset $(\mathscr{P}_{0},A, \mathcal{T}_{0})$ as in Section \ref{genSonALg} or \ref{pequipadosSonALgRep}, we construct $(\mathscr{P},A, \mathcal{T})$, where $\mathscr{P}=\underline{\mathscr{P}_0}^\text{\large \Fontlukas m}$ is obtained by adding to $\mathscr{P}_0$, a  maximal point $\text{\large \Fontlukas m}$ and a minimal point $0$, and $\mathcal{T}=\underline{\mathcal{T}_0}^\text{\large \Fontlukas m}$ is obtained from $\mathcal{T}_0$, as in (\ref{sisMulti1gor}). 

Recall that $A=(\End_{\textsf{K}}\textsf{L})^{op}$, where $\textsf{L}$ is a normal extension of degree $n$ over $\textsf{K}$ (it could be purely inseparable or a Galois extension). Clearly, $\textsf{L}=\textsf{K}(\xi)$, for some primitive element $\xi$.

For each $x\in \mathscr{P}_{0}$, we have $\mathcal{T}_{x}\cong (\End_{\textsf{K}(x)}\textsf{L})^{op}$, for some field $\textsf{K} \subseteq \textsf{K}(x) \subseteq \textsf{L}$. Hence, if $n(x)$ is the degree of $\textsf{L}$ over $\textsf{K}(x)$, a $\textsf{K}(x)$-basis for $\textsf{L}$ is $\{1,\xi,\xi^2,\ldots,\xi^{n(x)-1}\}$.
 
Therefore $\textsf{L} = \textsf{K}(x)\oplus \xi \textsf{K}(x)+...+\xi ^{n(x)-1}\textsf{K}(x)$. For all $i\in \{0,1,\ldots,n(x)-1\}$, there are projections $\pi_{x}^{i}: \textsf{L} \rightarrow \xi ^{i}\textsf{K}(x)$, and inclusions $\textsf{\emph{i}}^{i}_{x}:\xi ^{i}\textsf{K}(x)\rightarrow \textsf{L}$. In particular, we have idempotents $\varepsilon_{x}^{i}=\textsf{\emph{i}}_{x}^{i}\pi _{x}^{i}$, and the identity $id_{\textsf{L}}=\sum _{i=0}^{n(x)-1}\varepsilon_{x}^{i}$ is a sum of orthogonal primitive idempotents mutually isomorphic. 

As in Section \ref{Morita}, we construct an algebraically equipped poset $(\mathscr{P},A, \mathcal{R})$, with the admissible system $\mathcal{R} = \{\mathcal{R}_{x,y} = \varepsilon_{x}^{0}\mathcal{T}_{x,y}\varepsilon_{y}^{0}\}_{\substack{x\leq y \\ x,y\in \mathscr{P}}}$, in such a way that the algebras $\Lambda(\mathcal{T})$ and $\Lambda(\mathcal{R})$ are Morita equivalent. Notice that $\mathcal{R}_x \cong \textsf{K}(x)$.

For any $f\in A$, the image of $\xi^i$, for all  $i\in \{0,1,\dots,n-1\}$, has the form $f(\xi^i)=\sum_{j=0}^{n-1}a_{i+1,j+1}\xi^j$, with some $a_{i+1,j+1}\in \textsf{K}$. We have an isomorphism
\[\Theta: A \rightarrow \mathscr{M}_n(\textsf{K}): f\mapsto (a_{i+1,j+1}).\] 

If $X$ is a subset of $A$, we denote by $\underline{X}$ its image in $\mathscr{M}_n(\textsf{K})$ under the isomorphism $\Theta$.

There are two additional algebraically equipped posets $(\mathscr{P},\mathscr{M}_n(\textsf{K}),\underline{\mathcal{T}})$, where $\underline{\mathcal{T}}=\{ \underline{\mathcal{T}}_{x,y}\}_{\substack{x\leq y \\ x,y\in \mathscr{P}}}$, and $(\mathscr{P}, \mathscr{M}_n(\textsf{K}),\underline{\mathcal{R}})$, where $\underline{\mathcal{R}}=\{ \underline{\mathcal{R}}_{x,y}\}_{\substack{x\leq y \\ x,y\in \mathscr{P}}}$. The algebras $\Lambda(\underline{\mathcal{T}})$ and $\Lambda(\underline{\mathcal{R}})$ are Morita equivalent.

We will use the following notation:
\[\varepsilon_{x} = \Theta(\varepsilon_{x}^{0}) \text{ for all } x\in \mathscr{P};\]
\[m_a = \Theta(\mu_a) \text{ for all } a \in \textsf{L};\]
\[\Delta _{x} = \textsf{K} \langle m_a \in \mathscr{M}_n(\textsf{K}) | a\in \textsf{K}(x) \rangle;\]
the definition of $\mu_a$ is in (\ref{defmu}).

The admissible system $\mathcal{R}$ has the following property.

\begin{lemma} 
For every $x\in \mathscr{P}$, the $\mathcal{R}_{0}-\mathcal{R}_{x}$-bimodule $\mathcal{R}_{0,x}$, is one-dimensional over $\mathcal{R}_{x}$.
\end{lemma}
 
\begin{proof}
We have that $\mathcal{R}_{0,x}$ is an $\mathcal{R}_{x}$-vector space, hence, its dimension over $\textsf{K}$ is a multiple of $\mathrm{dim}_{\textsf{K}}\mathcal{R}_{x}=\mathrm{dim}_{\textsf{K}}\textsf{K}(x)$. On the other hand, $\mathcal{R}_{0,x}$ is isomorphic to $\varepsilon_{0}M_{n}(\textsf{K})\varepsilon _{x}$, which is isomorphic, as $\textsf{K}$-vector space, to $\mathrm{Hom}_{\textsf{K}}(\textsf{K},\textsf{K}(x))$. This last space has dimension $\mathrm{dim}_{\textsf{K}}\textsf{K}(x)$, consequently $\mathrm{dim}_{\mathcal{R}_{x}}\mathcal{R}_{0,x}=1$, and the lemma is proved.
\end{proof}

For any $\mathcal{D}^{e}$-module $M$, with finite local basis $\mathscr{B}(M)$, we construct a matrix representation such that $\mathscr{M}(M)_{x} \in \mathscr{M}_{d_0,d_x}(\mathcal{R}_{x})$, for $x\in \mathscr{P}$, $x>0$.

Each $\mathcal{R}_{x}$ is isomorphic to $\textsf{K}(x)$, then, by using the isomorphism $\Theta$, we will construct matrix representations with coefficients in $\textsf{K}(x)$. To do this we will use the following result:

\begin{lemma} 
For every $x\in \mathscr{P}$, $x>0$, we have
$$\underline{\mathcal{R}}_x=\varepsilon_{x}\Delta _{x}\varepsilon_{x}.$$
\end{lemma}
 
\begin{proof}
Each $a\in \textsf{K}(x)$, determines a $\mu_a \in \mathcal{T}_{x}$, then $\varepsilon_{x} m_a \varepsilon_{x} \in \varepsilon_{x}\underline{\mathcal{T}}_{x}\varepsilon_{x} = \underline{\mathcal{R}}_x$. Hence, $\varepsilon_{x}\Delta _{x}\varepsilon_{x} \subseteq \underline{\mathcal{R}}_x$.

The elements in $\underline{\mathcal{R}}_x$ have the form $\Theta(\alpha)$, for $\alpha =\varepsilon_{x}^{0}\alpha \varepsilon_{x}^{0}$, such that $\alpha \in \End_{\textsf{K}(x)}\textsf{L}$. 

If $\alpha \neq 0$, notice that $\alpha (1) \in \textsf{K}(x)$ and $\alpha (a) = \alpha (1) a$, for all $a\in \textsf{K}(x)$. Denoting $c=\alpha (1)$, we have $\alpha = \varepsilon_{x}^{0} \mu_c \varepsilon_{x}^{0}$. Hence, $\Theta(\alpha) = \varepsilon_{x} m_c \varepsilon_{x}$, therefore $\underline{\mathcal{R}}_x \subseteq \varepsilon_{x}\Delta _{x}\varepsilon_{x}$, which finishes the proof.
\end{proof}

Every $r\in \underline{\mathcal{R}}_x=$ can be written as $r=\varepsilon_{x}m_{a}\varepsilon_{x}$, for some $a\in \textsf{K}(x)$, so we have an isomorphism:
$$\phi _{x}:\underline{\mathcal{R}}_x \rightarrow \textsf{K}(x) : \varepsilon_{x}m_{a}\varepsilon_{x} \mapsto a.$$

From the matrix $\mathscr{M}(M)_{x}$, with values in $\mathcal{R}_{x}$, we get the matrix $\underline{\mathscr{M}}(M)_{x}$ over $\underline{\mathcal{R}}_x$, and then, we obtain the matrix $\phi _{x}(\underline{\mathscr{M}}(M)_{x})$ with coefficients in $\textsf{K}(x)$. We denote:
$$\underline{M}_{x} = \phi _{x}(\underline{\mathscr{M}}(M)_{x}).$$

The matrix $\underline{M}$ divided into vertical stripes $\underline{M}_{x}$, for all $x\in \mathscr{P}$, $x>0$, is called matrix representation of $M$.

To calculate $\rho _{x}: \underline{\mathcal{R}}_{0} \rightarrow \underline{\mathcal{R}}_{x}$, we use the next lemma:

\begin{lemma} 
The idempotents satisfy
$$ \varepsilon_{0} \varepsilon_{x} = \varepsilon_{0},$$
for every $x\in \mathscr{P}$.
\end{lemma}
 
\begin{proof}
We have $\varepsilon_{0}=\Theta(\varepsilon_{0}^{0})$ and $\varepsilon_{x}=\Theta(\varepsilon_{x}^{0})$. To prove our lemma is enough to show that
 $\varepsilon^{0}_{x}\varepsilon ^{0}_{0}=\varepsilon ^{0}_{0}$.
 
Let $i_{x}:\textsf{K}\rightarrow \textsf{K}(x)$ be the inclusion, then,
 $$\varepsilon^{0}_{x}\varepsilon ^{0}_{0}=\textsf{\emph{i}}_{x}^{0}\pi _{x}^{0}\textsf{\emph{i}}_{0}^{0}\pi ^{0}_{0}.$$
 
Clearly, $\pi_{x}^{0}\textsf{\emph{i}}_{0}^{0}=i_{x}$ and $\textsf{\emph{i}}^{0}_{x}i_{x}=\textsf{\emph{i}}_{0}^{0}$, hence, 
 $$\textsf{\emph{i}}_{x}^{0}\pi _{x}^{0}\textsf{\emph{i}}_{0}^{0}\pi ^{0}_{0}=\textsf{\emph{i}}_{0}^{0}\pi ^{0}_{0}=\varepsilon _{0}^{0};$$
and our result is proved.
\end{proof}
 
In particular, $\varepsilon_{0}\in \underline{\mathcal{R}}_{0,x}$ for all $x\in \mathscr{P}$. We choose $v_{0,x}=\varepsilon_{0}$, to calculate $\rho _{x}: \underline{\mathcal{R}}_{0} \rightarrow \underline{\mathcal{R}}_{x}$.
For any $\varepsilon_{0}m_{a}\varepsilon_{0} \in \underline{\mathcal{R}}_{0}$, with $a\in \textsf{K}(0) = \textsf{K}$, we have
 $$\varepsilon_{0}m_{a}\varepsilon_{0}v_{0,x}=\varepsilon_{0}m_{a}\varepsilon_{0}=\varepsilon_{0}\varepsilon_{0}m_{a}$$
 $$=\varepsilon_{0}\varepsilon_{x}m_{a}=\varepsilon_{0}\varepsilon_{x}\varepsilon_{x}m_{a}=\varepsilon_{0}\varepsilon_{x}m_{a}\varepsilon_{x}=v_{0,x}\varepsilon_{x}m_{a}\varepsilon_{x},$$
then:
\begin{equation}\label{rhoRepsGeneralizados}
\rho _{x}(\varepsilon_{0}m_{a}\varepsilon_{0})=\varepsilon_{x}m_{a}\varepsilon_{x}.
\end{equation}

For every $y<x$, let $\tau ^{y,x}_{1},...,\tau ^{y,x}_{l(y,x)}$ be an $\underline{\mathcal{R}}_{x}$-basis for $\underline{\mathcal{R}}_{y,x}.$ With the morphism  $\chi _{y,x}:\underline{\mathcal{R}}_{y,x} \rightarrow \underline{\mathcal{R}}_{x}$, each  $\tau ^{y,x}_{i}$ induces a $\textsf{K}$-vector space morphism $u^{y,x}_{i}: \textsf{K}_{y}\rightarrow \textsf{K}_{x}$, given by  
$$u^{y,x}_{i}(a)=\phi_{x}(\chi _{y,x}[(\phi_{y})^{-1}(a) \tau^{y,x}_{i})]),$$
for all $a\in K_{y}$.

With this notation we have the following result.

\begin{theorem}\label{RepsMatricesGenPequipados}
Let $M, N$ be objects in $\mathrm{mod}\, \mathcal{D}^e$, with local basis $\mathscr{B}(M)$ and $\mathscr{B}(N)$, respectively. Then $M\cong N$ if and only if there exist:
\begin{enumerate}[·]
\item a non-singular square matrix $L^{0}\in \mathscr{M}_{d_{0}}(\textsf{K})$;
\item for every $x \in \mathscr{P}$, $x>0$, a non-singular square matrix $L^{x}\in \mathscr{M}_{d_{x}}(\textsf{K}(x))$;
\item for each $0< y < x$ in $\mathscr{P}$ some matrices $L_{1}^{y,x},...,L_{l(y,x)}^{y,x}\in \mathscr{M}_{d_{y},d_{x}}(\textsf{K}(x))$;
\end{enumerate}
such that
\begin{equation}\label{ecuacionRepsMatricesGen}
\underline{M}_{x}=L^{0} \left( \underline{N}_{x}L^{x}+\sum _{y<x}\sum _{i=1}^{l(y,x)}u_{i}^{y,x}(\underline{N}_{y})L_i^{y,x} \right).
\end{equation}
\end{theorem}

\begin{proof}
By Theorem \ref{RepsMatriciales}, and the relation between the algebraically equipped posets $(\mathscr{P},A, \mathcal{R})$ and $(\mathscr{P}, \mathscr{M}_n(\textsf{K}),\underline{\mathcal{R}})$, we know that $M\cong N$ if and only if there exist non-singular matrices $T_{0}$, of size $d_{0}\times d_{0}$ with values in $\underline{\mathcal{R}}_{0}$, for every $x\in \mathscr{P}$, $x>0$, a  matrix $T_{x}$, of size $d_{x}\times d_{x}$ over $\underline{\mathcal{R}}_{x}$, and matrices $T_{y,x}$, for each $y<x$ with coefficients in $\underline{\mathcal{R}}_{y,x}$ such that:
\begin{equation}\label{ecuacionDelTeorema}
\underline{\mathscr{M}}(M)_{x}=\rho _{x}(T_{0})\underline{\mathscr{M}}(N)_{x} T_x +\sum _{y<x}\rho _{x}(T_{0})\chi _{y,x}(\underline{\mathscr{M}}(N)_{y}T_{y,x}).
\end{equation}

Every matrix $T_{y,x}$ can be written as $T_{y,x}=\sum _{i=1}^{l(y,x)}\tau ^{y,x}_{i}T_{i}^{y,x}$ with $T^{y,x}_{i}\in \mathscr{M}_{d_{x},d_{y}}(\underline{\mathcal{R}}_{x})$.

It is clear that $\phi _{x}\rho _{x}=\phi_{0}$. Then, with notation $L^{0} = \phi _{0}(T_{0})$, $L^x = \phi _{x}(T_{x})$, $L^{x,y}_{i}=\phi _{x}(T^{y,x}_{i})$, we evaluate $\phi_{x}$ in \ref{ecuacionDelTeorema}, to get:
$$\underline{M}_{x}=L^{0}(\underline{N}_{x}L^{x}+\sum _{y<x}\sum _{i=0}^{l(y,x)}\phi _{x}\chi _{y,x}(\mathscr{M}(N)_{y}\tau _{i}^{y,x})L^{y,x}_i).$$

Besides, 
$$\phi _{x}\chi _{y,x}(\mathscr{M}(N)_{y}\tau _{i}^{y,x})=\phi _{x}\chi _{y,x}(\phi _{y}^{-1}(\underline{\mathscr{M}}(N)_{y}\tau _{i}^{y,x})=u_{i}^{y,x}(\underline{N}_{y}).$$

Converssely, if $(\ref{ecuacionRepsMatricesGen})$ holds, by denoting $T_{0}=\phi _{0}^{-1}(L^{0})$, $T_{x}=\phi _{x}^{-1}(L^{x})$ and
$T_{x,y}=\sum _{i=1}^{l(y,x)}\tau _{i}^{y,x}\phi _{x}^{-1}(L_{i}^{y,x})$, and evaluating $\phi _{x}^{-1}$ in $(\ref{ecuacionRepsMatricesGen})$, we obtain $(\ref{ecuacionDelTeorema})$, which is equivalent to $M\cong N$. The theorem is proved.
\end{proof}

Let $\varepsilon_{1,1}\in \mathscr{M}_n(\textsf{K})$ be the idempotent matrix having 1 at the place (1,1) and 0 otherwise, and $I$ be the identity matrix.

In particular, when $\mathscr{P}_0$ is a $p$-equipped poset, $n=p$, and for $x\in \mathscr{P}$
$$\varepsilon_{x} = 
\begin{cases} I, & \text{ if } x \text{ is weak,}\\
\varepsilon_{1,1}, & \text{ if } x \text{ is strong.} \end{cases}$$ 

We are going to use some notation from Section \ref{pequipadosSonALgRep}, and let us introduce the following:
\begin{center}
$\underline{\vartheta} = \Theta(\vartheta);$\\
$U = \varepsilon_{1,1} \mathscr{M}_p(\textsf{K});$\\
$W = \mathscr{M}_p(\textsf{K}) \varepsilon_{1,1};$\\
$\R : \textsf{L} \rightarrow \textsf{K} : a_0+a_1\xi + \cdots + a_{p-1} \xi^{p-1} \mapsto a_0$, with $a_0,\ldots,a_{p-1}\in \textsf{K}$.
\end{center}

For every $x,y \in \mathscr{P}$, we have:
\begin{enumerate}[(i)]
\item if $y \leq^p y \leq x\leq^p x$, then $\underline{\mathcal{R}}_{y,x} = \varepsilon _{1,1} \textsf{K} \varepsilon _{1,1}$;
\item if $y \leq^p y < x\leq^1 x$, then $\underline{\mathcal{R}}_{y,x} = U$;
\item if $y \leq^1 y < x\leq^p x$, then $\underline{\mathcal{R}}_{y,x} = W$;
\item if $y \leq^1 y \leq^\ell x\leq^1 x$, then $\underline{\mathcal{R}}_{y,x} = \underline{A}_\ell$.
\end{enumerate}

Notice that, for a strong point $x$, we are in the case (i), with $y=x$, then $\underline{\mathcal{R}}_{x} = \varepsilon _{1,1} \textsf{K} \varepsilon _{1,1}$, and when $x$ is weak, $\underline{\mathcal{R}}_{x} = \underline{A}_1$ (case (iv) with $y=x$, $\ell=1$). Obviously, if $x$ is strong $\textsf{K}(x)=\textsf{K}$, or $\textsf{K}(x)=\textsf{L}$ if $x$ is weak.

Now we can choose an $\underline{\mathcal{R}}_{x}$-basis for $\underline{\mathcal{R}}_{y,x}$ in each situation, and calculate the induced morphisms $u^{y,x}_{i}$:

When (i) or (ii) hold, $\mathrm{dim}_{\mathcal{R}_{x}}\mathcal{R}_{x,y}=1$, so we put $\tau _{1}^{y,x}=\varepsilon _{1,1}$. For $x$ strong (case (i)), $\tau _{1}^{y,x}$ induces $u^{y,x}_{1}: \textsf{K} \rightarrow \textsf{K}$ which is the identity. When $x$ is a weak point, $u^{y,x}_{1}: \textsf{K} \rightarrow \textsf{L}$ is the inclusion, because for all $a\in \textsf{K}$,
$$u_{1}^{y,x}(a)=\phi_{x}(\chi _{y,x}(\varepsilon _{1,1}m_{a}\varepsilon _{1,1}))=\phi_{x}(\chi _{y,x}(\varepsilon _{1,1}m_{a}))=\phi_{x}(m_{a})=a\in \textsf{L}.$$

In the case (iii), $\mathrm{dim}_{\mathcal{R}_{x}}\mathcal{R}_{x,y}=p$, then $\tau _{i}^{y,x}=m_{\xi^{i-1}}\varepsilon _{1,1}$, for $1\leq i \leq p$. Each $\tau _{i}^{y,x}$ induces a $u^{y,x}_{i}: \textsf{L} \rightarrow \textsf{K}$, such that for all $b \in \textsf{L}$,
$$u^{y,x}_{i} (b) = \phi_{x}(\chi _{y,x}(m_b m_{\xi^{i-1}}\varepsilon _{1,1})) = \phi_{x}(\chi _{y,x}(m_{b \xi^{i-1}}\varepsilon _{1,1})),$$
as $\varepsilon _{1,1}m_{b \xi^{i-1}}\varepsilon _{1,1} = \varepsilon _{1,1}m_{\R(b \xi^{i-1})}\varepsilon _{1,1}$, then
$$u^{y,x}_{i} (b) = \phi_{x}(\varepsilon _{1,1}m_{\R(b \xi^{i-1})}\varepsilon _{1,1}) = \R(b \xi^{i-1}).$$

Suppose that $x$ and $y$ satisfy (iv), we choose $\tau _{j}^{y,x} = \underline{\vartheta}^{j-1}$, with $1\leq j \leq \ell$. For any $b \in \textsf{L}$,
$$u^{y,x}_{j} (b) = \phi_{x}(\chi _{y,x}(m_b \underline{\vartheta}^{j-1})).$$

If $\ch \textsf{K}\neq p$, then $\underline{\vartheta} = \underline{\sigma}$, where $\sigma$ is a Galois automorphism. We have $\varepsilon _{1,1} \underline{\sigma}^{j-1} = \varepsilon _{1,1}$. Hence, 
$$\varepsilon _{1,1} m_b \underline{\sigma}^{j-1} = \varepsilon _{1,1} \underline{\sigma}^{j-1} m_{\sigma^{j-1}(b)} = \varepsilon _{1,1} m_{\sigma^{j-1}(b)}.$$

When $\ch \textsf{K}= p$, we have $\underline{\vartheta} = \underline{\delta}$, and $\varepsilon _{1,1} \underline{\delta} = 0$. Then,
$$\varepsilon _{1,1} m_b \underline{\delta}^{j-1} = \sum_{k=0}^{j-1} \binom{j-1}{k} \varepsilon _{1,1}\delta^{j-1-k} m_{\delta^k(b)} = \varepsilon _{1,1} m_{\delta^{j-1}(b)}.$$

Therefore,
$$u^{y,x}_{j} (b) = \phi_{x}(m_{\vartheta^{j-1}(b)}) = \vartheta^{j-1}(b).$$

Every $M \in \mathrm{mod}\, \mathcal{D}^e$ with a local basis, has a matrix representation $\underline{M}$, in which the stripe $\underline{M}_x$, for all $x\in \mathscr{P}$ has coefficients in \textsf{K}, if $x$ is a strong point, or in \textsf{L}, if $x$ is weak.

Denote by $\mathrm{rep}\, \mathscr{P}_0$  the full subcategory of $\Rep \mathscr{P}_0$ whose objects are of the form $(V,r; V_x:x\in \mathscr{P}_0)$, with $V$ a finite dimensional \textsf{L}-vector space. There is an equivalence between $\Rep \mathscr{P}_0$ and a full subcategory of $\mathrm{Mod}\, \mathcal{D}^e$, which induces an equivalence of $\mathrm{rep}\, \mathscr{P}_0$ with a full subcategory of $\mathrm{mod}\, \mathcal{D}^e$. Then, with the previous calculations, the next result determines a matrix problem which allow us to classify finite dimensional representations of a $p$-equipped poset $\mathscr{P}_0$.

\begin{corollary}
Let $M, N$ be objects in $\mathrm{mod}\, \mathcal{D}^e$, with local basis $\mathscr{B}(M)$ and $\mathscr{B}(N)$, respectively. Then $M\cong N$ if and only if there exist:
\begin{enumerate}[·]
\item a non-singular square matrix $L^{0}\in \mathscr{M}_{d_{0}}(\textsf{K})$;
\item for every $x \in \mathscr{P}$, if $x>0$ is a strong point, a non-singular square matrix $L^{x} \in \mathscr{M}_{d_{x}}(\textsf{K})$, and for each $0< y < x$, a matrix $L^{y,x}$ of size $d_{y}\times d_{x}$ with values in \textsf{K} (\textsf{L}) if $y$ is strong (weak);
\end{enumerate}
such that
\[
\underline{M}_{x}=L^{0} \left( \underline{N}_{x}L^{x} + \sum _{y\leq^py<x}\underline{N}_{y}L^{y,x} + \sum _{y\leq^1y<x} \R(\underline{N}_{y}L^{y,x}) \right);
\]
\begin{enumerate}[·]
\item for every weak point $x \in \mathscr{P}$, a non-singular square matrix $L^{x} \in \mathscr{M}_{d_{x}}(\textsf{L})$, for each $0< y\leq^p y < x$, a matrix $L^{y,x}$ of size $d_{y}\times d_{x}$ with values in \textsf{K},  and for $0< y\leq^1 y <^\ell x$, some matrices $L_{1}^{y,x},...,L_{\ell}^{y,x}\in \mathscr{M}_{d_{y},d_{x}}(\textsf{L})$
\end{enumerate}
such that
\[
\underline{M}_{x}=L^{0} \left( \underline{N}_{x}L^{x} + \sum _{y\leq^py<x}\underline{N}_{y}L^{y,x} + \sum _{y\leq^1y<x} \sum _{i=1}^{\ell}\vartheta^i(\underline{N}_{y})L_i^{y,x} \right).
\]
\end{corollary}

\begin{proof}
It is enough to follow the theorem \ref{RepsMatricesGenPequipados} and the calculus of each $u^{y,x}_{i}$, when $\mathscr{P}_0$ is a $p$-equipped poset. 

In particular, for $x,y \in \mathscr{P}$, with $x$ strong, $y$ weak, and $y<x$, Theorem \ref{RepsMatricesGenPequipados} says that there are $p$ matrices $L_{i}^{y,x} \in \mathscr{M}_{d_{y},d_{x}}(\textsf{K})$, for $i\in \{1,2,\ldots, p\}$ which we multiply by $u_{i}^{y,x}(\underline{N}_{y}) = \R(\underline{N}_{y}\xi^{i-1})$.

We have $\sum _{i=1}^{p} \R(\underline{N}_{y}\xi^{i-1})L_{i}^{y,x}$, which is equivalent to have $\R(\underline{N}_{y}L^{y,x})$ for a matrix $L^{y,x}$ over \textsf{L}, and finishes the proof.
\end{proof}


\appendix
\section{Appendix}

Here we follow \cite{BSZ}, but instead of consider left-modules, we consider right-modules.

\subsection{Ditalgebras}

Let $\textsf{K}$ be a field, $R$ be a $\textsf{K}$-algebra and $W$ be an $R$-bimodule graded over $\mathbb{Z}$ such that $W_{i}=0$ for $i>1$, $i<0$. That is $W=W_{0}\oplus W_{1}$. 

Consider the tensor algebra $T_{R}(W)$ with the graduation induced by the graduation of $W$. For every $r\in R\subset T_{R}(W)$, we have $\mathrm{gr}(r)=0$, and, for  $w_{1}\otimes \cdots \otimes w_{l}\in T_{R}(W)$, with $w_{i}\in W$ an homogeneous element, $i\in \{1,2,\ldots , l\}$, then $\mathrm{gr}(w)=\sum _{i=1}^{l}\mathrm{gr}(w_{i})$. 

For $n\in \mathbb{N}$, we denote by $T_{R}(W)_{n}$ the $\textsf{K}$-vector subspace of $T_{R}(W)$ generated by products of homogeneous elements of $W$ of degree $n$. We have $T_{R}(W)=\oplus _{i=0}^{\infty }T_{R}(W)_{n}$, and $T_{R}(W)_{n}T_{R}(W)_{m}\subset T_{R}(W)_{m+n}$, then $T_{R}(W)$ is a graded $\textsf{K}$-algebra. Besides $T_{R}(W)_{0}$ is a $\textsf{K}$-subalgebra and $T_{R}(W)_{0}=T_{R}(W_{0})$. Each $T_{R}(W)_{n}$ is a $T_{R}(W)_{0}$-bimodule.

\begin{definition}\label{defDiferencial} 
A differential in $T_{R}(W)$ is a $\textsf{K}$-linear transformation
$$d:T_{R}(W)\rightarrow T_{R}(W)$$ 
such that
\begin{enumerate}[(1)]
 \item $d(T_{R}(W)_{m})\subset T_{R}(W)_{m+1}$.
\item $d(R)=0$.
\item If $a$ and $b$ are homogeneous elements of $T_{R}(W)$ then
$$d(ab)=d(a)b+(-1)^{\mathrm{gr}(a)}ad(b).$$
\item $d^{2}=0$.
\end{enumerate} 
\end{definition}

\begin{proposition}\label{extiendeADiferencial} 
Let $d:W\rightarrow T_{R}(W)$ be an $R$-bimodule morphism such that $d(W_{0})\subset T_{R}(W)_{1}$ y $d(W_{1})\subset T_{R}(W)_{2}$, then $d$ can be extended to an $R$-bimodule morphism $\hat{d}:T_{R}(W)\rightarrow T_{R}(W)$, satisfying the conditions $(1)$,$(2)$ and $(3)$ of the previous definition. Moreover, $\hat{d}$ satisfies $(4)$ if
$\hat{d}^{2}(W)=0$.
\end{proposition}

\begin{proof}
We define $\hat{d}(r)=0$ for $r\in R\subset T_{R}(W)$, and for $w_{1},...,w_{n}$ homogeneous elements of $W$,
$$\hat{d}(w_{1}\otimes w_{2}\otimes \cdots \otimes w_{n})= d(w_1)\otimes w_2\otimes \cdots \otimes w_{n}$$
$$+ \sum _{i=1}^{n}(-1)^{\mathrm{gr}(w_{1}\cdots w_{i-1})}w_{1}\otimes \cdots \otimes w_{i-1}\otimes d(w_{i})\otimes w_{i+1}\otimes \cdots \otimes w_{n}.$$

We have $R$-bimodule transformations 
$$\hat{d}_{(c_{1}, \ldots, c_{n})}:W_{c_{1}}\otimes \cdots \otimes W_{c_{n}}\rightarrow T_{R}(W)$$ 
with $c_{i}\in \{0,1\}$.

Every $W^{\otimes n}$ is a direct sum of $R$-modules of the form $W_{\lambda_{1}}\otimes W_{\lambda_{2}}\otimes \cdots \otimes W_{\lambda_{n}}$ with $\lambda_{i}\in \{0,1\}$, then we have an $R$-bimodule morphism $\hat{d}_{n}:W^{\otimes n}\rightarrow T_{R}(W)$, and therefore an $R$-bimodule morphism $\hat{d}:T_{R}(W)\rightarrow T_{R}(W)$.

By its construction, $\hat{d}$ satisfies $(1),(2)$ and $(3)$ of Definition \ref{defDiferencial}. 

Notice that $\hat{d}^{2}$ has the following property: if $a$ and $b$ are homogeneous elements of $T_{R}(W)$ then:
$$\hat{d}^{2}(ab)=\hat{d}^{2}(a)b+a\hat{d}^{2}(a)b.$$

When $\hat{d}^{2}(W)=0$, we prove using induction on $n$, and the previous equation, that $\hat{d}^{2}(w_{1}\cdots w_{n})=0$ for every homogeneous $w_{i} \in W$. 
\end{proof}

We describe next some examples of ditalgebras.

\subsubsection{Example 1}

Let $Q=(Q_0, Q_1, s, t)$ be a finite quiver with $n$ elements, $A=\textsf{K}Q$ be its path algebra, and $e_1, e_2, \ldots, e_n$ be the trivial paths. We put
\[R=\sum_{i=1}^{n} \textsf{K}e_i, \hspace{1cm} W_0=\sum_{\gamma \in Q_1} \textsf{K}\gamma , \hspace{1cm} W_1=0.\]

Then, the graded algebra $T_R(W)=T_R(W_0)$ is a ditalgebra with differential $d=0$.

\subsubsection{Example 2}

From a partially ordered set $(\mathcal{P}, \leq )$, we construct a bigraph whose vertex are the same as $\mathcal{P}$, plus an additional point $w$. For every $i\in \mathcal{P}$, there is a solid arrow $\alpha_i:i\rightarrow w$, and for all $i,j \in \mathcal{P}$, such that $i< j$, there is a dashed arrow $x_{i,j}:j \dashrightarrow i$. Consider the path algebra where $e_i$ is the trivial path in the point $i\in \mathcal{P}\cup \{w\}$. We define:
\[R=\sum_{i\in \mathcal{P}\cup \{w\}} \textsf{K}e_i, \hspace{1cm} W_0=\sum_{i\in \mathcal{P}} \textsf{K}\alpha_i , \hspace{1cm} W_1=\sum_{i,j\in \mathcal{P}} \textsf{K}x_{i,j},\]
for each solid arrow $\alpha_i$
\[d(\alpha_i)= -\sum_{i< j} \alpha_ix_{i,j},\]
and for every dashed arrow $x_{i,j}$
\[d(x_{i,j})=\sum_{i<r<j} x_{i,r}x_{r,j}.\]

We have a tensor $\textsf{K}$-algebra $T_R(W)$, with $W=W_0\oplus W_1$, besides $d(W_0)\subseteq T_R(W)_1$ and $d(W_1)\subseteq T_R(W)_2$, then $d$ is extended to a morphism $d :T_{R}(W)\rightarrow T_{R}(W)$ satisfying $(1),(2),(3)$ of Definition \ref{defDiferencial}.

Now we evaluate $d^2(W)$. For a solid arrow,
\begin{align*}
d^2(\alpha_i) &=  d \left( -\sum_{i< j} \alpha_ix_{i,j} \right)\\
 &= -\sum_{i< j} d(\alpha_i)x_{i,j} - \sum_{i< j} \alpha_id(x_{i,j})\\
 &= \sum_{i< j} \sum_{i< r} \alpha_ix_{i,r} x_{i,j} - \sum_{i<r<j} \alpha_ix_{i,r}x_{r,j}\\
 &= \sum_{i<r<j} \alpha_ix_{i,r}x_{r,j} - \sum_{i<r<j} \alpha_ix_{i,r}x_{r,j}\\
 &= 0;
\end{align*}
and for a dashed arrow
\begin{align*}
d^2(x_{i,j}) &= d \left( \sum_{i<r<j} x_{i,r}x_{r,j} \right)\\
 &= \sum_{i<r<j} d(x_{i,r})x_{r,j} - \sum_{i<r<j} x_{i,r}d(x_{r,j})\\
 &= \sum_{i<r<j} \sum_{i<t<r} x_{i,t}x_{t,r}x{r,j} - \sum_{i<r<j} \sum_{r<t<j} x_{i,r}x_{r,t}x_{t,j}\\
 &= 0;
\end{align*}

Hence, $(T_{R}(W),d)$ is a ditalgebra.

\subsubsection{Example 3}

Let $R$ be a $\textsf{K}$-algebra, consider $W=\left( \begin{array}{cc}R&R\\0&R \end{array}\right)$, which is a $\textsf{K}$-algebra and an $R$-bimodule, with the operations:
$$r\left( \begin{array}{cc}x_{1}&x_2\\0&x_{3}\end{array} \right)= \left( \begin{array}{cc}rx_{1}&rx_{2}\\0&rx_{3}\end{array} \right) \ \ \text{ and }\ \  \left( \begin{array}{cc}x_{1}&x_2\\ 0&x_{3}\end{array} \right)r= \left( \begin{array}{cc}x_{1}r&x_{2}r\\0&x_{3}r\end{array} \right).$$

We define $w_{1}=\left( \begin{array}{cc}1&0\\0&0 \end{array} \right)$, $w_{2}=\left( \begin{array}{cc}0&0\\0&1 \end{array} \right)$ and $w_{1,2}=\left( \begin{array}{cc}0&1\\0&0 \end{array} \right)$.

Then $W=Rw_{1,2}R\oplus Rw_{1}R\oplus Rw_{2}R$. 

We put: $$W_{0}=Rw_{1,2}R, \quad  W_{1}=Rw_{1}R\oplus Rw_{2}R.$$

Then $T_{R}(W)$ is a graded $\textsf{K}$-algebra. 

The $R$-morphism $\text{\large\Fontauri\slshape d}:W\rightarrow T_{R}(W)$ given by
$$\text{\large\Fontauri\slshape d} (w_{1,2})=w_{1,2}\otimes w_{2} - w_{1}\otimes w_{1,2},$$ 
$$\text{\large\Fontauri\slshape d} (w_{1})=-w_{1}\otimes w_{1},\ \ \   \text{\large\Fontauri\slshape d} (w_{2})=-w_{2}\otimes w_{2},$$ 
can be extended to a morphism $\text{\large\Fontauri\slshape d} :T_{R}(W)\rightarrow T_{R}(W)$ satisfying $(1),(2),(3)$ of Definition \ref{defDiferencial}. 

We verify  
\begin{equation}\label{igual0}
\begin{split}
\text{\large\Fontauri\slshape d}\,^2 (w_{1})= & -\text{\large\Fontauri\slshape d}(w_{1})\otimes w_{1} + w_1 \otimes \text{\large\Fontauri\slshape d}(w_{1})=0\\
\text{\large\Fontauri\slshape d}\,^2 (w_{2})= & -\text{\large\Fontauri\slshape d}(w_{2})\otimes w_{2} + w_2 \otimes \text{\large\Fontauri\slshape d}(w_{2})=0\\
\text{\large\Fontauri\slshape d}\,^2 (w_{1,2})=\, & \text{\large\Fontauri\slshape d}(w_{1,2})\otimes w_{2} + w_{1,2} \otimes\, \text{\large\Fontauri\slshape d}(w_{2}) - \text{\large\Fontauri\slshape d}(w_{1})\otimes w_{1,2} + w_{1} \otimes \text{\large\Fontauri\slshape d}(w_{1,2})\\
=&0.
\end{split}
\end{equation}

Therefore $\text{\large\Fontauri\slshape d}$ is a differential.

\subsubsection{Example 4}\label{ditProj}

Let $\Lambda$ be a finite dimensional $\textsf{K}$-algebra, which admits a decomposition:
$$\Lambda =S\oplus J$$
where $S$ is a semi simple subalgebra of $\Lambda $ and $J$ is the radical of $\Lambda $.

Consider $^*J=\mathrm{Hom}_{S}(_SJ,S)$ and $\{ p_{i},\gamma _{i}\}_{i=1}^{n}$ a dual basis for $_SJ$, with $p_{i}\in J\,$ and $\,\gamma_{i}\in\,^*J$.

We define $\mu :\,^*J \rightarrow \,^*J\otimes_S \,^*J$, 
\[\mu(\gamma )= \sum_{i,j=1}^{n} \gamma_i \otimes \gamma_j \gamma(p_jp_i);\]
for all $\,\gamma \in\,^*J$. We say that $\mu$ is a \textit{comultiplication} because, given the multiplication
$$m:J\otimes _{S}J\rightarrow J:a\otimes b \mapsto ab,$$ 
the following diagram commutes 
\[\begin{diagram}
\node{^{*}J} \arrow{e,t}{\mu} \arrow{se,b}{^{*}m} \node{^{*}J\otimes _{S}\!\,\!^{*}J} \arrow{s,r}{\Phi}\\
\node{} \node{^{*}(J\otimes _{S}J)}
\end{diagram}\]
where $\Phi$ is the $S$-bimodule isomorphism defined as follows 
$$\Phi (\rho \otimes \nu )(a\otimes b)=\nu (a\rho (b)).$$ 

In fact,
\begin{align*}
\Phi (\mu(\gamma ))(a\otimes b) &= \Phi \left( \sum_{i,j=1}^{n} \gamma_i \otimes \gamma_j  \right)\gamma(p_jp_i)(a\otimes b)\\
 &= \Phi \left( \sum_{i,j=1}^{n} \gamma_i \otimes \gamma_j  \right)(a\otimes b)\gamma(p_jp_i)\\
 &= \sum_{i,j=1}^{n} \gamma_j(a \gamma_i(b)) \gamma(p_jp_i) \\
 &= \sum_{i,j=1}^{n} \gamma( \gamma_j(a \gamma_i(b))p_jp_i) \\
 &= \sum_{i=1}^{n} \gamma(a \gamma_i(b)p_i)\\
 &= \gamma(ab).
\end{align*}

We have the graded algebra $T_{S}(^{*}J)$. By Proposition \ref{extiendeADiferencial}, the morphism $\mu$ is extended to a $S$-bimodule morphism $\delta :T_{S}(^{*}J)\rightarrow T_{S}(^{*}J)$ satisfying $(1),(2),(3)$ of Definition \ref{defDiferencial}. 

For every $a=\sum _{i}a_{i}^{1}\otimes a_{i}^{2}\in (^{*}J)^{2}$, 
\begin{equation}\label{deltaCuadrado}
\begin{split}
 \delta (a) &= \sum _{i}\delta (a_{i}^{1})\otimes a_{i}^{2}-a_{i}^{1}\otimes \delta (a_{i}^{2})\\
 &=\sum _{i}\mu (a_{i}^{1})\otimes a_{i}^{2}- a_{i}^1\otimes \mu (a_{i}^{2})=(\mu \otimes id) (a) - (id \otimes \mu )(a).
\end{split}
\end{equation}

For $\gamma \in \,^{*}J$, we have
\begin{equation}
\delta ^{2}(\gamma)=\delta (\delta (\gamma))=\delta (\mu (\gamma))=(\mu \otimes id)\mu (\gamma)-(id \otimes \mu )\mu (\gamma).
\end{equation}

Notice that
\begin{align*}
(id \otimes \mu )\mu (\gamma) &= \sum_{i,j=1}^{n} \gamma_i \otimes \mu(\gamma_j) \gamma(p_jp_i)\\
 &= \sum_{i,j,r,t=1}^{n} \gamma_i \otimes \gamma_r \otimes \gamma_t \gamma_j(p_tp_r) \gamma(p_jp_i)\\
 &= \sum_{i,j,r,t=1}^{n} \gamma_i \otimes \gamma_r \otimes \gamma_t \gamma(\gamma_j(p_tp_r) p_jp_i)\\
 &= \sum_{i,r,t=1}^{n} \gamma_i \otimes \gamma_r \otimes \gamma_t \gamma(p_tp_rp_i).
\end{align*}

Besides, 
\begin{align*}
(\mu \otimes id)\mu (\gamma) &= \sum_{i,j=1}^{n} \mu(\gamma_i) \otimes \gamma_j \gamma(p_jp_i)\\
 &= \sum_{i,j=1}^{n} \sum_{r,t=1}^{n} \gamma_r \otimes \gamma_t \gamma_i(p_tp_r) \otimes \gamma_j \gamma(p_jp_i)\\
 &= \sum_{i,j,r,t=1}^{n} \gamma_r \otimes \gamma_t \otimes \gamma_i(p_tp_r) \gamma_j \gamma(p_jp_i);
\end{align*}
and, for $w\in J$
\begin{align*}
\sum_{i,j=1}^{n} \gamma_i(p_tp_r) \gamma_j \gamma(p_jp_i)(w) &= \sum_{i,j=1}^{n} \gamma_j (w \gamma_i(p_tp_r)) \gamma(p_jp_i)\\
 &= \sum_{i,j=1}^{n} \gamma(\gamma_j (w \gamma_i(p_tp_r)) p_jp_i)\\
 &= \sum_{i=1}^{n} \gamma(w \gamma_i(p_tp_r) p_i)\\
 &= \gamma(wp_tp_r)\\
 &= \sum_{l=1}^{n} \gamma(\gamma_l(w)p_lp_tp_r)\\
 &= \sum_{l=1}^{n} \gamma_l(w) \gamma(p_lp_tp_r);
\end{align*}
hence,
\[\sum_{i,j=1}^{n} \gamma_i(p_tp_r) \gamma_j \gamma(p_jp_i) = \sum_{l=1}^{n} \gamma_l \gamma(p_lp_tp_r).\]

Therefore, 
\begin{equation}\label{igualMU}
(id \otimes \mu )\mu (\gamma)=(\mu \otimes id)\mu (\gamma), 
\end{equation}
then $\delta ^{2}(\gamma )=0$. By Proposition \ref{extiendeADiferencial}, $\delta$ satisfies $(4)$, and $\mathcal{A}=(T_{S}(^{*}J),\delta )$ is a ditalgebra.

\subsubsection{Example 5. Drozd's ditalgebra}\label{ditDrozd}
  
Let $\Lambda$ be a $\textsf{K}$-algebra as in example $4$. We define
$$R=\left( \begin{array}{cc}S&0\\0&S \end{array} \right), W_{0}=\left( \begin{array}{cc}0&^{*}J\\0&0 \end{array} \right), W_{1}=\left( \begin{array}{cc}^{*}J&0\\0&^{*}J \end{array} \right).$$

As $W_{0}$ and $W_{1}$ are $R$-bimodules, we put $W=W_{0}\oplus W_{1}$. We obtain a graded $\textsf{K}$-algebra $T_{R}(W)$. 

Consider a $\textsf{K}$-bimodule $Z=\left( \begin{array}{cc}\textsf{K}&\textsf{K}\\0&\textsf{K} \end{array} \right)$, as in Example $3$, the elements
\begin{equation}\label{w}
w_{1}=\left( \begin{array}{cc}1&0\\0&0\end{array} \right), w_{2}=\left( \begin{array}{cc}0&0\\0&1 \end{array} \right), w_{1,2}=\left( \begin{array}{cc}0&1\\0&0 \end{array} \right),
\end{equation}
and the $\textsf{K}$-algebra $T_{\textsf{K}}(Z)$ with differential {\large\Fontauri\slshape d}, such that
$$\text{\large\Fontauri\slshape d} (w_{1,2})=w_{1,2}\otimes w_{2} - w_{1}\otimes w_{1,2},\ \ \    \text{\large\Fontauri\slshape d} (w_{1})=-w_{1}\otimes w_{1},\ \ \   
\text{\large\Fontauri\slshape d} (w_{2})=-w_{2}\otimes w_{2}.$$

We multiply $^{*}J$ by $Z$ using a $S$-$\textsf{K}$-bimodule morphism $\phi _{1}$ and a $\textsf{K}$-$S$-bimodule morphism $\phi_2$ 
$$\phi _{1}:\,^{*}J\otimes _{\textsf{K}}Z\rightarrow W : \gamma \otimes \left( \begin{array}{cc}x_{1}&x_2\\0&x_3 \end{array} \right) \mapsto 
\left( \begin{array}{cc}\gamma x_{1}&\gamma x_{2}\\0&\gamma x_{3} \end{array} \right);$$
$$\phi _{2}:Z\otimes _{\textsf{K}}\! \,^{*}J\rightarrow W : \left( \begin{array}{cc}x_{1}&x_2\\0&x_{3} \end{array} \right)\otimes \gamma \mapsto 
\left( \begin{array}{cc} x_{1}\gamma &x_{2}\gamma\\ 0& x_{3}\gamma  \end{array} \right).$$

We denote $\gamma u = \phi _{1}(\gamma\otimes u)$, and $u\gamma = \phi _{2}(u\otimes \gamma)$, for $\gamma \in \,^{*}J$, $z\in Z$. We have, for all $s\in S$
$$(\gamma s)z=(\gamma z)s,\ \ \ \ \  z(s\gamma)=s(z\gamma).$$

For every $n\in \mathbb{N}$, there is a $\textsf{K}$-$S$-bimodule morphism
$$*: Z^{\otimes n} \otimes _{\textsf{K}} (^{*}J)^{\otimes n}\rightarrow W^{\otimes n},$$
such that for $\underline{z}=(z_{1}\otimes z_{2}\otimes \cdots \otimes z_{n})\in Z^{\otimes n}, 
\underline{a}=(a_{1}\otimes a_{2}\otimes \cdots \otimes a_{n})\in (^{*}J)^{\otimes n}$,
$$*(\underline{z} \otimes \underline{a}) = \underline{z} * \underline{a}
= (z_1a_1 \otimes z_{2}a_2 \otimes \cdots \otimes z_{n}a_n).$$

Any element $w\in W$ can be written univocally as:
$$w=w_1a_1 + w_{1,2} a_{1,2} + w_2a_2, \text{ con } a_1, a_{1,2}, a_2 \in \,^{*}J,$$
then, we define
\begin{align*}
d(w) &= d(w_1a_1) + d(w_{1,2} a_{1,2}) + d(w_2a_2)\\ 
 &= \text{\large\Fontauri\slshape d} (w_{1})*\mu(a_1) + \text{\large\Fontauri\slshape d} (w_{1,2})*\mu(a_{1,2}) + \text{\large\Fontauri\slshape d} (w_{2})*\mu(a_2).
\end{align*}

Clearly $d :W\rightarrow T_{R}(W)$ is an $R$-bimodule morphism which can be extended to $d : T_{R}(W) \rightarrow T_{R}(W)$, satisfying $(1),(2),(3)$ of Definition \ref{defDiferencial}. 

Consider $w=w_1a_1 + w_{1,2} a_{1,2} + w_2a_2\in W$, to show that $d$ is a differential, we calculate
\begin{align*}
d^2(w) &= d(\text{\large\Fontauri\slshape d} (w_{1})*\mu(a_1) + \text{\large\Fontauri\slshape d} (w_{1,2})*\mu(a_{1,2}) + \text{\large\Fontauri\slshape d} (w_{2})*\mu(a_2))\\
 &= d \left( \sum_{i,j=1}^{n} w_1\gamma_i \otimes w_1\gamma_j a_1(p_jp_i) + \sum_{i,j=1}^{n} w_2\gamma_i \otimes w_2\gamma_j a_2(p_jp_i) \right)\\
 &\ \ \ \ +d \left( \sum_{i,j=1}^{n} w_1\gamma_i \otimes w_{1,2}\gamma_j a_{1,2}(p_jp_i) - \sum_{i,j=1}^{n} w_{1,2}\gamma_i \otimes w_2\gamma_j a_{1,2}(p_jp_i) \right)\\
 &=\ \  \sum_{i,j=1}^{n} \text{\large\Fontauri\slshape d} (w_{1})\mu(\gamma_i) \otimes w_1\gamma_j a_1(p_jp_i) - \sum_{i,j=1}^{n} w_1\gamma_i \otimes \text{\large\Fontauri\slshape d}(w_1)\mu(\gamma_j) a_1(p_jp_i)\\
 &\ \ + \sum_{i,j=1}^{n} \text{\large\Fontauri\slshape d} (w_{2})\mu(\gamma_i) \otimes w_2\gamma_j a_2(p_jp_i) - \sum_{i,j=1}^{n} w_2\gamma_i \otimes \text{\large\Fontauri\slshape d}(w_2)\mu(\gamma_j) a_2(p_jp_i)\\
 &\ \ + \sum_{i,j=1}^{n} \text{\large\Fontauri\slshape d} (w_1)\mu(\gamma_i) \otimes w_{1,2}\gamma_j a_{1,2}(p_jp_i) - \sum_{i,j=1}^{n} w_1\gamma_i \otimes \text{\large\Fontauri\slshape d}(w_{1,2})\mu(\gamma_j) a_{1,2}(p_jp_i)\\
 &\ \ - \sum_{i,j=1}^{n} \text{\large\Fontauri\slshape d} (w_{1,2})\mu(\gamma_i) \otimes w_2\gamma_j a_{1,2}(p_jp_i) - \sum_{i,j=1}^{n} w_{1,2}\gamma_i \otimes \text{\large\Fontauri\slshape d}(w_2)\mu(\gamma_j) a_{1,2}(p_jp_i)\\
 &=\ \  (\text{\large\Fontauri\slshape d}(w_{1})\otimes w_{1})*(\mu \otimes id)\mu(a_1)-(w_1 \otimes \text{\large\Fontauri\slshape d}(w_{1}))*(id \otimes \mu )\mu(a_1)\\
 &\ \ + (\text{\large\Fontauri\slshape d}(w_{2})\otimes w_{2})*(\mu \otimes id)\mu(a_2)-(w_2 \otimes \text{\large\Fontauri\slshape d}(w_{2}))*(id \otimes \mu )\mu(a_2)\\
 &\ \ + (\text{\large\Fontauri\slshape d}(w_{1})\otimes w_{1,2})*(\mu \otimes id)\mu(a_{1,2}) - (w_1 \otimes \text{\large\Fontauri\slshape d}(w_{1,2}))*(id \otimes \mu )\mu(a_{1,2})\\
 &\ \ - (\text{\large\Fontauri\slshape d}(w_{1,2})\otimes w_{2})*(\mu \otimes id)\mu(a_{1,2}) - (w_{1,2} \otimes \text{\large\Fontauri\slshape d}(w_{2}))*(id \otimes \mu )\mu(a_{1,2}),\\
\intertext{by using (\ref{igualMU}),}
 &= -(-\text{\large\Fontauri\slshape d}(w_{1})\otimes w_{1} + w_1 \otimes \text{\large\Fontauri\slshape d}(w_{1}))*(\mu \otimes id)\mu(a_1)\\
 &\ \ \  - (-\text{\large\Fontauri\slshape d}(w_{2})\otimes w_{2} + w_2 \otimes \text{\large\Fontauri\slshape d}(w_{2}))*(\mu \otimes id)\mu(a_2)\\
 &\ \ \  - (\text{\large\Fontauri\slshape d}(w_{1,2})\otimes w_{2} + w_{1,2} \otimes \text{\large\Fontauri\slshape d}(w_{2}) - \text{\large\Fontauri\slshape d}(w_{1})\otimes w_{1,2} + w_{1} \otimes \text{\large\Fontauri\slshape d}(w_{1,2}))(\mu \otimes id)\mu(a_{1,2})\\
 &= 0,\\
\intertext{due to the equations (\ref{igual0}).}
\end{align*}

Therefore, we have the ditalgebra
\[\mathcal{D}=(T_{R}(W),d).\]


\subsection{Representations of ditalgebras and Equivalences}

In this section, given a ditalgebra $\mathcal{A}=(T_{R}(W),d)$, we will deal with the category  $\mathrm{Mod}\, \mathcal{A}$, of representations of $\mathcal{A}$, whose objects are the right $T_R(W)_0$-modules. 

A morphism $f:M\rightarrow N$ in $\mathrm{Mod}\, \mathcal{A}$, is a pair $f=(f^{0},f^{1})$, where $f^{0}:M\rightarrow N$ is a right $R$-module morphism, and $f^{1}:T_{R}(W)_{1}\rightarrow \mathrm{Hom}_{k}(M,N)$ is a $T_{R}(W)_{0}$-bimodule morphism, such that
\begin{equation}\label{morfismosDit}
f^{0}(m)a=f^{0}(ma)+f^{1}(\delta (a))(m),
\end{equation}
for all $a\in T_{R}(W)_{0}$, $m\in M$.

Recall that for $v\in T_{R}(W)_{1}$, we have $d(v)=\sum_{i\in I} v_i^1\otimes v_i^2$, for a set $I$, and some $v_i^1, v_i^2\in T_{R}(W)_{1}$.

We have $d(v)\in T_{R}(W)_{1}\otimes_{T_{R}(W)_0} T_{R}(W)_{1}$, if $f:M\rightarrow N$ and $g:N\rightarrow L$ are morphisms in $\mathrm{Mod}\, \mathcal{A}$, consider the following $T_{R}(W)_{0}$-bimodule morphisms 
\[T_{R}(W)_{1}\otimes T_{R}(W)_{1}\xrightarrow{\ f^1\otimes g^1} \mathrm{Hom}_{k}(M,N)\otimes \mathrm{Hom}_{k}(N,L) \longrightarrow \mathrm{Hom}_{k}(M,L);\]
where, for all $v,w\in T_{R}(W)_{1}$
\[v\otimes w \mapsto f^1(v)\otimes g^1(w) \mapsto g^1(w)f^1(v).\]

We define $gf=((gf)^{0},(gf)^{1})$ as follows
\begin{equation}\label{composicion}
\begin{split}
(gf)^{0} &= g^0f^0;\\
(gf)^{1}(v) = g^1(v)f^0 +&g^0f^1(v)+\sum g^1(v_i^2)f^1(v_i^1).
\end{split}
\end{equation}

The pair $id=(id_M,0)$ is the identity for this composition.

We are going to see, next, some equivalences between the categories of representations of the ditalgebras of Examples \ref{ditProj} and \ref{ditDrozd}, and some categories determined by modules over an algebra $\Lambda =S\oplus J$, where $S$ is a semisimple sub algebra of $\Lambda $ and $J$ is the radical of $\Lambda $.

\begin{proposition}\label{projEquivDitalg}
Let $\Lambda$ be a $\textsf{K}$-algebra such that $\Lambda =S\oplus J,$ and $\mathcal{A}=(T_{S}(^{*}J),\delta )$ be a ditalgebra as in Example \ref{ditProj}. Then $\mathrm{Mod}\, \mathcal{A}$ is equivalent to the category of right projectve $\Lambda$-modules $\mathrm{Proj}\, \Lambda$.
\end{proposition}

\begin{proof}
Given a dual basis $\{ p_{i},\gamma _{i}\}_{i=1}^{n}$, for $_SJ$, with $p_{i}\in J\,$ y $\,\gamma_{i}\in\,^*J$, for two objects $M,N\in \mathrm{Mod}\, \mathcal{A}$, and a morphism $f:M\rightarrow N$ between them, we define
\[F:\mathrm{Mod}\, \mathcal{A} \rightarrow \mathrm{Proj}\, \Lambda;\]
as follows
\begin{equation}\label{}
\begin{split}
F(M) &= M \otimes_S \Lambda;\\
F(N) &= N \otimes_S \Lambda;\\
F(f)(m \otimes \lambda) &= f^0(m)\otimes \lambda+ \sum_{i=1}^{n}f^1(\gamma _{i})(m) \otimes  p_i \lambda, \text{ with } \lambda \in \Lambda \text{ y } m\in M.
\end{split}
\end{equation}

Consider $l\in \Lambda$,
\begin{align*}
F(f)((m \otimes \lambda)l) &= F(f)(m \otimes \lambda l)\\ 
&= f^0(m)\otimes \lambda l + \sum_{i=1}^{n}f^1(\gamma _{i})(m) \otimes  p_i \lambda l\\ 
&= \left[ f^0(m)\otimes \lambda + \sum_{i=1}^{n}f^1(\gamma _{i})(m) \otimes  p_i \lambda \right] l\\
&= F(f)(m \otimes \lambda) l;
\end{align*}
then $F(f)$ is a right $\Lambda$-module morphism.

For every $s\in S$, $i\in \{1,\ldots n\}$
\begin{align*}
(f^1\otimes id) (s\gamma_i \otimes  p_i) &= (f^1\otimes id) (\gamma_i \otimes  p_i s)\\
f^1(s\gamma_i) \otimes  p_i &= f^1(\gamma_i) \otimes  p_i s\\
f^1(s\gamma_i) \otimes  p_i &= (f^1(\gamma_i) \otimes  p_i) s.\\
\end{align*}

Therefore
\begin{align*}
F(f)(m \otimes s \lambda) &= f^0(m)\otimes s \lambda + \sum_{i=1}^{n}f^1(\gamma _{i})(m) \otimes  p_i s \lambda\\
 &= f^0(m)s \otimes \lambda + \sum_{i=1}^{n} f^1(s\gamma_i)(m) \otimes  p_i \lambda\\
 &= f^0(ms) \otimes \lambda + \sum_{i=1}^{n} f^1(\gamma_i)(ms) \otimes  p_i \lambda\\
 &= F(f)(ms \otimes \lambda);
\end{align*}
then, $F$ is balanced.

Let $M\xrightarrow{f}N\xrightarrow{g}L$ be a pair of morphisms in $\mathrm{Mod}\, \mathcal{A}$. Consider the composition $F(g)F(f)$, 
\begin{align*}
F(g)F(f)(m \otimes \lambda) &= F(g) \left( f^0(m)\otimes \lambda+ \sum_{i=1}^{n}f^1(\gamma _{i})(m) \otimes  p_i \lambda \right)\\
 &= g^0f^0(m) \otimes \lambda + \sum_{i=1}^{n} g^1(\gamma _{i})(f^0(m)) \otimes  p_i \lambda\\
 &\ \ \  + \sum_{i=1}^{n} g^0(f^1(\gamma _{i})(m)) \otimes  p_i \lambda+ \sum_{i,j=1}^{n} g^1(\gamma _{j}) f^1(\gamma _{i})(m) \otimes p_j p_i \lambda.
\end{align*}

Recall that $\delta (\gamma _{i})=\sum_{r,t=1}^{n} \gamma_r \otimes \gamma_t \gamma_i(p_tp_r)$, for all $i\in \{1,\ldots n\}$, then the image of the composition $gf$ is
\begin{align*}
F(gf)(m \otimes \lambda) &= (gf)^0(m) \otimes \lambda + \sum_{i=1}^{n} (gf)^1(\gamma _{i})(m) \otimes  p_i \lambda\\
 &= g^0f^0(m) \otimes \lambda + \sum_{i=1}^{n} g^1(\gamma _{i})(f^0(m)) \otimes  p_i \lambda + \sum_{i=1}^{n} g^0(f^1(\gamma _{i})(m)) \otimes  p_i \lambda\\
 &\ \ \  + \sum_{i,r,t=1}^{n} g^1(\gamma_t \gamma_i(p_tp_r)) f^1(\gamma _{r})(m) \otimes p_i \lambda\\
 &= g^0f^0(m) \otimes \lambda + \sum_{i=1}^{n} g^1(\gamma _{i})(f^0(m)) \otimes  p_i \lambda + \sum_{i=1}^{n} g^0(f^1(\gamma _{i})(m)) \otimes  p_i \lambda\\
 &\ \ \  + \sum_{i,r,t=1}^{n} g^1(\gamma_t) \gamma_i(p_tp_r) f^1(\gamma _{r})(m) \otimes p_i \lambda\\
 &= g^0f^0(m) \otimes \lambda + \sum_{i=1}^{n} g^1(\gamma _{i})(f^0(m)) \otimes  p_i \lambda + \sum_{i=1}^{n} g^0(f^1(\gamma _{i})(m)) \otimes  p_i \lambda\\
 &\ \ \  + \sum_{i,r,t=1}^{n} g^1(\gamma_t)f^1(\gamma _{r})(m) \gamma_i(p_tp_r) \otimes p_i \lambda\\
 &= g^0f^0(m) \otimes \lambda + \sum_{i=1}^{n} g^1(\gamma _{i})(f^0(m)) \otimes  p_i \lambda + \sum_{i=1}^{n} g^0(f^1(\gamma _{i})(m)) \otimes  p_i \lambda\\
 &\ \ \  + \sum_{i,r,t=1}^{n} g^1(\gamma_t)f^1(\gamma _{r})(m) \otimes \gamma_i(p_tp_r) p_i \lambda\\
 &= g^0f^0(m) \otimes \lambda + \sum_{i=1}^{n} g^1(\gamma _{i})(f^0(m)) \otimes  p_i \lambda + \sum_{i=1}^{n} g^0(f^1(\gamma _{i})(m)) \otimes  p_i \lambda\\
 &\ \ \  + \sum_{i,r,t=1}^{n} g^1(\gamma_t)f^1(\gamma _{r})(m) \otimes p_tp_r \lambda\\
 &= F(g)F(f)(m \otimes \lambda). 
\end{align*}

Hence, $F$ is a well-defined functor.

Any right projective $\Lambda$-module $M$, can be written as $M \otimes_S \Lambda$, then $F$ is dense.

Suppose that $f=(f^{0},f^{1}):M \rightarrow N$ is a morphism in $\mathrm{Mod}\, \mathcal{A}$ such that $F(f)=0$. Then, when $\lambda=1\in  \Lambda,$
\[f^0(m)\otimes 1 + \sum_{i=1}^{n}f^1(\gamma _{i})(m) \otimes  p_i = 0,\]
for all $m\in M$.

Notice that 
\begin{align*}
F(N) &= N\otimes_S \Lambda\\
 &= N  \otimes_S (S\oplus J)\\
 &= (N \otimes_S S) \oplus (N \otimes_S J);
\end{align*}
we have $f^0(m)\otimes 1 \in N \otimes_S S$ y $\sum_{i=1}^{n}f^1(\gamma _{i})(m) \otimes  p_i\in N \otimes_S J$, then
\[f^0(m)\otimes 1=0\   \text{  y }  \sum_{i=1}^{n}f^1(\gamma _{i})(m) \otimes  p_i=0, \text{ para toda } m\in M,\]
therefore $f^0=0$.

For $\gamma \in \,^{*}J$, consider 
$$M \otimes_S J \xrightarrow{ id_M \otimes \gamma } M \otimes_S S \xrightarrow{\ \cong\ } M;$$
such that, for all  $u\in J, m\in M$
$$m\otimes u \mapsto m \otimes \gamma(u) \mapsto m\gamma(u).$$

Hence,
\begin{align*}
id_M \otimes \gamma \left( \sum_{i=1}^{n} f^1(\gamma _{i})(m) \otimes  p_i  \right) &= 0\\
\sum_{i=1}^{n} f^1(\gamma _{i})(m) \otimes \gamma (p_i)  &= 0;\\
\intertext{then,}
\sum_{i=1}^{n} f^1(\gamma _{i})(m) \gamma (p_i) &= 0;\\
\sum_{i=1}^{n} f^1(\gamma _{i}\gamma (p_i))(m)  &= 0;\\
\intertext{by the properties of a dual basis $\gamma=\sum_{i=1}^{n} \gamma_i \gamma(p_i)$, therefore}
f^1(\gamma)(m) &= 0,
\end{align*}
so $f^1=0$, and we conclude that $F$ is faithful.

Consider the projection $\pi_0:\Lambda \rightarrow S$, and the morphisms
\begin{equation}\label{morfismoENtodoD}
\gamma : \Lambda \rightarrow S\text{, obtained by defining } \gamma (s)=0, \text{ for all } \gamma \in \,^{*}J, s\in S.
\end{equation}

The set $\{ 1, p_{i}, \pi_0, \gamma _{i}\}_{i=1}^{n}$ is a dual basis for $_S\Lambda$. Indeed, for every $\lambda \in \Lambda$,
\begin{align*}
\lambda &= \pi_0(\lambda) + (\lambda - \pi_0(\lambda));
\intertext{we have $\lambda - \pi_0(\lambda)\in J$,}
\lambda &= \pi_0(\lambda) + \sum_{i=1}^{n} \gamma_i(\lambda - \pi_0(\lambda))p_i\\
 &= \pi_0(\lambda) + \sum_{i=1}^{n} (\gamma_i(\lambda) - \gamma_i(\pi_0(\lambda)))p_i\\
 &= \pi_0(\lambda) + \sum_{i=1}^{n} \gamma_i(\lambda) p_i;\\
\intertext{because $\pi_0(\lambda)\in S$.}
\end{align*}

Let $h: M\otimes_S  \Lambda \rightarrow N \otimes_S \Lambda$ be a $\Lambda$-module morphism, we define a $S$-module morphism, as follows
\[f^0_h(v)=m( id_N \otimes \pi_0)(h(v\otimes 1)),\]
where $v\in M$, and $m$ is an isomorphism $m:N\otimes_S S \rightarrow N: n\otimes s \mapsto ns.$

For every $\gamma \in \,^{*}J$, taking into account  (\ref{morfismoENtodoD}), we define 
\[f^1_h(\gamma)(v)=m(id_N \otimes \gamma)(h(v\otimes 1)).\]

Suppose that $h(v\otimes 1)= \sum_{r=1}^{t} n_r \otimes \lambda_r$, for some $n_r \in N, \lambda_r\in \Lambda$, then, for $s\in S$,
\begin{align*}
f^1_h(s\gamma)(v) &= m(id_N \otimes s\gamma)(h(v\otimes 1))\\
 &=m \left( \sum_{r=1}^{t} n_r \otimes s\gamma(\lambda_r) \right)\\
 &= m \left( \sum_{r=1}^{t} n_r \otimes \gamma(\lambda_r s) \right)\\
 &= \sum_{r=1}^{t} n_r \gamma(\lambda_r s)\\
 &= m(id_N \otimes s\gamma)(h(v\otimes 1)s)\\
 &= m(id_N \otimes s\gamma)(h(v\otimes s))\\
 &= m(id_N \otimes s\gamma)(h(vs\otimes 1))\\
 &= f^1_h(\gamma)(vs).
\end{align*}

Besides
\begin{align*}
f^1_h(\gamma s)(v) &= m(id_N \otimes \gamma s)(h(v\otimes 1))\\
 &=m \left( \sum_{r=1}^{t} n_r \otimes \gamma s(\lambda_r) \right)\\
 &= m \left( \sum_{r=1}^{t} n_r \otimes \gamma(\lambda_r )s \right)\\
 &= \sum_{r=1}^{t} n_r \gamma(\lambda_r )s\\
 &= m(id_N \otimes s\gamma)(h(v\otimes 1))s\\
 &= f^1_h(\gamma)(v)s.
\end{align*}

Therefore $f^1_h: \,^{*}J \rightarrow \mathrm{Hom}_{k}(M,N)$ is a $S$-bimodule morphism, and $f_h=(f^0_h,f^1_h):M \rightarrow N$ is a morphism in $\mathrm{Mod}\, \mathcal{A}$, such that, for all $\lambda \in \Lambda, v \in M$
\begin{align*}
F(f_h)(v \otimes \lambda) &= f_h^0(v)\otimes \lambda+ \sum_{i=1}^{n}f^1_h(\gamma _{i})(v) \otimes  p_i \lambda \\
 &= m( id_N \otimes \pi_0)(h(v\otimes 1)) \otimes \lambda + \sum_{i=1}^{n} m(id_N \otimes \gamma_i)(h(v\otimes 1)) \otimes  p_i \lambda \\
 &= \left[ m( id_N \otimes \pi_0)(h(v\otimes 1)) \otimes 1 + \sum_{i=1}^{n} m(id_N \otimes \gamma_i)(h(v\otimes 1)) \otimes  p_i  \right] \lambda \\
\intertext{using the dual basis of $_S\Lambda$,}
 &= h(v\otimes 1)\lambda \\
 &= h(v\otimes \lambda).
\end{align*}

We conclude that $F$ is a full functor, which finishes the proof.
\end{proof}

Consider the ditalgebra $\mathcal{D}$, as in Example \ref{ditDrozd}
\[\mathcal{D}=(T_{R}(W),d).\]

Every $a\in T_{R}(W)_0$, can be written as $a=w_1s_1+w_2s_2+w_{1,2}\gamma$, for some $s_1, s_2 \in S$ and $\gamma \in \,^{*}J$.
 
A morphism $f:M \rightarrow N$ in the category $\mathrm{Mod}\, \mathcal{D}$, between two objects $M$ and $N$, satisfies for all $a\in T_{R}(W)_0$, $m \in M$
\begin{align*}
f^0(m)(w_1s_1+w_2s_2+w_{1,2}\gamma) &= f^0(m(w_1s_1+w_2s_2+w_{1,2}\gamma))\\
 &\ \  + f^1(d(w_1s_1+w_2s_2+w_{1,2}\gamma))(m)\\
f^0(m)w_1s_1+f^0(m)w_2s_2+f^0(m)w_{1,2}\gamma &= f^0(m)w_1s_1+f^0(m)w_2s_2 + f^0(mw_{1,2}\gamma)\\
 &\ \  + f^1(d(w_{1,2}\gamma))(m),
\end{align*}
then,
\begin{equation}\label{morfismoD}
\begin{split}
f^0(m)w_{1,2}\gamma &= f^0(mw_{1,2}\gamma) + f^1(\delta(w_{1,2})*\mu(\gamma))(m)\\
 &= f^0(mw_{1,2}\gamma) + \sum_{i,j=1}^{n} w_{1,2}\gamma_if^1(w_2\gamma_j\gamma(p_jp_i)) (m)\\
 &\ \   -  \sum_{i,j=1}^{n} f^1(w_1\gamma_i)w_{1,2}\gamma_j \gamma(p_jp_i)(m)\\
 &= f^0(mw_{1,2}\gamma) + \sum_{i,j=1}^{n}f^1(w_2\gamma_j)(m w_{1,2}\gamma_i)\gamma(p_jp_i)\\
 &\ \   -  \sum_{i,j=1}^{n} f^1(w_1\gamma_i)(m)w_{1,2}\gamma_j \gamma(p_jp_i).
\end{split}
\end{equation}

We are going to define, next, a category of morphisms which is equivalent to $\mathrm{Mod}\, \mathcal{D}$.

\begin{definition}\label{catMorfismos}
Let  $\mathcal{A}=(T_{S}(^{*}J),\delta )$ be the ditalgebra of Example  \ref{ditProj}, we denote by $\mathcal{M}^1_{\Lambda}\mathcal{(A)}$ a category whose objects are morphisms in $\mathrm{Mod}\, \mathcal{A}$, of the form 
$$\psi=(0,\psi):M_1 \rightarrow M_2.$$ 

The morphisms between them are pairs of morphisms $(f,g)$ in $\mathrm{Mod}\, \mathcal{A}$, such that the following diagram commutes
\[
\begin{CD}
M_1 @>{(0,\psi_M)}>>  M_2\\
 @V{f=(f^0,f^1)}VV @VV{g=(g^0,g^1)}V \\
N_1 @>>{(0,\psi_N)}>  N_2
\end{CD}
\]
\end{definition}

We have the following result:

\begin{proposition}\label{DitalgEquivDrozd}
From a  $\textsf{K}$-algebra $\Lambda =S\oplus J$, using the notation of this section, we construct the ditalgebras $\mathcal{A}=(T_{S}(^{*}J),\delta )$ and $\mathcal{D}=(T_{R}(W),d)$, and the category of morphisms $\mathcal{M}^1_{\Lambda}\mathcal{(A)}$.

Then the categories $\mathrm{Mod}\, \mathcal{D}$ and $\mathcal{M}^1_{\Lambda}\mathcal{(A)}$ are equivalent.
\end{proposition}

\begin{proof}
Let $1_S$ be the unit in $S$, notice that $w_11_S$ y $w_21_S$ are idempotents in $T_R(W)_0= \left( \begin{array}{cc}S&^{*}J\\0&S \end{array} \right)$.

For any object $M \in \mathrm{Mod}\, \mathcal{D}$, we denote
\begin{equation}\label{dosSmodulosD}
M_1=Mw_11_S; \hspace{1cm} M_2=Mw_21_S.
\end{equation}

By the right $T_R(W)_0$-module structure of $M$, we have that $M_1$ and $M_2$ are right $S$-modules, $M=M_1 \oplus M_2$, and for all $\gamma \in \,^{*}J$, the multiplication induces a $\textsf{K}$-transformation 
\[\psi_M(\gamma): M_1 \rightarrow M_2: m\mapsto m w_{1,2} \gamma.\]

For every $s\in S$, $m\in M_1$, 
\[\psi_M(s\gamma)(m)= m w_{1,2}s \gamma = m sw_{1,2} \gamma = \psi_M(\gamma)(ms),\]
besides,
\[\psi_M(\gamma s)(m) = m w_{1,2} \gamma s = \psi_M(\gamma)(m)s,\]
which proves that $\psi_M:\,^{*}J \rightarrow \mathrm{Hom}_{k}(M_1,M_2)$ is a $S$-bimodule morphism.

Let $M$ and $N$ be objects in $\mathrm{Mod}\, \mathcal{D}$, and $f=(f^0,f^1):M \rightarrow N$ be a morphism between them. We define
$$G: \mathrm{Mod}\, \mathcal{D} \rightarrow \mathcal{M}^1_{\Lambda}\mathcal{(A)}$$
as follows
$$G(M)=M_1 \xrightarrow{(0,\psi_M)} M_2,$$
$$G(N)=N_1 \xrightarrow{(0,\psi_N)} N_2,$$
$$G(f) = (f_1,f_2) = ((f_1^0,f_1^1),(f_2^0,f_2^1));$$
where $f_i^0=f^0|_{M_i}$, and $f^1_i(\gamma)=f^1(w_i\gamma)$, for every $\gamma \in\!\,^{*}J$, for $i=1,2$.

The morphism $f^0$ is a right $R$-module morphism, so 
\[f^0(M_i) = f^0(Mw_i1_S)=f^0(M)w_i1_S\subseteq N_i,\]
then $f_i^0: M_i \rightarrow N_i$ is a right $R$-module morphism.

Consider $\gamma \in \,^{*}J$, the product $w_i\gamma\in T_{R}(W)_1$, is such that
$$w_i\gamma=w_i1_S w_i \gamma w_i1_S.$$

We have that $f^1: T_{R}(W)_1 \rightarrow \mathrm{Hom}_{k}(M,N)$, is a $T_R(W)_0$-bimodule morphism 
$$f^1_i(\gamma)=f^1(w_i\gamma)= w_i1_S f^1(w_i\gamma) w_i1_S,$$
hence,
$$f^1_i(\gamma)\in \mathrm{Hom}_{k}(M_i,N_i).$$

Notice that, for $\gamma \in J^{*}, m\in M_1$, the composition $(f_2^0,f_2^1)(0,\psi_M)=(0,(f_2\psi_M)^1)$, is such that
\begin{align*}
(f_2\psi_M)^1(\gamma)(m) &=  f_2^0(\psi_M(\gamma)(m)) + \sum_{i,j=1}^{n} f_2^1(\gamma_j\gamma(p_jp_i))\psi_M(\gamma_i)(m)\\
 &=  f^0(m w_{1,2} \gamma) + \sum_{i,j=1}^{n} f^1(w_2 \gamma_j\gamma(p_jp_i))(m w_{1,2} \gamma_i)\\
 &=  f^0(m w_{1,2} \gamma) + \sum_{i,j=1}^{n} f^1(w_2 \gamma_j)(m w_{1,2} \gamma_i)\gamma(p_jp_i);
\end{align*}
and for $(0,\psi_N)(f_1^0,f_1^1)=(0,(\psi_Nf_1)^1)$, we have
\begin{align*}
(\psi_Nf_1)^1(\gamma)(m) &= \psi_N(\gamma)(f^0_1(m)) + \sum_{i,j=1}^{n} \psi_N(\gamma_j\gamma(p_jp_i)) f^1_1(\gamma_i)(m)\\
 &= f^0(m) w_{1,2} \gamma + \sum_{i,j=1}^{n} f^1(w_1\gamma_i)(m) w_{1,2} \gamma_j\gamma(p_jp_i).
\end{align*}

Substracting,
\begin{align*}
(f_2\psi_M)^1(\gamma)(m) &- (\psi_Nf_1)^1(\gamma)(m) = f^0(m w_{1,2} \gamma) - f^0(m) w_{1,2} \gamma\\
& + \sum_{i,j=1}^{n} (f^1(w_2 \gamma_j)(m w_{1,2} \gamma_i)-f^1(w_1\gamma_i)(m) w_{1,2} \gamma_j)\gamma(p_jp_i).
\end{align*}

Hence, $(f_2\psi_M)^1(\gamma)(m) - (\psi_Nf_1)^1(\gamma)(m) = 0$ if and only if (\ref{morfismoD}) holds. Therefore, $G(f)$ is a morphism in $\mathcal{M^{\mathrm{1}}(A)}$, if and only if $f$ is a morphism in $\mathrm{Mod}\, \mathcal{D}$. 

We conclude that $G$ is a full functor, because $h=((h_1^0,h^1_1),(h_2^0,h_2^1)):G(M) \rightarrow G(N)$ is a morphism in $\mathcal{M}^1_{\Lambda}\mathcal{(A)}$ if and only if $f=(w_11_Sh_1^0\oplus w_21_Sh_2^0, w_11_Sh_1^1 \oplus w_21_Sh_2^1)$ is a morphism in $\mathrm{Mod}\, \mathcal{D}$, and $G(f)=h$.

If a morphism $f:M \rightarrow N$ in $\mathrm{Mod}\, \mathcal{D}$, is such that $G(f)=0$, then $f^0(M_1)=f^0(M_2)=0$, hence $f^0=0$. We have that $x\in T_{R}(W)_1$ can be written as $x= w_1 a_1 + w_2 a_2$, for some $a_1, a_2 \in\!\,^{*}J$, then $f^1(x)=f^1_1(a_1)+f^1_2(a_2)=0$. Therefore $G$ is faithful.

Let $M_1 \xrightarrow{(0,\psi_M)} M_2$ be an object of $\mathcal{M}^1_{\Lambda}\mathcal{(A)}$. We give to $M_1\oplus M_2$ a right $T_R(W)_0$-module structure, with the product
\[(m_1\ \ m_2) \left( \begin{array}{cc}s_1&\gamma\\ 0 & s_2 \end{array} \right) = \left( m_1s_1\ \ \ \psi_M(\gamma)(m_1)+m_2s_2 \right).\]

Therefore, the functor $G$ is dense, and the proposition is proved.
\end{proof}

Consider the following category of morphisms.

\begin{definition}\label{catMorfismosProj}
We denote by $\mathcal{M}^1(\Lambda)$ the category whose objects are morphisms $\varphi: P \rightarrow Q$, between right projective $\Lambda$-modules $P$ and $Q$, such that $\mathrm{Im}\,\varphi \subseteq \mathrm{rad}Q$. Its morphisms are pairs of right projective $\Lambda$-module morphisms, in such a way that the following diagram commutes 
\[
\begin{CD}
P @>{\varphi}>>  Q\\
 @V{h}VV @VV{h'}V \\
P' @>>{\varphi'}>  Q'
\end{CD}
\]
for every pair of objects $\varphi, \varphi' \in \mathcal{M}^1(\Lambda)$.
\end{definition}

Due to the proof of Proposition \ref{projEquivDitalg}, every morphism $F(f):M \otimes_S \Lambda \rightarrow N \otimes_S \Lambda$ between projective $\Lambda$-modules can be written as
\[F(f)(m \otimes \lambda) = f^0(m) \otimes\lambda + \sum_{i=1}^{n}  f^1(\gamma _{i})(m)\otimes p_i  \lambda \]
for some morphism $f=(f^0,f^1)$ in $\mathrm{Mod}\, \mathcal{A}$.

We have that $\sum_{i=1}^{n}  f^1(\gamma _{i})(m) \otimes  p_i \lambda \in N\otimes_S J \subseteq \mathrm{rad}(N\otimes_S \Lambda)$, then $\mathrm{Im}\,F(f)\subseteq \mathrm{rad}N\otimes_S \Lambda$ if and only if $f^0=0$.

Therefore $\mathcal{M}^1(\Lambda)$ is equivalent to $\mathcal{M}^1_{\Lambda}\mathcal{(A)}$, and, in turn, to $\mathrm{Mod}\, \mathcal{D}$ by the previous proposition.



\end{document}